\documentclass[11pt]{amsart}

\usepackage{fullpage}
\usepackage{amsfonts}
\usepackage{bbding}
\usepackage[dvips]{graphicx}
\usepackage{amsfonts,amssymb,amsmath,amsthm,amscd}
\usepackage{mathrsfs}
\usepackage{xcolor}
\usepackage{empheq}
\usepackage{adforn}
\usepackage{cancel}

\usepackage{xr}
\usepackage{mdframed}

\usepackage{enumerate}
\usepackage{lmodern}
\usepackage{mdframed}
\usepackage{stmaryrd}
\usepackage{blindtext}
\usepackage[all, knot]{xy}
\usepackage{leftidx}
\usepackage{accents}

\definecolor{slightblue}{rgb}{.8, .8, 1}
\definecolor{hair}{RGB}{100,225,190}
\definecolor{ruby}{RGB}{220,50,120}
\definecolor{grass}{RGB}{150,220,110}

\usepackage[colorlinks=true, pdfstartview=FitP, linkcolor=hair!80!black, citecolor=hair!80!black, urlcolor=hair!80!black]{hyperref}

\usepackage[T1]{fontenc}

\newtheorem{theorem}{Theorem}[section] \newtheorem{proposition}[theorem]{Proposition}
\newtheorem{lemma}[theorem]{Lemma} \newtheorem{corollary}[theorem]{Corollary}

\theoremstyle{definition}

\theoremstyle{remark} \newtheorem{remark}[theorem]{Remark} \numberwithin{equation}{section}
\numberwithin{figure}{section}

\newcommand{\eps}{\varepsilon}

\newcommand{\R}{\mathbb{R}}

\renewcommand{\H}{\mathbb{H}}

\newcommand{\N}{\mathbb{N}}
\newcommand{\C}{\mathbb{C}}

\renewcommand{\P}{\mathbf{P}}
\newcommand{\E}{\mathbf{E}}

\newcommand{\G}{\mathcal{G}}

\renewcommand{\a}{{a}}
\renewcommand{\b}{{b}}
\renewcommand{\c}{{c}}

\newcommand{\capacity}{\mathrm{Cap}}
\newcommand{\dist}{\mathrm{dist}}

\newcommand{\dif}{d}

\usepackage{sidecap,mathtools,cases,caption,subcaption,empheq}
\usepackage[section]{placeins}
\usepackage{wrapfig}
\usepackage{lpic}

\theoremstyle{definition}
\newtheorem{construction}[theorem]{Construction}

\begin{document}

\title {Conformal restriction: The trichordal case}
\author{Wei Qian}
\address{Department of Mathematics\\
 ETHZ\\ 
 R\"amistrasse 101, 8092 Z\"urich, Switzerland}
\email[]{wei.qian@math.ethz.ch}
\date{}

\begin{abstract}
The study of conformal restriction properties in two-dimensions has been initiated by Lawler, Schramm and Werner in \cite{MR1992830} who focused on the 
natural and important chordal case: They characterized and constructed all random subsets of a given simply connected domain that join two marked boundary points and that satisfy the 
additional restriction property. The radial case (sets joining an inside point to a boundary point) has then been investigated by Wu in \cite{MR3293294}. In the present paper, we study the third natural 
instance of such restriction properties, namely the ``trichordal case'', where one looks at random sets that join three marked boundary points. 
This case involves somewhat more technicalities than the other two, as the construction of this family of random sets relies on
special variants of SLE$_{8/3}$ processes with a drift term in the driving function that involves hypergeometric functions. 
It turns out that such a random set can not be a simple curve simultaneously in the neighborhood of all three marked points, and that the exponent $\alpha = 20/27$ shows up in the 
description of the law of the skinniest possible symmetric random set with this trichordal restriction property.  
\end{abstract}
\maketitle
\tableofcontents

\section{Introduction}

\subsection {Background}
In the present paper we further study random two-dimensional sets that satisfy the conformal invariance property combined with the restriction property, following the work of Lawler, Schramm and Werner  in \cite{MR1992830}  and the paper of Wu  in \cite{MR3293294}.

Measures that satisfy {conformal restriction property} were introduced and first studied by Lawler, Schramm and Werner in \cite{MR1992830}:
For a simply connected domain $D\not=\mathbb{C}$ with two marked boundary points $a$ and $b$ (we will say ``boundary points'' instead of prime ends in the present introduction), they studied 
a class of random simply connected and relatively closed sets $K\subset \overline{D}$ such that $K$ intersects $\partial D$ only at $a$ and $b$.
Such a set (or rather, its distribution) is said to satisfy  \emph{chordal conformal restriction property} if the following two conditions hold:
\begin{itemize}
\item[(i)](conformal invariance) the law of $K$ is invariant under any conformal map from $D$ onto itself that leave the boundary points $a$ and $b$ invariant.
\item[(ii)](restriction) for any simply connected subset $D'$ of $D$ such that {$\dist(D \setminus D', \{ a, b \})>0$},
the conditional distribution of $K$ given $K \subset D'$ is equal to the image of the law of $K$ under $\phi$, where $\phi$ is any conformal map from $D$ onto $D'$ that 
leaves the points $a,b$  invariant (property (i) actually ensures that if this holds for one such map $\phi$, then it holds also for any other such map). See Figure \ref{fig:two-point-restriction}.
\end{itemize}
\begin{figure}[h]
\centering
\includegraphics[width=0.78\textwidth]{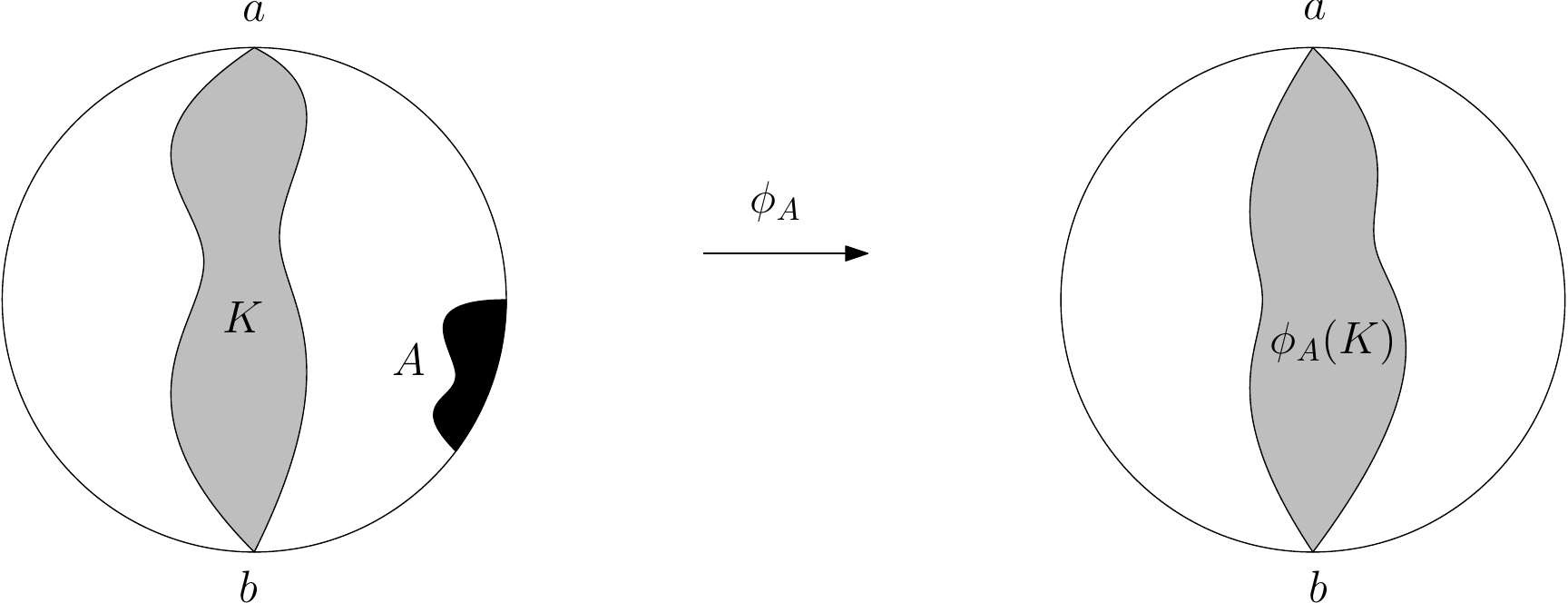}
\caption{
{Chordal restriction: $D'=D\setminus A$ and $\phi_A$ is a conformal map from $D\setminus A$ onto $D$ that leaves $a,b$ invariant.}
The conditional law of $\phi_A(K)$ 
given $\{K\cap A=\emptyset\}$ is equal to the (unconditional) law of $K$. }
\label{fig:two-point-restriction}
\end{figure}

It is straightforward to check that if  a random set $K$ satisfies this chordal conformal restriction property in one simply connected domain $D$ with boundary points $a$ and $b$, then if we map $D$ conformally to another simply connected domain $\Phi ( D)$ via some fixed deterministic map $\Phi$, then $\Phi (K)$ satisfies chordal conformal restriction in $\Phi (D)$ with boundary points $\Phi (a)$ and $\Phi (b)$ (and  property (i) ensures that the image law {depends} only on 
the triplet $(\Phi (D), \Phi (a), \Phi(b))$, and not on the particular instance of $\Phi$). Hence, it is sufficient to study this property in one particular given domain $D$, such as the unit disc or the upper half-plane.

Recall that conformal invariance is believed to hold  in the scaling limit for a large class of two-dimensional models from statistical physics. A chordal restriction
property can be interpreted as follows: On a lattice, one can associate to each set $K$ an energy, and the restriction property means that this energy of $K$ is ``intrinsic'' in 
the sense that it does not depend on the domain in which it lives but only on $K$ (and the extremal points $a$ and $b$) itself.
For example, for a simple random walk on a square lattice in a discretized domain $D$, which is conditioned
to go from one boundary point $a$ to another boundary point $b$, the probability of a given path $\gamma$ will be proportional to $4^{-|\gamma|}$ where $|\gamma|$ is the number of steps of $\gamma$. 
This weight $4^{-|\gamma|}$ is then intrinsic because it only depends on the path $\gamma$ itself but not on the domain $D$. As a consequence, if we condition such a random walk excursion from $a$ to $b$ in $D$  to stay in a subdomain $D'$ of $D$ which still has $a$ and $b$ on its boundary, then one gets exactly a random walk excursion from $a$ to $b$ in this smaller domain $D'$. In the continuous limit, the Brownian excursion from $a$ to $b$ in $D$ (or rather its ``filling'' in order to get a simply connected set) does indeed satisfy chordal conformal restriction. chordal restriction measures therefore describe the natural conformally invariant and intrinsic ways to join two boundary points in a simply connected domain. 

Lawler, Schramm and Werner proved in \cite{MR1992830} that such measures are fully characterized by one real parameter $\alpha$ (here and in the sequel, this means that there is an injection from the set of conformal restriction measures into $\R$) that can be described as follows. For each chordal conformal restriction measure, there exists a positive $\alpha$ such that for all $A\subset D$ such that
$D\setminus A$ is again a simply connected domain and $\dist(A,\{a,b\})>0$, one has
\begin{align}\label{two-point-formula}
\P(K\cap A=\emptyset) =\left(\phi_A'(a)\phi_A'(b)\right)^\alpha
\end{align}
(It is straightforward to see that the product $\phi_A'(a) \phi_A'(b)$ does not depend on the choice of $\phi_A$ among the one-dimensional family of conformal maps from $D \setminus A$ onto $D$ that leave $a$ and $b$ invariant). Conversely, it is easy to see that for each given $\alpha$, there exists at most one law of $K$ satisfying this relation for all $A$. 

The more challenging part is then to investigate for which values of $\alpha$ such a random set $K$ does indeed exist. Lawler, Schramm and Werner showed in the same paper \cite{MR1992830} that 
 such a probability measure exists if and only if $\alpha\ge5/8$, and they provided a detailed description of these measures: When the smallest value $\alpha=5/8$, it is exactly  the law of chordal SLE$_{8/3}$ from $a$ to $b$ in $D$, so that $K$ is almost surely a simple curve from $a$ to $b$ in $D$. 
When $\alpha>5/8$, they showed that $K$ is almost surely not a simple curve anymore (the case $\alpha=1$ is the aforementioned law of the filling of a Brownian excursion in $D$ from $a$ to $b$).
In fact, they also described the law of the right boundary of $K$ for all $\alpha\ge 5/8$ in terms of a variant of SLE$_{8/3}$ (the SLE$_{8/3} ( \rho)$ processes), which showed in particular that 
the boundary of all these random sets $K$ look locally like an SLE$_{8/3}$ or equivalently like the boundary of a two-dimensional Brownian motion.
Intuitively, the larger $\alpha$ is, the ``fatter'' the random set $K$ is. For instance, one can prove that $K$ will contain cut-points if and only if $\alpha < 35/24$, see \cite {MR2060031} (where 
the law of the left boundary given the right one is also described).

\bigskip

These chordal restriction measures can in fact be viewed as special limiting cases of a larger class of restriction measures 
defined similarly, but with three marked boundary points instead of two; this class will be the topic of the present paper: 
For a simply connected domain $D\not=\mathbb{C}$ with three (different) marked boundary points $a,b$ and $c$, we consider probability measures supported on simply connected and relatively closed $K\subset \overline{D}$ such that $K\cap \partial D = \{a,b,c\}$. Such a measure is said to satisfy the \emph{trichordal conformal restriction property} if for $D'\subset D$ such that $D'$ is simply connected and $\dist( D \setminus D', \{a,b,c\})>0$, the law of $K$ conditioned to stay in $D'$ is identical to the law of $\phi(K)$ where $\phi$ is the unique conformal map from $D$ onto $D'$ that leaves the points $a,b,c$ invariant, see Figure \ref{fig:trichordal-restriction}. One expects intuitively this family of measures to be larger (ie. parametrized by more than one real parameter) because condition (i) 
of the chordal case is no longer required, so that the measures are invariant under a smaller semi-group of transformations.
\begin{figure}[h!]
\centering
\includegraphics[width=0.78\textwidth]{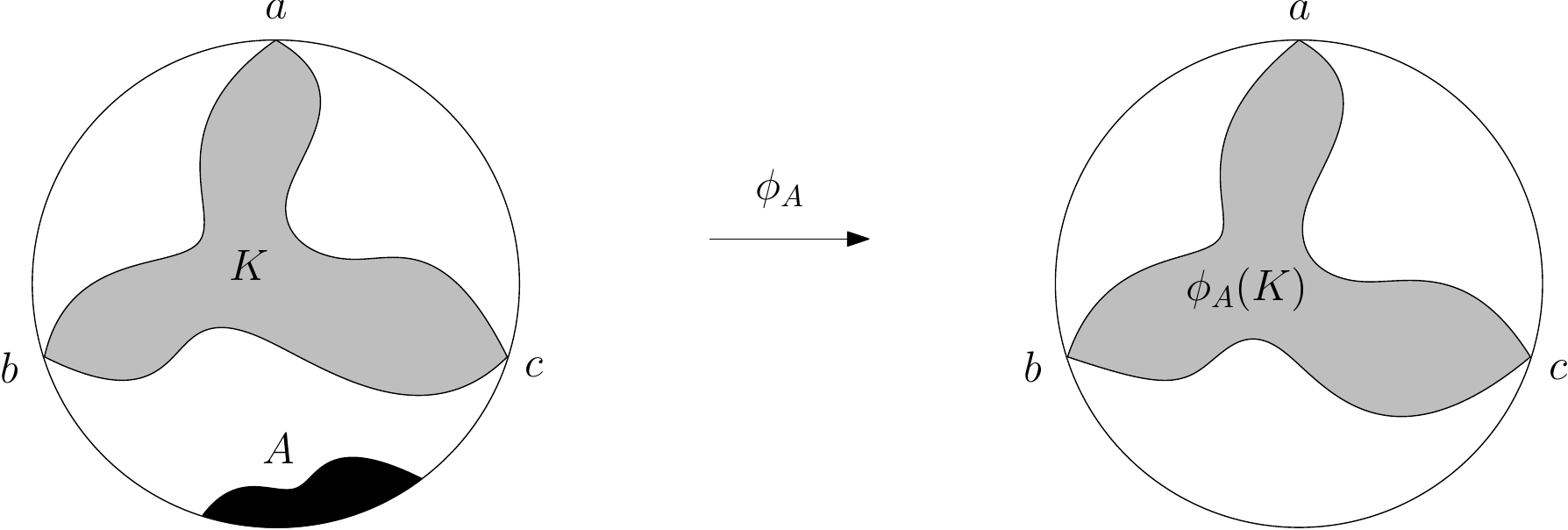}
\caption{{Trichordal restriction: } $\phi_A$ is the conformal map from $D\setminus A$ onto $D$ that leaves $a,b$ and $c$ invariant. The law of $\phi_A (K)$ given  $\{K\cap A=\emptyset\}$ is equal to  the (unconditioned) law  of $K$. }
\label{fig:trichordal-restriction}
\end{figure}

Before going into the details of this trichordal case, let us say a few words about the radial case studied by Wu  in \cite{MR3293294}. 
Given that the group of conformal automorphism of a domain 
is three-dimensional, this is the other natural variant to consider: 
For two simply connected domains, if we fix one interior point and one boundary point for each domain, then  there is a unique conformal map from one domain onto the other that sends both of the prefixed interior point  and boundary point of one domain to the corresponding points of the other.

For a simply connected domain $D\not=\mathbb{C}$ with an interior point $a$ and a boundary point $b$, one looks at probability measures supported on simply connected and relatively closed sets $K\subset \overline{D}$ such that $a\in K$ and $K\cap \partial D=\{b\}$. Such a measure is said to satisfy \emph{radial conformal restriction property} if for $D'\subset D$ such that $D'$ is simply connected and
 $\dist(D \setminus D', \{a,b\})>0$, $K$ conditioned to stay in $D'$ has the same law as $\phi(K)$ where $\phi$ is the unique conformal map from $D$ onto $D'$ that leaves the points $a,b$ invariant.

Wu showed that this family was characterized by two real parameter: For such measures, there exists $\alpha$ and $\beta$ such that for all  $A\subset D$ such that $D\setminus A$ is again a simply connected domain that contains $a$ in its interior and $b$ on its boundary, one has
\begin{align*}
\P(K\cap A=\emptyset)=|\phi_A'(a)|^\alpha \phi_A'(b)^\beta.
\end{align*}
She also showed that the range of admissible values of $\alpha$ and $\beta$ is given by $\beta \ge 5/8$ and $\alpha \le \xi (\beta)$ where $\xi (\beta) := ( ( \sqrt{24\beta+1}-1 )^2-4 ) / 48$ is the so-called disconnection exponent of $\alpha$.

Among the natural conformal restriction measures, the trichordal case was in some sense, the only one that was left to be studied.
If ones adds more marked points, then there will typically be no conformal map $\phi_A$ that fix simultaneously all these points. 
Even if it is possible to define a similar notion of $n$-point conformal restriction, we will not be able to speak of 'one measure' that satisfies conformal restriction, but one has to consider 
instead a  family of measures indexed by conformally equivalent configurations of the $n$ marked points. For $n\ge 4$, the family of all families of $n$-point chordal restriction measures will be much bigger than that of chordal, trichordal or radial cases, and it will in fact be infinite dimensional.
Intuitively speaking, one can associate to a four-point restriction sample $K$, an intrinsic conformally invariant quantity, such as the modulus $M(K)$ of the quadrilateral $a,b,c,d$ in $K$ (where $a$, $b$, $c$ and $d$ are the four marked boundary points). 
One can then weight the law of  $K$ by any function $f(M)$ with mean $1$, one thus still gets a four-point restriction measure. 
Indeed, Dub\'edat computed in \cite{MR2358649} formulae about $n$-point chordal restriction measures which involve such  functions $f$. This will be further explained in the preliminary section \S\,\ref{sec:dubedat}.

\subsection {Discussion of our results on trichordal restriction measures}

As in the chordal and radial cases, the goal in the trichordal case is to characterize all random sets satisfying the trichordal conformal restriction property.
Recall that one heuristic motivation is that one would like to describe all possible ways to connect three boundary points of a domain via random sets with ``intrinsic energy'', and one 
major question is to see what the ``thinest'' sets $K$ look like.

Let us first make the following observation: 
If we consider $K_1, K_2, K_3$ three independent chordal restriction measures in $D$ that intersect $\partial D$ respectively at pairs of points $\{a,b\}, \{b,c\}, \{c,a\}$, then the filling $K$ of $K_1\cup K_2\cup K_3$ obviously satisfies trichordal restriction. Moreover, if $K_1,K_2,K_3$ respectively  have exponents $\alpha_1, \alpha_2, \alpha_3$ then 
$$
\P\left( K \cap A=\emptyset \right) =\prod_{i=1,2,3}\P(K_i\cap A=\emptyset)=\phi_A'(a)^{\alpha_1+\alpha_3}\phi_A'(b)^{\alpha_1+\alpha_2} \phi_A'(c)^{\alpha_2+\alpha_3}.
$$
This suggests that the family of trichordal restriction measures might depend on three parameters.

This indeed turns out to be the case. The following characterization is of the same type as the characterization of the chordal and radial cases:
\begin{proposition}\label{prop:charac}[Characterization]
For each trichordal restriction measure, one can find three real numbers $\alpha,\beta,\gamma$ so that for all $A\subset D$ such that $D\setminus A$ is simply connected and $\dist(A, \{a,b,c\})>0$, one has
\begin{align}\label{thm1}
\P(K\cap A=\emptyset)=\phi_A'(a)^{\alpha}\phi_A'(b)^{\beta}\phi_A'(c)^{\gamma}.
\end{align}
Conversely, for each $\alpha, \beta, \gamma$, there exists at most one probability measure so that this holds. We then denote the law of $K$ by  $\P(\alpha,\beta,\gamma)$. 
\end{proposition}

Let us insist on the fact that this result does not tell us for which values of $\alpha$, $\beta$ and $\gamma$ such a $\P(\alpha,\beta,\gamma)$ exists. 
The main contribution of the present paper will be to give the exact range of $\alpha,\beta,\gamma$ for which  $\P(\alpha,\beta,\gamma)$ exist and to construct the corresponding restriction sets. The exponents $\alpha,\beta,\gamma$ of the trichordal case play symmetric roles and are interrelated, as opposed to the radial case.

Let us first make a few comments: 
\begin {itemize}
 \item 
First note that the points $a,b,c$ cut $\partial D$ into three arcs, namely $(ab), (bc)$ and $(ca)$.
If $A\cap\partial D$ is a subset of the arc $(bc)$ (which is the case depicted in Figure \ref{fig:trichordal-restriction}), then 
it is easy to see that $\phi_A'(a)>1$ and  $\phi_A'(b)<1, \phi_A'(c)<1$. In order for the right-hand side of (\ref{thm1}) to be a probability, 
$\alpha$ can therefore not be much bigger than both $\beta$ and $\gamma$. By symmetry, this indicates that each of the three  exponents needs to be bounded from above by a certain function of the other two. On the other hand, it is 
easy to see that $\phi_A' (a) \phi_A' (b) \phi_A' (c)$ is always smaller than $1$. Hence, one can still say that in the symmetric family $\P (\alpha, \alpha, \alpha)$, the smaller $\alpha$ is, the skinnier $K$ is (because for each $A$, the probability 
to intersect $A$ is an increasing function of $\alpha$).
\item
Another type of condition that one should expect is that $\alpha$, $ \beta$ and $\gamma$ should all be greater or equal to $5/8$.  Intuitively, a random set $K$ with law $\P(\alpha,\beta,\gamma)$ should look like,  in a infinitesimal neighborhood of one marked point (say $a$),  a chordal restriction measure of the corresponding exponent (which is $\alpha$ in the case of $a$).
More precisely, if we let the point $b$ tend to the point $c$, the limiting measure can be proved to be equal to the chordal restriction measure of exponent $\alpha$.
Since chordal restriction measures of exponent $\alpha$ as characterized by (\ref{two-point-formula}) exist only for $\alpha\ge5/8$, the same condition should hold here.
\end {itemize}

A natural question is what is smallest $\alpha$ for which the measure $\P(\alpha, \alpha, \alpha )$ exists.  For the chordal case, the measure of exponent $\alpha=5/8$ corresponds to a simple curve, which is the thinest that one can obtain. If the measure $\P(5/8,5/8,5/8)$ exists, then the corresponding random set would a.s. be a simple curve near each of the boundary points $a,b,c$. Would {the} three legs of such an SLE$_{8/3}$ spider then merge at one single point in the middle  or would there be a fat ``body'' set near the center? 
The previously described method of taking the filled union of chordal restriction measures only constructs such measure with $\alpha \ge 5/4$, which is unlikely to be the thinest choice.
The answer to all these questions turns out to be somewhat surprising: A particular case of the following theorem is that that the smallest $\alpha$ for which $\P ( \alpha, \alpha, \alpha)$ exists is actually $\alpha=20/27$. In this special case, the corresponding random set $K$ consists of three parts, that have only one common point which is the unique \emph{triple disconnecting point}: When one removes this point, one disconnects $K$ into three disjoint connected components, that respectively contain $a$, $b$ and $c$. However, since $20/27  > 5/8$, each of these three ``legs'' of this 
symmetric restriction spider are a bit ``fatter'' than simple curves (even though they each of them will have quite a lot of cut-points). 

Before stating this main result, let us first recall the formulas for the half-plane and whole-plane \emph{Brownian intersection exponents} that have been determined 
by Lawler, Schramm and Werner in \cite{MR1879850, MR1879851, MR1899232} :
\begin{align*}
&\tilde\xi(\alpha_1,\alpha_2)=\frac{\left(\sqrt{24\alpha_1+1}+\sqrt{24\alpha_2+1}-1\right)^2-1}{24},\\
&\xi(\alpha_1,\alpha_2,\alpha_3)=\frac{\left(\sqrt{24\alpha_1+1}+\sqrt{24\alpha_2+1}+\sqrt{24\alpha_3+1}-3\right)^2-4}{48}.\\
\end{align*}
We will use here only the formulas and not the interpretation in terms of non-intersection probabilities for Brownian paths (but the latter are of course related 
to  restriction properties).

\begin{theorem}\label{thm:existence}[Existence]
The measure $\P(\alpha,\beta,\gamma)$ exists if and only if $\alpha,\beta,\gamma$ satisfy all the following conditions:
\begin{itemize}
\item[(i)] $\alpha,\beta,\gamma\ge 5/8,$
\item[(ii)] $\alpha\le\tilde\xi(\beta,\gamma), \beta\le\tilde\xi(\alpha,\gamma), \gamma\le\tilde\xi(\alpha,\beta),$
\item[(iii)] $\xi (\alpha,\beta,\gamma)\ge2,$
\end{itemize}

Furthermore, if $K$ is a random set  with law $\P(\alpha,\beta,\gamma)$, then
\begin{itemize}
\item[-] $K$ has a triple disconnecting point if and only if $\xi (\alpha,\beta,\gamma)=2$. Moreover in this case,  each of the three legs away from this unique triple disconnecting point
have almost surely infinitely many cut-points (this can be viewed as a consequence of the fact that necessarily, $\alpha, \beta, \gamma \le 1 < 35/24$). See Figure \ref{fg:tripple-cut-point}.

\item[-] The boundary point $c$ is a cut-point of $K$ if and only if $\gamma=\tilde\xi(\alpha,\beta)$. See Figure \ref{fig:two-conditioned}.
\begin{figure}[h]
\centering
\begin{subfigure}[t]{0.4\textwidth}
\centering
\includegraphics[width=0.74\textwidth]{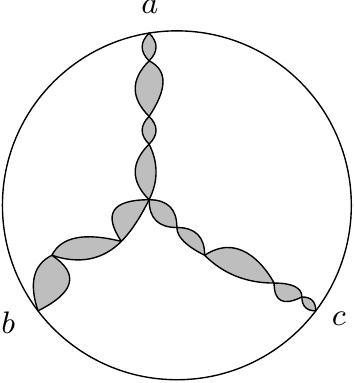}
\caption{The case $\xi(\alpha,\beta,\gamma)=2$ with the triple disconnecting point. This can be viewed as a sketch of the symmetric restriction spider when $\alpha=\beta=\gamma =20/27$.}
\label{fg:tripple-cut-point}
\end{subfigure}
\qquad\qquad
\begin{subfigure}[t]{0.4\textwidth}
\centering
\includegraphics[width=0.74\textwidth]{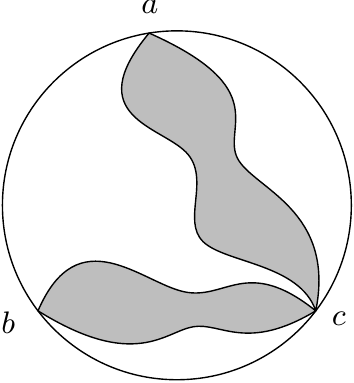}
\caption{The case $\gamma=\tilde \xi (\alpha,\beta)$. It can be interpreted as two chordal restriction samples conditioned not to intersect each other.
}
\label{fig:two-conditioned}
\end{subfigure}
\caption{Two extreme cases of trichordal restriction measures.}
\label{fig:existence}
\end{figure}
\end{itemize}
\end{theorem}

Note that if we define $\hat \alpha$, $\hat \beta$, $\hat \gamma$ to be the respective image of $\alpha$, $\beta$ and $\gamma$ under the monotone bijection from $\R_+$ onto itself defined by 
$ x \mapsto -1 + \sqrt {24x +1 }$, then one can easily express the conditions (i), (ii) and (iii) 
without reference to the Brownian intersection exponents: the condition (i) becomes $\min (\hat \alpha, \hat \beta, \hat \gamma) \ge 3$, 
condition (iii) becomes $\hat \alpha + \hat \beta + \hat \gamma \ge 10$ (while (i) implied only that this sum cannot be smaller than 9), and 
condition (ii) becomes three inequalities of the type $\hat \alpha + \hat \beta \ge \hat \gamma  $. In other words, all these conditions can be summed up as
$$ \min (\hat \alpha, \hat \beta, \hat \gamma) \ge 3, \quad \hat \alpha + \hat \beta + \hat \gamma \ge  \max ( 10, 2\hat \alpha, 2\hat \beta, 2\hat \gamma).$$

We have already highlighted to special role of $\P ( 20/27, 20/27, 20/27) $. Another interesting case worth mentioning is the measure $\P ( 5/8, 5/8, 1)$ (that also exhibits a triple disconnecting point). 
Our construction will in fact implicitly describe this measure as being constructed by an SLE$_{8/3}$ from $a$ to $b$, weighted according to its (renormalized)
harmonic measure seen from $c$ (this will make it more likely for the path to wander closer to $c$), and to which one attaches a Brownian excursion from $c$ to a point (chosen according to this renormalized harmonic measure) on this SLE. 

In general, when $\xi(\alpha,\beta,\gamma)=2$, it is natural to guess that conditionally on the triple disconnecting point, $K$ consists of three radial restriction samples (from the center point 
to the three points $a$, $b$ and $c$ respectively) conditioned not to intersect each other except at this center point. 
However, in practice, it would be a rather tedious and intricate venture to define this conditioning rigorously, and we will not follow this route. 

For the $\gamma=\tilde\xi(\alpha,\beta)$ case, one can make a similar remark, and this time, it can be made rigorous more easily. It will indeed follow from our construction that conditioned on one branch of $K$, the other branch is a chordal restriction sample in one of the connected components of the complement of the first branch. In this sense, one can view $K$ as the union of two chordal restriction samples conditioned not to intersect each other except at $c$. An interesting 
example is of course $\P ( 5/8, 5/8, 2)$, which corresponds to two chordal SLE$_{8/3}$ from $a$ to $c$ and from $b$ to $c$, conditioned not to intersect. This (and the three symmetric images when interchanging $a$, $b$ and $c$) is the 
only trichordal restriction measure that consists of a simple path in $\overline D$. 

The family of measure $\P(5/8, 5/8, \gamma)$ exists exactly in the range $\gamma \in [1,2]$, and the two cases $\P(5/8, 5/8, 1)$ and $\P( 5/8, 5/8, 2)$ that we have just briefly described are the two extremal ones. The case $\P (5/8, 5/8, 5/4)$ corresponds to {(the filling of) the union of} two independent SLE$_{8/3}$ curves from $a$ to $c$ and from $b$ to $c$. As explained above, when $\gamma$ increases, as opposed to the symmetric case, one can not really argue anymore that the sets $K$ become larger (and indeed, the $\gamma=2$ case corresponds to a simple curve, while this is not the case for $\gamma=1$). Instead, when $\gamma$ grows, the sets are in some sense  ``pushed towards $c$''.

\subsection {Outline and further results} 

The proof of the characterization result will follow somewhat similar ideas than in the chordal and than in the radial cases. The proof of the existence part will be divided in two parts: 
First, we will construct trichordal restriction measures for all admissible values of the parameters $\alpha$, $\beta$ and $\gamma$, and then we will see that if a measure $P({\alpha, \beta, \gamma})$ exists for the non-admissible part, one gets a contradiction.  

Our construction of trichordal restriction measures will rely on a special family of SLE processes with driving functions involving hypergeometric functions, that we call hypergeometric SLEs and denote by hSLE. This family of hypergeometric SLEs depends on $\kappa$ and three extra parameters $\mu,\nu,\lambda$ and contains the SLE$_\kappa(\rho)$ processes as special cases when $\mu,\nu,\lambda$ take extremal values. It also contains the intermediate SLEs that Zhan has defined in \cite{MR2646499} to describe the reverse of SLE$_\kappa(\rho)$ processes. 
In the present paper, we will make use of the hSLE curves with $\kappa=8/3$ to construct the boundaries of trichordal restriction samples.

The construction goes in fact through several steps. First of all, we define the notion of
{one-sided restriction} (in this trichordal case) -- which is reminiscent and analogous to the one-sided chordal restriction measures that were studied by Lawler, Schramm and Werner.
We consider measures supported on simply connected and relatively closed $K\subset\overline D$ such that $K\cap\partial D$ is the union of the arcs $(ab)$ and $(bc)$.
For all $A\subset \overline D$ such that $D\setminus A$ is again a simply connected domain and $\overline A \cap \partial D$ is a subset of the arc $(ca)$.
Such a measure is said to satisfy \emph{one-sided (trichordal) conformal restriction property} if $K$ conditioned on $\{K\cap A=\emptyset\}$ has the same law as $\phi_A(K)$ where $\phi$ is the unique conformal map from $D\setminus A$ onto $D$ that preserves the points $a,b,c$.
See Figure \ref{fig:one-sided}.
\begin{figure}[h]
\centering
\includegraphics[width=0.76\textwidth]{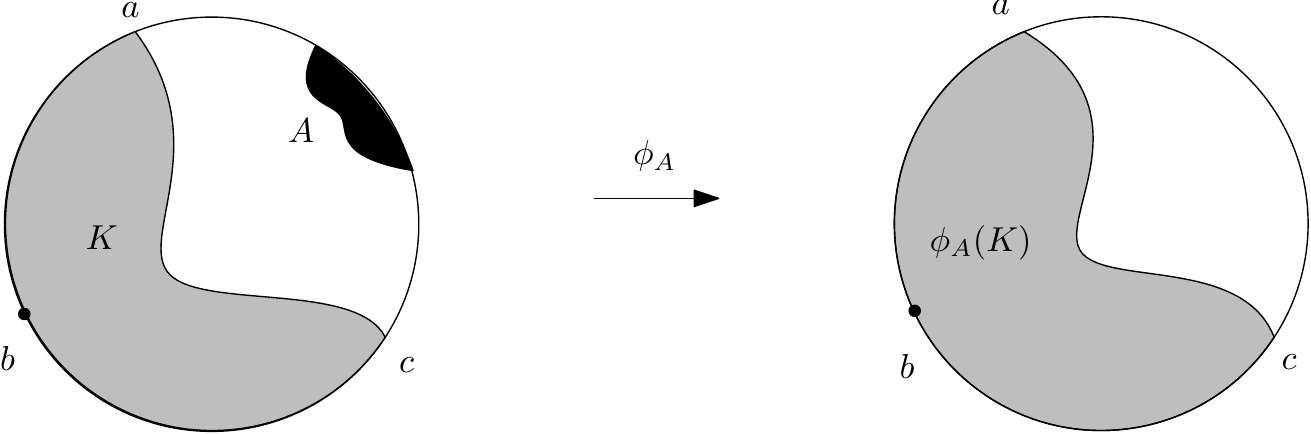}
\caption{{One-sided (trichordal) restriction:} $\phi_A$ is the conformal map from $D\setminus A$ onto $D$ that preserves the points $a,b,c$. The law of $\phi_A(K)$ conditioned on $\{K\cap A=\emptyset\}$ is equal to the (unconditioned) law of $K$.}
\label{fig:one-sided}
\end{figure}

One-sided restriction measures are also characterized by (\ref{thm1}) with three parameters $\alpha,\beta,\gamma$ and we denote the measure by $\P^1(\alpha,\beta,\gamma)$. The range of $\alpha,\beta,\gamma$ for which the measure $\P^1(\alpha,\beta,\gamma)$ exists is larger than the range for which the three-sided restriction measure $\P(\alpha,\beta,\gamma)$ exists. 
In particular, since we only consider $A$ that intersect $\partial D$ at the arc $(ac)$, there is nothing that prevents $\alpha$ or $\gamma$ to be much bigger than the other two exponents. Moreover,
some negative values of  $\beta$ will be allowed. We will show that the range is the following:
\begin{theorem}
A one-sided measure $\P^1(\alpha,\beta, \gamma)$ exists if and only if $\alpha, \beta, \gamma$ satisfy the following conditions:
\begin{align*}
\alpha\ge 0, \quad \gamma \ge 0,\quad \beta \le \tilde \xi(\alpha, \gamma).
\end{align*}
Moreover, if $K$ has law $\P^1(\alpha,\beta, \gamma)$ where $\beta=\tilde\xi(\alpha,\gamma)$, then the point $b$ is on the right boundary of $K$.
\end{theorem}
These one-sided restriction measures $\P^1(\alpha,\beta,\gamma)$ will be constructed, in most of the cases, by a hSLE$_{8/3}$ curve. However for a certain range of parameters, this construction does not work, and we will use Poisson point process of Brownian excursions instead.

Next we will study two-sided restriction measures. Their definition is similar to that of one-sided restriction, except that two-sided restriction measures are supported on $K$ such that $K\cap \partial D$ is the union of $\{ a \}$ and the arc $(bc)$, and the sets $A\subset D$ are such that  $\overline A \cap \partial D$ is a subset of  the union of the arcs $(ca)$ and $(ab)$, see Figure \ref{fig:two-sided}.
\begin{figure}[h]
\centering
\includegraphics[width=0.78\textwidth]{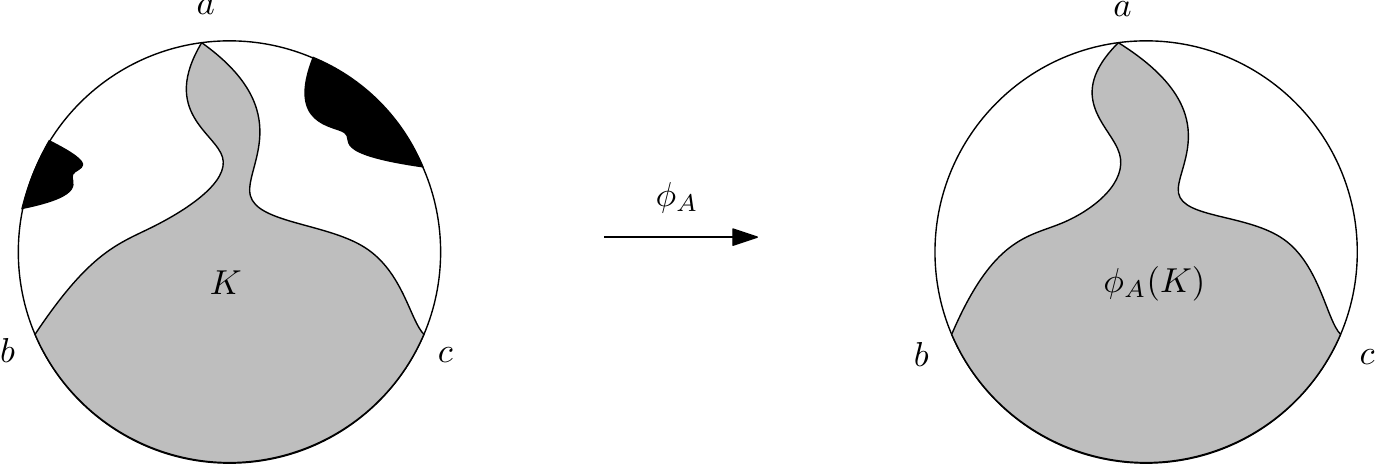}
\caption{{Two-sided (trichordal) restriction:} $\phi_A$ is the conformal map from $D\setminus A$ onto $D$ that preserves the points $a,b,c$. The law of $\phi_A(K)$  conditioned on $\{K\cap A=\emptyset\}$ is equal to the (unconditioned) law of $K$.}
\label{fig:two-sided}
\end{figure}

Two sided restriction measures are also characterized  by (\ref{thm1}) with three parameters $\alpha,\beta,\gamma$ and we denote the measure by $\P^2(\alpha,\beta,\gamma)$.
The range of $\alpha,\beta,\gamma$ for which the measure $\P^2(\alpha,\beta,\gamma)$ exists is given in the following theorem.
\begin{theorem}
The two-sided measure  $\P^2(\alpha,\beta, \gamma)$ exists if and only if $\alpha, \beta, \gamma$ satisfy the following conditions:
\begin{align*}
\alpha\ge\frac58,\quad\beta\ge 0,\quad \gamma\ge0,\quad \beta\le\tilde\xi(\alpha,\gamma),\quad \gamma\le\tilde\xi(\alpha,\beta), \quad \tilde\xi(\alpha,\beta,\gamma)\ge 1.
\end{align*}
Moreover, if $K$ has the law $\P^2(\alpha,\beta, \gamma)$  where $\tilde\xi(\alpha,\beta,\gamma)= 1$, then there exist a point on the arc $(bc)$ which is both on the left and the right boundary of $K$.
\end{theorem}
The two-sided restriction measures will be constructed, in most of the cases, as follows: Construct first its boundary at one side (say the right side) which is a one-sided restriction measure, by a hSLE, then conditionally on this side, sample an independent one-sided restriction measure in the connected component which is to the other side (say the left side) of the {first boundary}. 
However we will again use other methods to construct some limiting cases. 

In order to prove that the hSLE define such trichordal restriction measures, we will use a martingale-type argument, in the spirit of the proof of the chordal restriction measure construction. However, the technical implementation of this strategy can appear somewhat daunting, but things are not as bad as they look: One derives that a certain process is a local martingale by longish but straightforward It\^o formula calculations (having the a priori information that the computation should work out nicely) and in order to be able to apply the optional stopping theorem, we have to prove that the semi-martingale is anyway bounded (this is not obvious from the expression of the process as product of terms that are not bounded, but we have again the a priori knowledge that in the end, this quantity will be a conditional probability, so that it should be bounded), which one can then show using some a priori knowledge on the hypergeometric functions that we use.

Having a full description of all possible one-sided and two-sided restriction measures, we will be able to construct the family of three-sided restriction measures by first constructing the boundary on one side (say the right side) as a one-sided measure and then conditionally on this side, sampling  a two-sided restriction measure in the connected component which is to the other side (say the left side) of this first boundary.

The determination of the exact range of admissible parameters in one-sided and two-sided cases, will be obtained by investigating some geometric properties of the corresponding random sets that we 
construct in the limiting cases of the parameter-range. In the three-sided case, we will need to do a kind of reverse procedure of the previous construction (showing that a three-sided restriction measure can be decomposed as described above). This will be explained in \S\,\ref{sec:three-sided}.

\section{Notations and preliminaries}\label{sec:prel}

First we would fix some notations and give a more complete definition of the trichordal restriction property.
After that, we will briefly recall some relevant background material on SLEs and Brownian excursions. We will also review in more detail the 
chordal  and the radial  conformal restriction properties. 

\subsection {Notations}
Let $(D,z_1,z_2,z_3)$ be a configuration where $D\not=\C$ is a simply connected domain and $z_1,z_2,z_3$ denote three different prime ends of $D$. When the boundary of $D$ is a Jordan curve, then the set of prime ends is exactly in 
bijection with $\partial D$. In the sequel, by slight abuse of notation, we will often just refer to the set of prime ends as $\partial D$. 
 Let $\Omega(D, z_1,z_2,z_3)$ be the collection of simply connected and closed $K\subset \overline D$ such that $K\cap\partial D=\{ z_1,z_2,z_3\}$. 
Let $\mathcal{Q}_D^{z_1,z_2,z_3}$ be the collection of  $A\subset \overline D$ such that $A=\overline {D\cap A}$, $D\setminus A$ is simply connected and $A\cap\{z_1,z_2,z_3\}=\emptyset$.
We endow $\Omega(D, z_1,z_2,z_3)$ with the $\sigma-$algebra generated by events $\{K\in\Omega(D, z_1,z_2,z_3), K\cap A=\emptyset\}$ for $A\in\mathcal{Q}_D^{z_1,z_2,z_3}$.
In the sequel, $\P_{(D,z_1,z_2,z_3)}$ will always implicitly denote a probability measure on the space $\Omega(D,z_1,z_2,z_3)$.
A family of probability measures $\P_{(D,z_1,z_2,z_3)}$  is said to verify \emph{trichordal conformal restriction property} if it satisfies the following properties.
\begin{itemize}
\item[(i)] (Conformal invariance) If $\varphi: D\to D'$ is a bijective conformal map, then $$\varphi\circ\P_{(D,z_1,z_2,z_3)}=\P_{(D',\varphi(a),\varphi(b),\varphi(c))}.$$
\item[(ii)] (Restriction) For all $A\in\mathcal{Q}_D^{z_1,z_2,z_3}$, let $\varphi_A: D\setminus A\to D$ be the unique bijective conformal map such that $\varphi_A(z_i)=z_i$ for $i=1,2,3$. Let $K$ be a sample with law $\P_{(D,z_1,z_2,z_3)}$, then the law of $\varphi_A(K)$ conditioned on $K\cap A=\emptyset$ is equal to the (unconditioned) law of $K$.
\end{itemize}

By conformal invariance, one can restrict the study of such measures  for the  $(D, z_1, z_2, z_3) = (\H, 0, -1, \infty)$. 
In this upper-half plane setting, we 
let $\mathcal{Q}$ be the set of bounded closed $A\subset\overline\H$ such that $\H\setminus A$ is simply connected.
We will also frequently use the following subsets of $\mathcal{Q}$ :
\begin{itemize}
\item For $I\subset\R$, let $\mathcal{Q}^{I}:= \{A\in\mathcal{Q}; A\cap \R\subset I\}$. We will mainly consider $\mathcal{Q}^{(-1,0)}, \mathcal{Q}^{(-\infty,-1)}$ and $\mathcal{Q}^{(0,\infty)}$.
\item Let $\mathcal{Q}^*:=\{A\in\mathcal{Q}, 0\not\in A\}$. Let $\mathcal{Q}^{**}:=\{A\in\mathcal{Q};  -1,0\not\in A\}$.
\end{itemize}
For each $A\in\mathcal{Q}$, there is a unique conformal map $g_A:\H\setminus A\to\H$ with the normalization 
$$g_A(z)=z+\capacity(A)/z+o(1/z) \quad \mbox{as }  z\to\infty,$$ where $\capacity(A)\ge 0$ is called the \textit{half-plane capacity} of $A$.

In the $(\H,0,-1,\infty)$ setting, the formula (\ref{thm1}) in Proposition \ref{prop:charac} becomes
\begin{align}\label{c}
\P[K\cap A=\emptyset]=g_A'(-1)^{\beta}g_A'(0)^{\gamma}(g_A(0)-g_A(-1))^{\alpha-\beta-\gamma}.
\end{align}
Note that
 (\ref{c}) can be alternatively written as
\begin{align*}
\P[K\cap A=\emptyset]=\varphi_A'(-1)^\beta\varphi_A'(0)^\gamma\varphi_A'(\infty)^\alpha,
\end{align*}
where $\varphi_A(\cdot):=(g_A(\cdot)-g_A(0))/(g_A(0)-g_A(-1))$ is the conformal map from $\H\setminus A$ to $\H$ that preserves the three points $-1,0,\infty$, and $\varphi_A'(\infty)$ is the limit as $z\to\infty$ of $z/\varphi_A(z)$.

\subsection{Loewner evolutions}
Suppose that $W=(W_t,t\ge 0)$ is a real-valued continuous function. For each $z\in\overline \H$, let $g_t(z)$ be the solution of the initial value problem
\begin{align}\label{loewner}
\partial_t g_t(z)=\frac{2}{g_t(z)-W_t}, \quad g_0(z)=z.
\end{align}
For each $z\in\overline\H$ there is a time $\tau(z)\in[0,\infty]$ such that the solution $g_t(z)$ exists for $t\in[0,\tau]$ and $\lim_{t\to\tau}g_t(z)=W_\tau$ if $\tau<\infty$. 
The set $K_t:=\{z\in\overline\H:\tau(z)\le t\},t\ge 0$ is then the chordal Loewner hull generated by the driving function $(W)$ at time $t$.
It is easy to check that $g_t$ is the unique conformal map from $\H\setminus K_t$ onto $\H$ such that
$g_t(z)=z+2t/z+o(1/z)$ as $z\to\infty$ and that $\capacity(K_t)=2t.$

The \emph{Schramm-Loewner evolution} (denoted by SLE) was introduced by Schramm in \cite{MR1776084} as candidate for scaling limits of discrete models: 
Chordal SLE$_\kappa$ in $\H$ from $0$ to $\infty$ is the Loewner evolution that one obtains when the driving function is 
 $W_t=\sqrt{\kappa}B_t$ and $(B_t,t\ge 0)$ is  a standard one-dimensional Brownian motion.

A useful variant of SLE are the SLE$_\kappa (\rho)$ processes (see \cite {MR1992830,MR2118865}) is the following: 
 Let $x_1,\cdots,x_n\in\R, \rho_1,\cdots, \rho_n\in(-2,\infty)$. Let $(W_t,O_t^1,\cdots,O_t^n)$ be the solution to the system of stochastic differential equations
\begin{align*}
&d W_t=\sqrt{\kappa} dB_t+\sum_{k=1}^n \frac{\rho_k}{W_t-O^k_t}dt;\\
& dO^k_t=\frac{2dt}{O^k_t-W_t}, \quad k=1,\cdots,n;\\
& W_0=0;\quad O^k_0=x_k,\quad k=1,\cdots,n.
\end{align*}
The corresponding random Loewner chain generated by the driving function $W$ is called a  chordal SLE$_\kappa(\rho_1,\cdots,\rho_n)$ having $(x_1,\cdots,x_n)$ as marked points. It is possible to choose one of the marked points $x_k$ 
to be $0^-$ or $0^+$.

\subsection{Brownian excursions}\label{sec:brownian-excursions}
In the upper-half plane, for any point $x$ on the real line, let $\nu_{\H,x}$ be a probability measure which is the rescaled limit as $\eps\to 0$ of a Brownian motion started at $x+i\eps$ and killed at the exit of $\H$.
It is possible to decompose the measure according to the location of the exit point $y$. Then,
\begin{align*}
\nu_{\H,x}=\int_\R \frac{dy}{(x-y)^2}\,\nu_{\H,x,y},
\end{align*}
where $\nu_{\H,x,y}$ is the probability measure of a Brownian bridge from $x$ to $y$ that stays in $\H$.
It is also possible to integrate the starting point $x$ according to the Lebesgue measure: 
\begin{align}\label{brownian-excursion-measure}
\mu_\H = \int dx \nu_{\H, x} = \int \int \frac{dx\,dy}{(x-y)^2}\,\nu_{\H,x,y}.
\end{align}
Note that $\mu_\H$ is a measure with infinite total mass.
It is then well-known (\cite {MR2178043}) that
\begin{itemize}
\item[(i)]
 $\mu_\H$ is conformal invariant under M\"obius transformation from $\H$ to itself.
 \item[(ii)]
 $\mu_\H$ satisfies the following restriction property :
for all $A\in\mathcal{Q}$, $g_A \circ \mu_\H\mathbf{1}_{\gamma\in \H\setminus A}=\mu_\H$.
\end{itemize}

For $\theta>0$, consider a Poisson point process ${\mathcal P}$  of Brownian excursions with  intensity 
$$\theta \int_{\R-} \int_{\R_-} \frac{dx\,dy}{(x-y)^2}\,\nu_{\H,x,y}.$$
For any set $C\subset\H$, let $\mathcal{F}$ be the function that sends $C$ to the filling of $C$, i.e. $\mathcal{F}(C)$ is the closure of the complement in $\H$ of the unbounded connected component of $\H\setminus C$.
Then, as explained in \cite[\S\, 4.3]{MR2178043}, it turns out that $\mathcal{F}\left( \mathcal{P} \right)$ satisfies one-sided chordal conformal restriction 
with an exponent $c \theta$ for certain given positive proportionality factor $c$.

\subsection{Conformal restriction: chordal case}
Here we make a brief  summary of the main results on the chordal case studied by Lawler, Schramm and Werner \cite{MR1992830}.
Let $\Omega$ be the collection of simply connected and relatively closed $K\subset \overline \H$ such that $K\cap\partial\H=\{0,\infty\}$.
Recall that a probability measure $\P$ on $\Omega$ (or the random set $K$ under the law $\P$) is said to verify the (two-sided) \emph{chordal conformal restriction property} if 
\begin{itemize}
\item[(i)](conformal invariance) the law of $K$ is invariant under any Mobius transform from $\H$ to itself.
\item[(ii)](restriction) for all $A\in\mathcal{Q}^*$, $\phi_A(K)$ conditioned on $K\cap A=\emptyset$ has the same law as $K$.
\end{itemize}
For all $A\in\mathcal{Q}^*$, let $\phi_A(\cdot):=g_A(\cdot)-g_A(0)$.
Proposition 3.3 in \cite{MR1992830} states that if $K$ is a chordal restriction measure, then there exist a constant $\alpha$ such that for all $A\in\mathcal{Q}^*$,
\begin{align*}
\P(K\cap A=\emptyset)=\phi_A'(0)^\alpha
\end{align*}
and that this formula characterizes uniquely the restriction measure $\P$.

In order to show that the range of admissible values of $\alpha$ is $[5/8, \infty)$, the argument in \cite{MR1992830} goes as follows: 
\begin {itemize}
 \item The measures are constructed explicitly for all $\alpha \ge 5/8$. 
 \item SLE$_{8/3}$ satisfies two-sided conformal restriction with exponent $5/8$, and it is a simple curve.
 \item Sets that satisfy a one-sided conformal restriction properties are all constructed (this corresponds to all positive values of $\alpha$). If the two-sided restriction measure exists, the right-boundary of its samples does defines a one-sided conformal restriction measure.  By symmetry, the probability that the point $i$ lies to the right of the two-sided restriction measure can not be larger than $1/2$ (because it is also equal to the probability that it lies to the left). The fact that SLE$_{8/3}$ is a simple curve and that 
 this probability is decreasing with $\alpha$ then implies that when $\alpha < 5/8$, this probability would be larger than $1/2$, and therefore that two-sided restriction measures can not exist.
 \end {itemize}
 Let us mention some ways to construct chordal restriction measures (other constructions, for instance via SLE$_{8/3} (\rho)$ processes are also possible):
 \begin{enumerate}
 \item First construct in $\H$ a SLE$_\kappa$ from $0$ to $\infty$ with $\kappa={6}/(2\alpha+1)$. Then consider  independently a Brownian loop-soup in $\H$ with intensity $\lambda={(8-3\kappa)(6-\kappa)}/{2\kappa}$. Finally, define the filling $K$ of the union of the SLE curve and the loops that it intersects.
\item In Section 5.2 of \cite{MR2060031}, two-sided restriction samples are constructed by first constructing their right boundary $\gamma$ (that is a one-sided restriction sample), then putting another one-sided restriction sample in the domain between $\R^-$ and $\gamma$ (and we will be using similar ideas in the trichordal setting). 
It is possible to construct one-sided restriction samples of any exponent $\alpha\in(0,\infty)$ using SLE$_{8/3}(\rho)$ processes for $\rho\in(-2,\infty)$, or by Brownian excursions as 
recalled in \S\,\ref{sec:brownian-excursions}.
\end{enumerate}
Note that once restriction measures are described in $\H$, one gets for free the description in any simply connected domain $D$ other than $\C$ by conformal invariance.

\subsection{Conformal restriction: radial case}\label{sec:radial}
The radial case of conformal restriction was studied by Wu \cite{MR3293294}. Here we make a short review of the results and of some of the arguments. 
Let $\Omega$ be the collection of simply connected and relatively closed subsets $K$ of the closed unit disc $\overline{\mathbb{U}}$ such that $K\cap\partial\mathbb{U}=\{1\}$ and $0\in K$.
A probability measure $\P$ on $\Omega$ (or the random set $K$ under the law $\P$) is said to verify \emph{radial conformal restriction property} if for any closed subset $A$ of the closed unit disc such that $\dist(A,\{0,1\})>0$, and such that $\mathbb{U}\setminus A$ is simply connected, the law of $\phi_A(K)$ conditioned on $\{K\cap A=\emptyset\}$ is equal to the law of $K$ where $\phi_A$ is the conformal map from $\mathbb{U}\setminus A$ onto $\mathbb{U}$ that fixes both the boundary point $1$ and the origin.

As already mentioned in the introduction, the main theorem of \cite{MR3293294} states on the one hand that a radial restriction measure is fully characterized by a pair of real numbers $(\alpha,\beta)$ such that
\begin{align*}
\P(K\cap A=\emptyset)=\left| \phi_A'(0) \right|^\alpha \phi_A'(1)^\beta,
\end{align*}
and on the other hand that the range of admissible values of $\alpha$ and $\beta$ is exactly 
\begin{align*}
\beta\ge\frac58, \quad \alpha\le\left( \left( \sqrt{24\beta+1}-1 \right)^2-4 \right) / 48.
\end{align*}
The proof uses the following steps: 
\begin {itemize}
 \item First, the characterization part is established directly by rather soft arguments reminiscent of the chordal case. 
 \item The measures $P_{\alpha, \beta}$ for $\beta \ge 5/8$ and $\alpha = \xi (\beta)$ are constructed explicitly using variants of the radial SLE$_{8/3}$ processes. It turns out that they all have the property that a sample of this measure always has the origin on its boundary. 
 \item When a measure $P_{\alpha, \beta}$ exists, it is always possible to construct $P_{\alpha', \beta}$ for all $\alpha' < \alpha$ by adding to the former a Poisson point process of Brownian loops 
 that surround the origin. In particular, a sample of the latter then do not have the origin on its boundary anymore. This shows in particular that $P_{\alpha, \beta}$ can not exist for 
 $\beta \ge 5/8$ and $\alpha > \xi (\beta)$. 
 \item Finally, by comparison with the chordal case, one can see that $\beta$ can not be smaller than $5/8$.
\end {itemize}

\subsection{Dub\'edat's remark on the $n$-point case}\label{sec:dubedat}

There is a natural extension of the conformal restriction property to the $n$-point case. 
Let $(D,z_1,\cdots,z_n)$ be a configuration where $D\subset\C$ is a simply connected domain and $z_1,\cdots,z_n\in\partial D$. Let $\Omega(D, z_1, \cdots, z_n)$ be the collection of simply connected and relatively closed $K\subset \overline D$ such that $K\cap\partial D=\{ z_1, \cdots, z_n\}$. Let $\P_{(D,z_1,\cdots,z_n)}$ be a probability measure on $\Omega(D,z_1,\cdots,z_n)$.
A family of probability measures $\P_{(D,z_1,\cdots,z_n)}$  is said to verify \emph{$n$-point chordal conformal restriction property} if it satisfies the following properties.
\begin{itemize}
\item[(i)] (Conformal invariance) If $\varphi: D\to D'$ is a bijective conformal map, then $$\varphi\circ\P_{(D,z_1,\cdots,z_n)}=\P_{(D',\varphi(z_1),\cdots,\varphi(z_n))}.$$
\item[(ii)] (Restriction) For all $A\subset D$ such that $D\setminus A$ is simply connected and $\dist(A, \{z_1,\cdots,z_n\})>0$ is positive. Let $K$ be a sample with law $\P_{(D,z_1,\cdots,z_n)}$. Then $\P_{(D,z_1,\cdots,z_n)}$ conditionally on $\{K\cap A=\emptyset\}$ is equal to $\P_{(D\setminus A,z_1,\cdots,z_n)}$.
\end{itemize}

In  \cite{MR2358649} (see also \cite{MR2253875}), Dub\'edat, based on formal algebraic arguments (in the spirit of the formal Section 8.5 in \cite{MR1992830}) argues that if $(\P_{\H,z_1,\cdots,z_n})$ is a family of $n$-point restriction measures,
 then, under some regularity assumptions,
  there should exist
\begin{itemize}
\item
real numbers $\nu_{ij} , 1\le i<j\le n$
\item
a continuous function $f: \R^n \to \R$ which is invariant under $\H\to \H$ Mobius transformations
\end{itemize}
such that for all $A\subset \H$ such that $A\cap\{z_1,\cdots,z_n\}=\emptyset$ and $\H\setminus A$ is simply connected,
\begin{align}\label{dubedat}
\P(K\cap A=\emptyset)=\prod_{i<j} \left( \varphi_A'(z_i) \varphi_A'(z_j) \left( \frac{z_j-z_i}{\varphi_A(z_j)-\varphi_A(z_i)} \right)^2 \right)^{\nu_{ij}}
\frac{f\left( \varphi_A(z_1),\cdots,\varphi_A(z_n) \right)}{f(z_1,\cdots,z_n)}.
\end{align}

For $n=2$ and $3$, the function $f$ can only be constant hence families of chordal and trichordal restriction measures are characterized respectively by $1$ and $3$ parameters.
For $n \ge 4$, the function $f$ comes up and the families will be infinite-dimensional.

This formal argument is indeed convincing on heuristic level, but the actual formal regularity issues are not addressed in \cite{MR2358649}.  In  the next section of the 
present paper, we will provide a complete self-contained characterization of the special case of trichordal restriction measures. Let us stress also that the formula (\ref{dubedat}) does not 
indicate for which values of $\nu_j$ the measure exists, which will be one main focus of the present paper.

\subsection {Remarks on Brownian exponents}\label{sec:brownian-exponents}

As pointed out in \cite{MR1742883}, the structure of Brownian half-plane intersection exponents $\tilde \xi$ is of the form
$$ \tilde \xi ( \alpha_1, \ldots, \alpha_n) = U^{-1} ( U (\alpha_1) + \ldots + U (\alpha_n))$$
for some function $U$ (this relation was then interpreted and exploited by Duplantier  \cite {MR1666816} in the light of a quantum gravity approach to these 
exponents, see also \cite {MR2112128} for further references). It was then shown in \cite{MR1879850, MR1899232} that the explicit expression $\tilde \xi$ corresponds to the function 
$$ U(x) = (-1 + \sqrt { 24x + 1 }) / 4$$
(note that the relation between $U$ and $\tilde \xi$ determines $U$ up to a multiple constant; for future notational convenience, we choose to divide the function $x \mapsto -1 + \sqrt { 24 x +1}$ that  appeared in the 
introduction by this factor $4$). In particular, one gets that 
$\tilde \xi (\alpha, \beta, \gamma)$ is smaller than $1$ if and only if $U(\alpha_1) + \ldots + U(\alpha_n)$ is smaller than $U(1)=1$. 

It was further shown in \cite{MR1742883}, that the full-plane disconnection exponent $\xi ( \alpha_1, \ldots, \alpha_n)$ was in fact a function of the 
half-plane exponent $\tilde \xi ( \alpha_1, \ldots, \alpha_n)$, and the explicit expression for this function was determined 
in \cite{MR1879851}. In particular, it turns out that the former exponent $\xi (\alpha_1, \ldots, \alpha_n)$ is smaller than $2$ if and only the latter 
exponent $\tilde \xi$ is smaller than $5$ ie. if and only if $U(\alpha_1) + \ldots + U (\alpha_n) < U (5)= 5/2$.

\section{Characterization}\label{Characterization}

In the present section, we will adapt ideas used in \cite{MR1992830} and \cite{MR3293294} in order to derive Proposition \ref{prop:charac}:
Let us first define a family of deterministic Loewner curves in $\H$ that is stable under the mapping down procedure with respect to conformal maps that preserve the three marked points $-1,0,\infty$:
For all $x\in \R\setminus\{-1,0\}$ and all $z\in\overline\H$, let $g_{x,t}(z)$ be the solution to the initial value problem
\begin{align}\label{perfect equation}
\partial_t  g_{x,t}(z)=\frac{a_x(t)}{ g_{x,t}(z)-b_x(t)}, \quad g_0(z)=z
\end{align}
where
\begin{align*}
a_x(t)=\exp\left( - \frac{2t}{x^2+x} \right), \quad b_x(t)=\left(2x+1\right)\exp\left(-\frac{t}{x^2+x}\right)-x-1.
\end{align*}
The evolving hull generated by (\ref{perfect equation}) is a continuous curve that we will denote by $(\eta_x(t), t\ge 0)$ and $g_{x,t}$ is the conformal map from  $\H\setminus \eta_x([0,t])$ onto $\H$ such that $g_{x,t}(z)=z+a_x(t)/z+o(1/z)$. Here $a_x(t)>0$ is not constantly equal to $2$ as in (\ref{loewner}), which corresponds to a continuous and monotone reparametrization in time of the standard Loewner evolution.

Let $\varphi_{x,t}(\cdot):=( g_{x,t}(\cdot)- g_{x,t}(0))/( g_{x,t}(0)- g_{x,t}(-1))$ be the unique conformal map from $\H\setminus \eta_x([0,t])$ onto $\H$ that sends
$-1,0,\infty$ to themselves. The following lemma shows a nice semi-group property of the curves $\eta_x$.

\begin{lemma}
For all $s,t>0$,
\begin{align}\label{semi-group}
\varphi_{x,t+s}=\varphi_{x,t}\circ\varphi_{x,s}.
\end{align}
\end{lemma}

\begin{proof}
It is not difficult to verify that
\begin{align*}
&g_{x,t}(0)=(x+1) \left( \exp\left( -\frac{t}{x^2+x} \right)-1 \right)\\ 
&g_{x,t}(-1)=x  \exp\left( -\frac{t}{x^2+x} \right)-x-1
\end{align*}
are solutions of equation (\ref{perfect equation}). Then we differentiate $\varphi_{x,t}(z)$ using chain rules to get
\begin{align}\label{varphi}
\partial_t \varphi_{x,t} (z) =\frac{\varphi_{x,t}(z)}{x^2+x}+\frac{1}{x}+\frac{1}{\varphi_{x,t}(z)-x}.
\end{align}
Let $$f(x,z):=\frac{z}{x^2+x}+\frac{1}{x}+\frac{1}{z-x}$$
for all $z\in\overline\H$. Then we have $\partial_t \varphi_{x,t} (z)=f\left(x, \varphi_{x,t} (z) \right)$. The fact that $f$ does not depend on $t$ readily implies (\ref{semi-group}).
\end{proof}
Let us immediately compute $g_{x,t}'(-1)$ and $g_{x,t}'(0)$ that we will need later. Differentiating (\ref{perfect equation}) w.r.t. $z$ at $z=-1$ leads to
$$
\partial_t g_{x,t}'(-1)=-\frac{a_x(t)g_{x,t}'(-1)}{(g_{x,t}(-1)-b_x(t))^2}.
$$
Solving the equation with initial condition $g_0'(-1)=1$ leads to
\begin{align}\label{-1}
g_{x,t}'(-1)=\exp\left(-\frac{t}{(x+1)^2}\right).
\end{align}
In the same way, we have
\begin{align}\label{0}
g_{x,t}'(0)=\exp\left(-\frac{t}{x^2}\right).
\end{align}
We will also need
\begin{align}\label{0-1}
g_{x,t}(0)-g_{x,t}(-1)=\exp\left( -\frac{t}{x^2+x} \right).
\end{align}

\begin{lemma}
For all $x\in\R\setminus\{-1,0\}$, there exist $\lambda(x)\in\R$ such that,
\begin{align}\label{exp-lambda}
\P\left( K\cap \eta_x([0,t])=\emptyset \right)=\exp(-\lambda(x) t).
\end{align}
\end{lemma}
\begin{proof}
The semigroup property (\ref{semi-group}) of the curve $\eta_x$ and the conformal restriction property of $\P$ implies that
\begin{align*}
\P\left( K\cap \eta_x([0,t+s])=\emptyset \right)=\P\left( K\cap \eta_x([0,t])=\emptyset \right)\P\left( K\cap \eta_x([0,s])=\emptyset \right).
\end{align*}
The existence of $\lambda(x)$ follows. 
\end{proof} 

We will endow  $\mathcal{Q}^{\R \setminus\{-1,0\}}$ with the same topology $\mathcal{T}$ as in \cite{MR1992830} (which is also closely related to Carath\'eodory topology, see the comments in 
\cite {MR1992830}): Given $A\in \mathcal{Q}^{\R \setminus\{-1,0\}}$ and a sequence $\{A_n\}\subset \mathcal{Q}^{\R\setminus\{-1,0\}}$, we say that $A_n$ converges to $A$ if $\varphi_{A_n}$ converges to $\varphi_A$ uniformly on compact subsets of $\overline \H\setminus A$ and $\bigcup_n A_n$ is bounded away from $-1,0$ and $\infty$.
 
Then we have the following lemma which is an analogue of \cite[Lemma 3.5]{MR1992830}.

\begin{lemma}\label{topology}
On $\mathcal{Q}^{(0,\infty)}$ (resp. on $\mathcal{Q}^{(-1,0)}$, on $\mathcal{Q}^{(-\infty,-1)}$) with the topology $\mathcal{T}$, the semigroup generated by $\{\eta_x([0,t]), t\ge 0, x\in(0,\infty)\}$ (resp. by $\{\eta_x([0,t]), t\ge 0, x\in(-1,0)\}$, by $\{\eta_x([0,t]), t\ge 0, x\in(-\infty, -1)\}$) is dense, the application $A\mapsto \P(K\cap A=\emptyset)$ is continuous, and $A\mapsto \varphi_A'(0)$ is continuous.
\end{lemma}

We refer the readers to \cite[Lemma 3.5]{MR1992830} for a complete proof. The idea is that any simple Loewner curve can be approximated by a curve which has a driving function that takes piecewise the form of the driving function of $\eta_x$. Any hull $A$ can be approximated by a fattened and smoothed hull $A^\eps$ whose boundary is a simple curve.

\begin{lemma}\label{lem:continuous}
The function $\lambda$ is continuous. 
\end {lemma}

\begin{proof}
For a fixed $t>0$, as $x\to y$, the Hausdorff distance between $\eta_x([0,t])$ and $\eta_y([0,t])$ goes to $0$ , which implies the convergence of  $\eta_x([0,t])$ to $\eta_y([0,t])$ under the topology $\mathcal{T}$.
According to Lemma \ref{topology}, we have
$$ \left| \exp(-\lambda(x)t)-\exp(-\lambda(y)t)\right|= \left| \P(K\cap \eta_x([0,t])=\emptyset)-\P(K\cap \eta_y([0,t])=\emptyset) \right|\underset{x\to y}{\longrightarrow} 0.$$
Therefore  $|\lambda(x)-\lambda(y)|\to 0$ as $x\to y$.
\end{proof}

\begin{lemma}\label{lem:differentiable}
The function $\lambda$ is a.e. differentiable. 
\end{lemma}

\begin{proof}
First note that if $U\in\mathcal{Q}^{**}$, then for all $A\in\mathcal{Q}^{**}$ such that $A\subset U$, we have
\begin{align}\label{estimation}
\P(K\cap A)\le c(U) \cdot \capacity(A)
\end{align}
where $c(U)$ is a constant depending on $U$.
This is a consequence of the two facts:
for small $t$, $\P(K\cap \eta_x([0,t])\not=\emptyset)=\lambda(x)t+o(t^2)$ where $t$ is the capacity of $\eta_x([0,t])$; and
$A$ can be approximated by a curve which has a driving function piece-wisely taking the form of the driving function of $\eta_x$ for $x\in\varphi_U(\partial U)$.

Secondly, note that by conformal restriction property of $K$, we have
\begin{align*}
\P(K\cap \eta_x([0,1])=\emptyset, K\cap \eta_y([0,1])\not=\emptyset)=\P(K\cap \varphi_{x,1}(\eta_y([0,1]))\not=\emptyset)\P(K\cap \eta_x([0,1])=\emptyset).
\end{align*}
The set $A_{x,y}:=\varphi_{x,1}(\eta_y([0,1]))$ gets smaller as $x\to y$ in the following sense (for example see \cite[Proposition 3.82]{MR2129588}): 
\begin{align*}
\sup\{\Re(z_1-z_2); z_1, z_2\in A_{x,y}  \} \mbox{ is bounded };\quad \sup\{\Im(z); z\in A_{x,y}\}\le c |x-y|^{1/2}.
\end{align*}
Hence $\capacity(A_{x,y})\le c' |x-y|$. Then
by (\ref{estimation}), we have $\P(K\cap A_{x,y}\not=\emptyset)\le c'' \cdot \capacity(A_{x,y})=c_0 |x-y|.$

Therefore
$
\P(K\cap \eta_x([0,1])=\emptyset, K\cap \eta_y([0,1])\not=\emptyset)\le c_0' |x-y|
$
and consequently
\begin{align*}
&|\exp(-\lambda(x))-\exp(-\lambda(y))| \\
=&  \left| \P(K\cap \eta_x([0,1])=\emptyset)-\P(K\cap \eta_y([0,1])=\emptyset) \right| \\
\le& \P(K\cap \eta_x([0,1])=\emptyset, K\cap \eta_y([0,1])\not=\emptyset)+ \P(K\cap \eta_x([0,1])\not=\emptyset, K\cap \eta_y([0,1])=\emptyset)\\
\le& c_0''|x-y|.
\end{align*}
Therefore $e^{-\lambda}$ is a.e. differentiable hence $\lambda$ is a.e. differentiable.
\end{proof}
From (\ref{varphi}) we know that
\begin{align*}
\partial_t \varphi_{x,t} (z)|_{t=0}=f(x,z).
\end{align*}
Therefore as $t\to 0$, we have
\begin{align*}
\varphi_{x,t}(z)=z+f(x,z)t+o(t).
\end{align*}
By differentiating (\ref{varphi}), we get
\begin{align*}
\partial_t \varphi_{x,t}' (z)|_{t=0}=\partial_2 f(x,z)
\end{align*}
(here and in the following lines $\partial_2$ will mean the derivative with respect to the second variable). Therefore as $t\to 0$, we have
\begin{align*}
\varphi'_{x,t}(z)=1+\partial_2 f(x,z)t+o(t).
\end{align*}
Hence we have
\begin{align*}
\lim_{t\to 0}\frac{\varphi_{y,t}(x)-x}{t}=f(y,x), \quad  \lim_{t\to 0} \frac{\varphi_{y,t}'(x)-1}{t}=\partial_2 f(y,x).
\end{align*}
Now compute
\begin{align}
\label{expression1}
&\P\left(K\cap \eta_{x}([0,\delta s])\not=\emptyset, K\cap \eta_y([0,\delta t])\not=\emptyset \right)\\
\notag
=&\P(K\cap \eta_{x}([0,\delta s])\not=\emptyset)-\P(K\cap \eta_{x}([0,\delta s])\not=\emptyset, K\cap \eta_y([0,\delta t])=\emptyset)\\
\notag
=&1-\exp(-\delta s\lambda(x))-\exp(-\delta t\lambda(y)) \P(K\cap \eta_x([0,\delta s])\not=\emptyset | K\cap \eta_y([0,\delta t])=\emptyset)\\
\notag
=&1-\exp(-\delta s\lambda(x))-\exp(-\delta t\lambda(y)) \P(K\cap \varphi_{y,\delta t}(\eta_x([0,\delta s]))\not=\emptyset)
\end{align}
Note that $\varphi_{y,\delta t}(\eta_x([0,\delta s])$ is approximately $\eta_{x'}([0,\delta s'])$ where
$x'=\varphi_{y,t}(x)$ and $s'=\varphi_{y,t}'(x)^2 s$. The Hausdorff distance between these two sets is of order $O(\delta^2)$. Therefore, applying again \cite[Proposition 3.82]{MR2129588} (taking into account that the diameters of the two sets are of order $O(\delta)$) and arguing similarly as in Lemma \ref{lem:differentiable}, we have that
$$
\P(K\cap \varphi_{y,\delta t}(\eta_x([0,\delta s]))\not=\emptyset,\, K\cup \eta_{x'}([0,\delta s'])=\emptyset)=O(\delta^3).
$$
The line (\ref{expression1}) is then equal to
$$
1-\exp(-\delta s\lambda(x))-\exp\left(-\delta t\lambda(y)\right) \left[1-\exp\left(-\varphi_{ y,\delta t}'(x)^2\delta s\lambda\left(\varphi_{y,\delta t}(x)  \right)\right)\right]+o(\delta^2).$$
Note that $\varphi_{ y,\delta t}'(x)=1+\delta t \partial_2f(y,x)+o(\delta)$ and that $\lambda\left(\varphi_{y,\delta t}(x)  \right)=\lambda(x)+\delta t f(y,x) \lambda'(x)+o(\delta)$ at points $x$ where $\lambda$ is differentiable. The equality can be continued as
\begin{align*}
=&s\lambda(x)\delta-\frac{\lambda(x)^2}{2}s^2\delta^2-(1-\delta t\lambda(y))\left(\varphi_{y,\delta t}'(x)^2\delta s\lambda(\varphi_{y,\delta t}(x))-\frac{\lambda(x)^2}{2}s^2\delta^2\right)+o(\delta^2)\\
=&s\lambda(x)\delta-\frac{\lambda(x)^2}{2}s^2\delta^2-\varphi_{y,\delta t}'(x)^2\delta s\lambda(\varphi_{y,\delta t}(x))+\frac{\lambda(x)^2}{2}s^2\delta^2-\delta t\lambda(y)\varphi_{y,\delta t}'(x)^2\delta s\lambda(\varphi_{y,\delta t}(x))+o(\delta^2)\\
=&s\lambda(x)\delta-\frac{\lambda(x)^2}{2}s^2\delta^2-\delta s\lambda(x) -\delta^2 st f(y,x)\lambda'(x) -\delta^2 st \lambda(x) 2\partial_2 f(y,x)\\
&+\frac{\lambda(x)^2}{2}s^2\delta^2-\delta^2 st\lambda(y)\lambda(x)+o(\delta^2)\\
=&-\delta^2 st \left(f(y,x)\lambda'(x)+2\lambda(x)\partial_2 f(y,x)+\lambda(x)\lambda(y)\right) +o(\delta^2).
\end{align*}
Note that we can do the same computation if we exchange the roles of $(x,s)$ and $(y,t)$, therefore we must have the following \emph{commutation relation} for all points $x,y$ at which $\lambda$ is differentiable:
\begin{align}\label{x-y}
f(y,x)\lambda'(x)+2\lambda(x)\partial_2 f(y,x)=f(x,y)\lambda'(y)+2\lambda(y)\partial_2 f(x,y).
\end{align}
Knowing that $f$ is $C^\infty$ and that $\lambda$ is continuous, this implies that $\lambda$ is also $C^\infty$ on its whole domaine of definition $\R\setminus\{-1,0\}$.

We now fix $x\in(0,\infty)$ and let $y\to x$ in (\ref{x-y}). Comparison of the coefficients of Taylor development of the $O(y-x)$ term on both sides yields
\begin{align*}
&x^2(x+1)^2\lambda'''(x)+6x(2x+1)(x+1)\lambda''(x)\\
&+6\left( 3x(x+1)+x(2x+1)+(x+1)^2 \right) \lambda'(x)
+12(2x+1)\lambda(x)=0.
\end{align*}
This is equivalent to $$\left( x^2(x+1)^2 \lambda(x) \right)'''=0.$$
Therefore, there exist three constants $a,b,c\in\R$ such that for all $x\in(0,\infty)$,
\begin{align*}
x^2(x+1)^2 \lambda(x)=ax^2+bx+c.
\end{align*}
Hence for $\alpha=a, \beta=a-b+c, \gamma=c$ we have that for all $x\in(0,\infty)$
\begin{align}\label{aaa}
\lambda(x)=\frac{ax^2+bx+c}{x^2(x+1)^2 } = \frac{\beta}{(x+1)^2}+\frac{\gamma}{x^2}+\frac{\alpha-\beta-\gamma}{x(x+1)}.
\end{align}
Putting (\ref{aaa}) back into (\ref{exp-lambda}) and knowing (\ref{-1}), (\ref{0}), (\ref{0-1}), we get
\begin{align}\label{kkk}
\P(K\cap \eta_x([0,t])=g_t'(-1)^\beta g_t'(0)^\gamma (g_t(0)-g_t(-1))^{\alpha-\beta-\gamma}.
\end{align}

In the same way as (\ref{aaa}) we can get that for $x\in(-\infty, -1)$, there exist $\alpha_1,\beta_1,\gamma_1\in\R$ such that
\begin{align*}
\lambda(x)= \frac{\beta_1}{(x+1)^2}+\frac{\gamma_1}{x^2}+\frac{\alpha_1-\beta_1-\gamma_1}{x(x+1)};
\end{align*}
and for $x\in(-1,0)$, there exist $\alpha_2,\beta_2,\gamma_2$ such that
\begin{align*}
\lambda(x)= \frac{\beta_2}{(x+1)^2}+\frac{\gamma_2}{x^2}+\frac{\alpha_2-\beta_2-\gamma_2}{x(x+1)}.
\end{align*}
\begin{lemma}\label{lem:123}
For a same $K$, we have $\alpha=\alpha_1=\alpha_2$, $\beta=\beta_1=\beta_2$, $\gamma=\gamma_1=\gamma_2$.
\end{lemma}
\begin{proof}
Using the commutation relation (\ref{x-y}) for $x\in(-\infty,-1)$ and $y\in(-1,0)$ yields
\begin{align*}
&\left[ -\frac{2\alpha_2}{y^3}-\frac{2\beta_2}{(y-1)^3}-\frac{\gamma_2(2y-1)}{y^2(y-1)^2} \right] \frac{y(1-y)}{x(y-x)(1-x)}\\
&+2\left[ \frac{\alpha_2}{y^2}+\frac{\beta_2}{(y-1)^2}+\frac{\gamma_2}{y(y-1)} \right] \frac{-y^2+2xy-x}{x(y-x)^2(1-x)}\\
=&\left[ -\frac{2\alpha_1}{x^3}-\frac{2\beta_1}{(x-1)^3}-\frac{\gamma_1(2x-1)}{x^2(x-1)^2} \right] \frac{x(1-x)}{y(x-y)(1-y)}\\
&+2\left[ \frac{\alpha_1}{x^2}+\frac{\beta_1}{(x-1)^2}+\frac{\gamma_1}{x(x-1)} \right] \frac{-x^2+2xy-y}{y(x-y)^2(1-y)}
\end{align*}
By simplifying it, we get for all for $x\in(-\infty,-1)$ and $y\in(-1,0)$
\begin{align*}
2(\alpha_2-\alpha_1)(1-x)(1-y)+2(\beta_2-\beta_1)xy+(\gamma_2-\gamma_1)(-y+2xy-x)=0.
\end{align*}
Hence $\alpha_1=\alpha_2, \beta_1=\beta_2, \gamma_1=\gamma_2$. In the same way we can identify $\alpha,\beta,\gamma$ with $\alpha_1,\beta_1,\gamma_1$. 
\end{proof}
Therefore we get the formula (\ref{kkk}) for all $x\in\R\setminus\{-1,0\}$.
This, together with Lemma \ref{topology} and the restriction property of $K$, proves Proposition \ref{prop:charac}.

\section{Hypergeometric SLE}\label{sec:hsle}

In this section, we will define the variants of SLE$_{8/3}$ that will turn out to describe
boundaries of trichordal restriction samples. The driving functions of these SLE are Brownian motions with drift terms that involve hypergeometric functions (the fact that this type of 
SLE drift terms arise for multiple SLEs was observed in the context of commutation relations and/or CFT partition function considerations, see \cite{MR2187598}). In the present paper, they will form a family of curves that depend on three parameters $\lambda,\mu,\nu$ and we will call them hSLE$_{8/3}(\lambda,\mu,\nu)$ (this parameter $\lambda$ is unrelated to the function $\lambda$ used in the previous function -- we will similarly use $\mu$, $\nu$, $a$ etc. for different purposes, but it will be
each time clear from the context, which one we will be talking about). 

Let us give a brief heuristic. Let $\gamma$ be the right boundary of a restriction sample in $\H$ which intersects the points $0,-1,\infty$. Parametrize $\gamma$ with Loewner's differential equation, i.e. map $\H\setminus \gamma(0,t)$ onto $\H$ with the normalized maps $g_t$. 
It is possible to prove (although we do not rely on such proof), in a similar way as we shall prove Lemma \ref{lem:J} later on, that
conditioned on $g_t(\gamma(t))$, the curve $g_t(\gamma[t,\infty))$ is independent of $\gamma(0,t)$ and moreover it
satisfies a kind of four-point restriction with marked points
$g_t(-1), g_t(0), g_t(\gamma_x(t))$ and $\infty$.  This domain-Markov property suggests
that there should exist a function $J$ such that the driving function of $\gamma$ is given by
\begin{align}\label{1st}
dW_t=\sqrt{\kappa}\dif B_t+J(W_t-g_t(0),W_t-g_t(-1))dt.
\end{align}

In this section, we will introduce some preliminaries about hypergeometric functions, and then give the expression of $J$ that we use to define our hypergeometric SLEs.  In \S\,\ref{s5}, we will then
use these processes and show that they satisfy restriction properties. 

Even though we will define $J$ without further explanations, the form of $J$ can be heuristically understood and guessed by the fact that, in the case $\kappa=8/3$, the hSLE$_{8/3}$ defined by (\ref{1st}) and $J$ should eventually be viewed as an ordinary SLE$_{8/3}$ conditioned to avoid a trichordal restriction sample $K$ to its left (such that $K$ intersects $-1,0,\infty$ -- note that at this point, the trichordal restriction measures have not yet been constructed).  This is analogous to the interpretation of SLE$_{8/3}(\rho)$ curves as ordinary SLE$_{8/3}$ conditioned to avoid a chordal restriction sample, see \cite{MR2060031}.

\subsection{Some parameters} \label{hyp fun}
First of all we introduce some parameters that are to be used until the end of this paper, even if we will not recall their definition each time.
As mentioned earlier, the family of  hypergeometric SLEs will depend on three parameters $\lambda,\mu,\nu$.
We also introduce three triplets of parameters $(a,b,c), (l,m,n)$ and $(\alpha,\beta,\gamma)$, each depending on $\lambda,\mu,\nu$.

The first triplet of parameters is given by
\begin{align}\label{abclmn}
&\a=\lambda+\mu+\nu+2, \quad \b=\mu+\nu-\lambda, \quad \c=2\nu+3/2.
\end{align}
They will be used to define the hypergeometric functions. 

Recall the definition of the function $U$ mentioned in \S\,\ref{sec:brownian-exponents}, which is related to the Brownian disconnection exponent $\tilde \xi$: 
$$
U(x)=\left(-1+\sqrt{24x+1}\right)/4.
$$
It can be viewed as a bijection from $[-1/24, \infty)$ into $[-1/4, \infty)$ (or as a bijection from $\R_+$ onto itself), and its inverse function is defined on $[-1/4, \infty)$ by
$$ U^{-1}(x)=x(2x+1)/3 .$$
We will also sometimes use $U^{-1}(x)$ as a shorthand notation for this polynomial when $x < -1/4$. 

The second triplet of parameters is defined to be
\begin{align}\label{align:lmn}
l=U^{-1}(\lambda),\quad m=U^{-1}(\mu), \quad n=U^{-1}(\nu).
\end{align}
The third triplet of parameters is defined to be
\begin{align}\label{alphabetagamma}
\alpha=U^{-1}\left(\lambda+\frac34\right), \quad
\beta=U^{-1}(\mu), \quad
\gamma=U^{-1}\left(\nu+\frac34\right)
\end{align}
which will turn out to be the restriction exponents associated to the corresponding hSLE$_{8/3}(\lambda,\mu,\nu)$.

From now on until the end of the paper, we also restrict $\lambda,\mu,\nu$ to stay in the range
\begin{align}\label{align:cond}
\lambda>-3/4, \quad\mu\ge-1/4, \quad\nu>-3/4,\quad \mu<\lambda+\nu+3/2.
\end{align}
Note that (\ref{align:cond}) makes $U^{-1}$ in (\ref{alphabetagamma}) a bijection, which allows us to write
\begin{align}\label{align:bijection}
\lambda=U(\alpha)-\frac34,\quad \mu=U(\beta), \quad \nu=U(\gamma)-\frac34.
\end{align}
The condition $\mu<\lambda+\nu+3/2$ in (\ref{align:cond}) is therefore equivalent to  $U(\beta)<U(\alpha)+U(\gamma)$, hence is also equivalent to $\beta<\tilde\xi(\alpha,\gamma)$.
Note that the $U^{-1}$ in (\ref{align:lmn}) is not necessarily a bijection.

Now let us summarize all the conditions that we imposed. The range (\ref{align:cond}) can be easily translated to equivalent ranges on $(a,b,c)$ or $(\alpha,\beta,\gamma)$. The equivalence with the $(\alpha,\beta,\gamma)$ range is to be understood in the sense that we assume above all $\lambda>-3/4, \mu\ge-1/4, \nu>-3/4$ (in this case (\ref{align:bijection}) is true).
\[
\left|\,
\begin{array}{lllllll}
\displaystyle{\lambda> -\frac34 } &\iff&  \displaystyle{a-b>\frac12 } &\iff&  \alpha>0\\[2.5mm]
\displaystyle{\mu\ge-\frac14}  &\iff&  a+b\ge c   &\iff& \displaystyle{\beta\ge-\frac1{24}}\\[2.8mm]
\displaystyle{\nu>-\frac34}  &\iff &  c>0  &\iff& \gamma>0\\[2mm]
\displaystyle{\mu<\lambda+\nu+\frac32 } &\iff& b<c  &\iff&   \beta<\tilde\xi(\alpha,\gamma).
\end{array}
\label{first-condition}\tag{Range 1}
\right.
\]
The (\ref{first-condition}) is chosen to make the process and other related functions well defined. Lemma \ref{G}, Lemma \ref{lem:nu} and Lemma \ref{lem:lambda} will explain why all
these conditions are indeed needed.

\subsection{Preliminaries on the hypergeometric and some related functions}
Hypergeometric functions are defined for $\a,\b \in \R, \c \in \R\setminus \N_-$ and for $|z|\le 1$ by
\begin{align*}
_2F_1(\a,\b;\c;z)=\sum_{n=0}^\infty \frac{(\a)_n(\b)_n}{(\c)_n}\frac{z^n}{n!}
\end{align*}
where $(x)_n$ is the Pochhammer symbol defined by
$(x)_0=1$   and  $(x)_n=x(x+1)\cdots(x+n-1)$ if $ n>0$.
The hypergeometric function $_2F_1(\a,\b;\c;\cdot)$ is a particular solution of Euler's hypergeometric differential equation
\begin{equation}\label{euler}
z(1-z)\frac{\dif^2 u}{\dif z^2} + (\c-(\a+\b+1)z)\frac{\dif u}{\dif z}-\a\b u=0.
\end{equation}
It is known that $_2F_1(\a,\b;\c;\cdot)$ can be analytically extended to a function  $w(\a,\b,\c;\cdot)$ defined on the larger domain  $\C\setminus[1,\infty)$, which is a maximal solution of (\ref{euler})  (see \cite{MR1225604}). 
This function $w$ restricted to the domain of definition $(-\infty,0]$ that will play in important role in the be sequel. Note that $_2F_1$ is real on $(-1,1)$, hence $w$ is real on $(-\infty,0]$.
From now on, except when stated otherwise, $w(z)$ will stand for $w(a,b;c;z)$.

\begin{lemma}
For $a,b,c,$ in (\ref{first-condition}), let $e:=\max(b-a, -1)$. Then $e<0$. There exist some constant $c_{0}\not=0, c_1$ such that as $z\to-\infty$
\begin{align}\label{expansion}
w(z)  = c_{0} (-z)^{-b} + c_1 (-z)^{e-b} + o ((-z)^{e-b}).
\end{align}
Therefore we also have as $z\to-\infty$
\begin{align}
\label{lim2}
&zw'(z)/w(z)= -\b+ O((-z)^e).
\end{align}
\end{lemma}
\begin{proof}
It can be found in \cite{MR1225604} that when $a,b,c,c-a,c-b$ are not negative integers and when $a-b$ is not an integer, we have
\begin{equation}
\begin{split}
w(z)=&\frac{\Gamma(\c)\Gamma(\b-\a)}{\Gamma(\b)\Gamma(\c-\a)}(-z)^{-\a}{_2F_1(\a,1-\c+\a;1-\b+\a; 1/z)}\\
&+\frac{\Gamma(\c)\Gamma(\a-\b)}{\Gamma(\a)\Gamma(\c-\b)}(-z)^{-\b}{_2F_1(\b,1-\c+\b;1-\a+\b; 1/z)}.
\end{split}
\end{equation}
Note that $_2F_1(0)=1$ and (\ref{first-condition}) implies $\a>\b$.  The lemma is obviously true in this case.

The (\ref{first-condition}) ensures that $a,c$ and $c-b$ are not negative integers. If $b$ is a negative integer, then $w$ becomes elementary: A polynomial with leading term (note that $a>b$)
\begin{align*}
 \frac{(\a)_{-b}(\b)_{-b}}{(\c)_{-b}}\frac{z^{-b}}{(-b)!}.
\end{align*}
The lemma is also true in this case.

If $c-a$ is a negative integers, then we can apply the identity (see  \cite{MR1225604})
\begin{align*}
w(a,b,c;z)=(1-z)^{c-a-b} w(c-a,c-b,c;z).
\end{align*}
In this case $w(c-a,c-b,c;z)$ is a polynomial in $z$ of order $a-c$, hence $w(z)$ is a polynomial in $z$ of order $-b$. Thus the lemma is true in this case.

If $\a-\b$ is an integer, explicit formula of the expansion of $w(z)$ at $\infty$ can be found in \cite{MR1225604} such that  (\ref{lim2}) is still true. 
\end{proof}

\begin{lemma}\label{G}
For $a,b,c$ in (\ref{first-condition}),
$\,w(\a,\b,\c;z)>0$ for all $z\in(-\infty,0)$.
\end{lemma}

\begin{proof}

We have the identity (see \cite{MR1225604})
$$
w(\a,\b,\c;x)=(1-x)^{-\a}{w}\left(\a,\c-\b,\c;\frac{x}{x-1}\right)\quad \mbox{for} \quad x\in(-\infty,0).
$$
Therefore it is equivalent to show that $w\left(\a,\c-\b,\c;\cdot\right)>0$ on $(0,1)$.
Since $w(\a,\c-\b,\c;0)=1$, it is sufficient to show that ${w}(\a,\c-\b,\c;\cdot)$ is increasing on $(0,1)$. 
Note that ${w'}(\a,\c-\b,\c;0)={a(c-b)}/{c}>0$ because (\ref{first-condition}) implies $a\ge c-b>0, c>0$.
If ${w}(\a,\c-\b,\c;z_0)$ is not increasing on $(0,1)$, then there must be a $z_0\in(0,1)$ which corresponds to the first point at which ${w}(\a,\c-\b,\c;\cdot)$ stops being increasing. Consequently, ${w}(\a,\c-\b,\c;z_0)>0$ and ${w}(\a,\c-\b,\c;z_0)$ is at its local maximum.

Note that ${w}(\a,\c-\b,\c;\cdot)$ satisfies the Euler's differential equation
\begin{equation}\label{w(a,c-b,c)}
z(1-z)\frac{\dif^2 w}{\dif z^2} + (c-(a+(c-b)+1)z)\frac{\dif w}{\dif z}-a(c-b) w=0.
\end{equation}
If we evaluate (\ref{w(a,c-b,c)}) at $z_0$, knowing that $w'(a,c-b;c; z_0)=0$, we have
\begin{align*}
z_0(1-z_0)w''(a,c-b;c; z_0)=a(c-b)w(\a,\c-\b;\c;z_0).
\end{align*}
The left-hand side is negative  but the right-hand side is strictly positive, which leads to a contradiction.
Hence ${w}(\a,\c-\b,\c;\cdot)$ is increasing on $(0,1)$ and the lemma follows.
\end{proof}

\subsubsection{The functions $G,Q,J$}
Now we define the related functions $G, Q, J$ and make some preparatory estimations. Later on,
 the function $J$  will be used to define the hSLE via the formula (\ref{1st}).
The function $G$ will be used in the definition of the martingale that will be related to the hSLE. The function $Q$  has an auxiliary role and will  be referred to in some computations.

For $\a,\b,\c$ within (\ref{first-condition}), define for all $z\in [0,\infty)$
\begin{align*}
&G(\a,\b;\c;z):=z^{\nu}(1+z)^{\mu} w(\a,\b;\c;-z).
\end{align*}
It is worth noting that $G$ is the solution of the following differential equation
\begin{align}\label{Gdiff}
\frac{\kappa}{2}G''(z)z(z+1)+2G'(z)(2z+1)-G(z)\left(2m\frac{z}{z+1}+2n\frac{z+1}{z}+2(l+\lambda-m-n)\right)=0.
\end{align}
It is immediate that
\begin{align}\label{z-to-0}
G(\a,\b;\c;z) =z^{\nu}(1+o(1))\quad\mbox{ as } z\to 0.
\end{align}
By (\ref{expansion}), we know that there exist a constant $c_0>0$ such that
\begin{align}\label{z-to-infty}
G(\a,\b;\c;z) =c_0\,z^{\lambda}(1+o(1))\quad\mbox{ as } z\to \infty.
\end{align}
Lemma \ref{G} ensures that $G$ is non zero on $(0,\infty)$.
Hence we can define
\begin{align*}
Q(\a,\b;\c;z):=zG'(\a,\b;\c;z)/G(\a,\b;\c;z) \quad\mbox{ for } z\in(0,\infty).
\end{align*}
According to earlier estimates, $Q(z)$ has a limit as $z$ goes to $0$ or $\infty$. 
\begin{align}
\label{q1}
&Q(z) =\nu+O(z) \mbox{ as } z\to 0;\\
\label{q2}
&Q(z) = \lambda+O(1/z) \mbox{ as } z\to \infty.
\end{align}
Always for $a,b,c$ under condition (\ref{first-condition}), for $p,q\in\R$ such that $0\le p\le q$ and $(p,q)\not=(0,0)$, define
\begin{align}\label{def:J}
J(a,b,;c; p,q):= \frac{\kappa}{p}Q\left(\a,\b;\c;\frac{p}{q-p}\right).
\end{align}

\subsubsection{Degenerate case}\label{degenerated case}
In (\ref{first-condition}), if we allow $\mu=\nu+\lambda+3/2$, then we have equivalently $b=c$.
Then  $$w(z)=(1-z)^{-\a}$$ and (\ref{lim2}) no longer holds. Instead we have
\begin{align}\label{lim3}
zw'(z)/w(z)\underset{z\to\infty}\longrightarrow-\a.
\end{align}
The functions $G, Q, J$ become
\begin{align}
\notag
&G(\a,\b;\c;z)=z^\nu(1+z)^{\mu-\a};\\
\label{Q}
&Q(\a,\b;\c;z)=\nu+(\mu-\a)\frac{z}{1+z};\\
\label{J}
&J_{\lambda,\mu,\nu}(p,q)=\frac{\kappa\nu}{p}+\frac{\kappa(\mu-\a)}{q}.
\end{align}

\subsection{Hypergeometric SLE}
In the present section, we are first going to define the hypergeometric SLE processes by giving the explicit form of their driving functions. We will then interpret these processes as weighted SLE$_\kappa(\rho)$ processes up to certain stopping times, by performing a Girsanov transformation (see \cite{MR2060031}). An alternative and equivalent way is to see these processes as SLEs weighted by partition functions as Lawler has described in \cite{MR2518970} (although we will not use this terminology in the following).
This shows that the hSLEs are indeed well defined and for $\kappa\le 4$ they are continuous simple curves. 
In order to see whether the hSLE touches certain parts of the boundary, we will do a coordinate change of the type Schramm and Wilson did for SLE$_\kappa(\rho)$ in \cite{MR2188260}. 

\subsubsection {Definition}

We define the hypergeometric SLE process to be the Loewner chain generated by the driving function $(W_t, t\ge 0)$, such that $(W_t, O_t, V_t)$ is the solution of the following system of stochastic differential equations
\begin{equation}\label{hsle}
\left\{
\begin{split}
&\; dO_t=\frac{2dt}{O_t-W_t},\\
&\; dV_t=\frac{2dt}{V_t-W_t},\\
&\; dW_t=\sqrt{\kappa}\dif B_t+J(W_t-O_t,W_t-V_t)dt, \\
&\; O_0=0,\, V_0=-1,\, W_0=x\ge 0,
\end{split}
\right.
\end{equation}
where $(B_t ,t\ge 0)$ is a standard Brownian motion.
We denote such a process hSLE$_\kappa(\lambda,\mu,\nu)$ where $\lambda,\mu,\nu$ are the parameters associated to the function $J$.
The solution to (\ref{hsle}) is clearly well defined at times when $W_t-O_t$ and $W_t-V_t$ are not $0$. In the next two sections, we will show that for $\lambda,\mu,\nu$ in (\ref{first-condition}), (\ref{hsle}) can be well defined even when $W_t-O_t$ or $W_t-V_t$ hits $0$, by showing that the law of the hSLE is  locally absolutely continuous w.r.t. that of an SLE$_\kappa(\rho)$ process. 

Recall that the  SLE$_\kappa(\rho)$ processes have been defined  in \cite{MR1992830} for all $\rho>-2$ and it has been proved there that they hit the boundary if and only if $\rho\in(-2,\kappa/2-2)$. For $\kappa\le 4$, when $\rho\ge \kappa/2-2$,  it is proved in \cite{MR2060031} that SLE$_\kappa(\rho)$s are absolutely continuous w.r.t. SLE$_\kappa$ processes, hence are continuous curves (see \cite{MR2153402}). Later, Miller and Sheffield \cite{MR3477777} have proved using imaginary geometry that SLE$_\kappa(\rho)$s are in fact continuous curves for all $\rho>-2$.

\subsubsection{Interpretation as weighted SLE$_\kappa(\rho)$ process}\label{sec:girsanov}
Let $(B_t, t\ge 0)$ be a standard Brownian motion for the probability measure $\P$ and the filtration $(\mathcal{F}_t, t\ge 0)$.
Let $(O_t, W_t)$ be the solution of the initial value problem
\[
\begin{dcases}
\; dO_t=\frac{2dt}{O_t-W_t},\\
\; dW_t=\sqrt{\kappa}\dif B_t+\frac{\rho dt}{W_t-O_t}, \\
\; O_0=0, \, W_0=x\ge 0.
\end{dcases}
\]
The Loewner curve $(\gamma_t, t\ge 0)$ generated by the equation
\begin{align*}
d g_t(z) =\frac{2dt}{g_t(z)-W_t}
\end{align*}
is a SLE$_\kappa(\rho)$ process.
Let $V_t:=g_t(-1)$. Let $$T_M:=\inf\left\{t\ge 0;  (W_s-O_s)/(O_s-V_s)\ge M \right\}.$$
Define the following $(\mathcal{F}_t)$ adapted local martingale
\begin{align*}
X_t:= \frac1{\sqrt{\kappa}} \int_0^t \mathbf{1}_{s< T_M} \left( J(W_s-O_s,W_s-V_s)-\frac{\rho}{W_s-O_s}\right) dB_s.
\end{align*}
Note that the function $J$ has a pole when $W_s-O_s=0$.
However when $\rho$ is chosen to be $\kappa\nu$, $J(W_s-O_s,W_s-V_s)-{\rho}/({W_s-O_s})$ remains bounded, because:
\begin{itemize}
\item[(i)] By definition
\begin{align*}
J(W_s-O_s,W_s-V_s)=\frac{\kappa}{W_s-O_s} Q\left( \frac{W_s-O_s}{O_s-V_s} \right).
\end{align*}
\item[(ii)]
By (\ref{q1}), $Q(z)=\nu+O(z)$ as $z\to 0$.
\item[(iii)]
For any $s<T_M$,  we have $(W_s-O_s)/(O_s-V_s)\le M$.
\end{itemize}
In this case,  we can then define the measure
$\mathbf{Q}_{T_M}$ on the $\sigma$-field  $\mathcal{F}_{T_M}$ by the following Girsanov transformation
\begin{align*}
\frac{d\mathbf{Q}_{T_M}}{d\P}: =\mathcal{E}(X)_{T_M}:=\exp\left( X_{T_M} -\frac{1}{2\kappa} \int_0^{T_M} \left( J(W_s-O_s,W_s-V_s)-\frac{\rho}{W_s-O_s}\right)^2 ds \right).
\end{align*}
The measure $\mathbf{Q}_{T_M}$ is absolutely continuous w.r.t. $\P_{T_M}$. Moreover, under $\mathbf{Q}_{T_M}$, the process 
\begin{align*}
\tilde B_t:= B_t-\frac{1}{\sqrt{\kappa}}\int_0^t \left( J(W_t-O_t,W_t-V_t)-\frac{\rho}{W_t-O_t}\right) ds
\end{align*}
is a standard Brownian motion up to time $t= T_M$. Therefore $(O_t, V_t, W_t; t<T_M)$ satisfies the differential equations
\[
\begin{dcases}
\; dO_t=\frac{2dt}{O_t-W_t},\\
\; dV_t=\frac{2dt}{V_t-W_t},\\
\; dW_t=\sqrt{\kappa}\dif \tilde B_t+J(W_t-O_t,W_t-V_t)dt, \\
\; O_0=0,\, V_0=-1,\, W_0=x\ge 0.
\end{dcases}
\]
Hence, under $\mathbf{Q}_{T_M}$, $\gamma$ is a hSLE$_\kappa(\lambda,\mu,\nu)$ up to time $T_M$.

This proves that the law of a hSLE$_\kappa(\lambda,\mu,\nu)$ is absolutely continuous w.r.t. the law of a SLE$_\kappa(\kappa\nu)$ process up to the stopping time $T_M$. For $\kappa<4$, a hSLE$_\kappa(\lambda,\mu,\nu)$ is therefore almost surely a simple curve up to $ T_M$.
Note that the measures $\mathbf{Q}_{T_M}$ are compatible for all $M>0$ hence we can let $M\to \infty$ and say that the hSLE$_\kappa(\lambda,\mu,\nu)$ is a simple curve up to the time
\begin{align}\label{ttt}
T=\lim_{M\to\infty} T_M.
\end{align}
However the absolute continuity of $\mathbf{Q}_{T_M}$ w.r.t. $\P$ is not preserved as $M\to\infty$ since $\mathcal{E}(X)_{T_M}$ might be no longer bounded away from $0$ or $\infty$.

Let us first mention that whether a hSLE curve touches the boundary in a neighborhood of the starting point (i.e. for $t<T_M$) can be deduced from the knowledge on SLE$_\kappa(\rho)$ processes, due to absolute continuity. 
We state this property only in the $\kappa=8/3, x=0$ case because it is all we need.
\begin{lemma}\label{lem:nu}
For $\lambda,\mu,\nu$ in (\ref{first-condition}), let $\gamma$ be a hSLE$_{8/3}(\lambda,\mu,\nu)$ curve as defined in (\ref{hsle}) which starts at $0$ and has $-1,0^-$ as boundary marked points. Let $T$ be a stopping time defined as (\ref{ttt}).
\begin{itemize}
\item 
If $\nu\ge -1/4$ , then $\gamma([0,T]) \cap (-1,0]=\emptyset$ a.s.
\item 
If $-3/4<\nu<-1/4$, then $\gamma([0,T]) \cap (-1,0]\not=\emptyset$ a.s.
\end{itemize}
\end{lemma}
\begin{remark}
 The condition $\nu>-3/4$ in (\ref{first-condition}) is due to the $\rho=-2$ threshold for SLE$_\kappa(\rho)$.
\end{remark}
Now we need to determine whether $T=\infty$. Let us observe that $T=\lim_{M\to\infty} T_M$ in fact corresponds to the first time $\gamma$ hits $(-\infty,-1]$ on the real line; and  $T=\infty$ if and only if $\gamma$ does not touch $(-\infty,-1]$.
More precisely, by knowledge on deterministic conformal maps, it is not difficult to see that as $M\to \infty$, 
\begin{itemize}
\item[(i)] if $\gamma_{T_M}\to\infty$, then $T=\infty$ and $(W_t-O_t)/(O_t-V_t)\to \infty$ as $t\to T$;
\item[(ii)] if $\gamma_{T_M}$ converges to a point on $(-\infty,-1)$, then $T< \infty$ and $(W_t-O_t)/(O_t-V_t)\to \infty$ as $t\to T$;
\item[(iii)] if $\gamma_{T_M}$ converges to the point $-1$, then $T< \infty$ and $(W_t-O_t)/(O_t-V_t)$ oscillates between $0$ and $\infty$ as $t\to T$.
\end{itemize}
The oscillation in case (iii) is due to the fact that $\gamma$ approaches $-1$ in a fractal way.
In order to determine in which case we are in, we will zoom into the neighborhood of the point $-1$, by doing a coordinate (and time) change of the type that Schramm and Wilson did for SLE$_\kappa(\rho)$s in \cite{MR2188260}. 

\subsubsection{Coordinate change}

Let $(\gamma_t, 0\le t <T)$ be a hSLE$_\kappa(\lambda,\mu,\nu)$ as previously defined. Let $f:z\mapsto {z}/({z+1})$ be a M\"obius transformation from $\H$ to itself, that sends the points $-1,0,\infty$ respectively to $\infty,0,1$. 
Let $f_t$ be a M\"obius transformation from $\H$ to itself such that $f_t\circ g_t\circ f^{-1}$ is normalized at infinity in a way that there exist $s(t)\in\R$,
\begin{align}\label{normalized gs}
\tilde g_s:= f_t\circ g_t\circ f^{-1}(y)=y+\frac{2s(t)}{y}+O\left(\frac{1}{y^2}\right) \mbox{ as } y\to\infty.
\end{align}
Then the image of $(\gamma_t, 0\le t <T)$ by $f$ 
reparametrized by $s(t)$ is a Loewner chain with driving function $\tilde W_{s(t)}:=f_t(W_t)$: For all $y\in\H$,
\begin{align*}
\partial_s \tilde g_s(y)=\frac{2}{\tilde g_s(y)-\tilde W_s}, \quad \tilde g_0(y)=y.
\end{align*}
The above-mentioned conformal maps are illustrated in Figure \ref{fig:coordinate--change}.

 \begin{figure}[h!]
    \centering
    \includegraphics[width=0.86\textwidth]{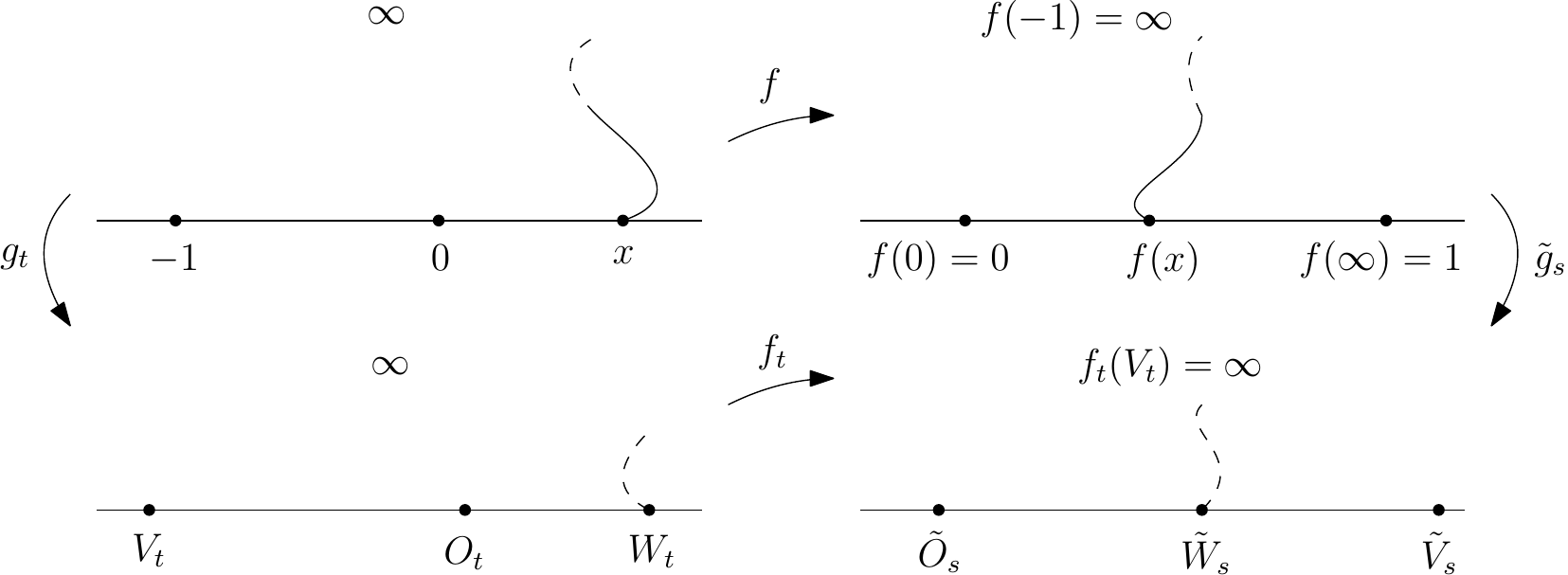}
    \caption{ Commutative diagram for coordinate change.}
    \label{fig:coordinate--change}
\end{figure}

Since $f_t$ sends $V_t$ to $\infty$, there exist $a_t,b_t\in\R$ such that
\begin{align*}
f_t(z)=a_t+\frac{b_t}{z-g_t(-1)}.
\end{align*}
We expand $\tilde g_s(y)= f_t\circ g_t\circ f^{-1}(y)$ at $y=\infty$,
\begin{align*}
\tilde g_s(y)
=&a_t-\frac{b_t}{g_t'(-1)}y+\frac{b_t}{g_t'(-1)}\left(1-\frac{g_t''(-1)}{2g_t'(-1)}\right)+O\left(\frac1y\right).
\end{align*}
In order that $\tilde g_s$ be normalized as (\ref{normalized gs}) , we should have
\begin{align*}
b_t=-g_t'(-1), \quad a_t=1-\frac{g_t''(-1)}{2g_t'(-1)}.
\end{align*}
Moreover, the time change should be
\begin{align}\label{align:time-change}
s'(t)=f_t'(W_t)^2=\frac{g_t'(-1)^2}{(W_t-g_t(-1))^4}.
\end{align}
Note that
\begin{align*}
\tilde W_s=f_t(W_t)=a_t+\frac{b_t}{W_t-g_t(-1)}=1-\frac{g_t''(-1)}{2g_t'(-1)}-\frac{g_t'(-1)}{W_t-g_t(-1)}.
\end{align*}
In order to get $d\tilde W_s$, let us calculate
\begin{align*}
dg_t'(-1)=-\frac{2g_t'(-1)}{(g_t(-1)-W_t)^2}dt,\quad dg_t''(-1)=-\frac{2g_t''(-1)}{(g_t(-1)-W_t)^2}dt+\frac{4g_t'(-1)^2}{(g_t(-1)-W_t)^3}.
\end{align*}
Applying the It\^o calculus rule, we have
\begin{align*}
d\tilde W_s
=&(\kappa-6)\frac{g_t(-1)-W_t}{g_t'(-1)} ds+ \sqrt{\kappa}d\tilde B_s+\frac{(g_t(-1)-W_t)^2}{g_t'(-1)} J(W_t-g_t(0), W_t-g_t(-1))ds.
\end{align*}
Let $\tilde O_s:=\tilde g_s(0)=f_t(g_t(0)), \tilde V_s:=\tilde g_s(1)=f_t(g_t(\infty))=a_t$, then we have
\begin{align*}
d\tilde W_s=\sqrt{\kappa}d\tilde B_s+(\kappa-6)\frac{ds}{\tilde W_s-\tilde V_s}+\kappa \frac{\tilde V_s-\tilde O_s}{(\tilde W_s-\tilde V_s)(\tilde O_s-\tilde W_s)} Q\left( \frac{\tilde O_s-\tilde W_s}{\tilde W_s-\tilde V_s} \right)ds.
\end{align*}

Thus we have proved the following Lemma.
\begin{lemma}
The image of $( \gamma_t, 0\le t< T)$ by $f$ under a proper time change $t\mapsto s(t)$ (see (\ref{align:time-change})) 
is a Loewner curve  
$(\tilde \gamma_s, s\le \tilde S=\inf \{ s\ge0, \tilde V_s -\tilde W_s=0\})$ 
generated by the Loewner equation
\begin{align*}
\partial_s \tilde g_s(z)=\frac{2}{\tilde g_s(z)-\tilde W_s}, \quad \tilde g_0(z)=z,
\end{align*}
where $(\tilde O_s,\tilde V_s, \tilde W_s)$ is the solution of the  differential system
\begin{equation}
\left\{
\begin{split}
&d\tilde O_s=\frac{2ds}{\tilde O_s-\tilde W_s},\\
&d\tilde V_s=\frac{2ds}{\tilde V_s-\tilde W_s},\\
&d\tilde W_s=\sqrt{\kappa}d\tilde B_s+\tilde J(\tilde V_s-\tilde W_s, \tilde W_s-\tilde O_s),\\
& \tilde O_0=0,\, \tilde V_0=1,\, \tilde W_0=f(x)\ge 0,
\end{split}
\right.
\end{equation}
and where
\begin{align*}
\tilde J (p,q): =\frac{6-\kappa}{p}+\kappa\left(\frac{1}{p}+\frac{1}{q}\right)Q\left(\frac{q}{p}\right).
\end{align*}
\end{lemma}

Note that the curve $\tilde \gamma$ is well defined when $\tilde W_s-\tilde O_s=0$ (or equivalently when $\tilde \gamma$ hits $(-\infty,0]$), because the curve $\gamma$ is well defined when it hits $(-1,0]$ (or we can argue similarly as in  \S\,\ref{sec:girsanov} by doing Girsanov transformation for $\tilde\gamma$ directly).

To see whether $\tilde V_s-\tilde W_s = 0$ (whether $\tilde S<\infty$), we need to do again a Girsanov transformation as in \S\,\ref{sec:girsanov}.
Note that when $(\tilde V_s-\tilde W_s)/(\tilde W_s-\tilde O_s)$ is small, by the estimate (\ref{q2}), the quantity $J(\tilde V_s-\tilde W_s, \tilde W_s-\tilde O_s)$ is approximately
\begin{align*}
(\lambda\kappa+6-\kappa)\frac{1}{\tilde V_s-\tilde W_s}.
\end{align*}
Hence when $(\tilde V_s-\tilde W_s)/(\tilde W_s-\tilde O_s)$ is small,  ${\tilde V_s-\tilde W_s}$ is absolutely continuous w.r.t. a Bessel process of dimension $d=3-2\lambda- (8/\kappa)$.
In the case when $\kappa=8/3, \lambda>-3/4$, we have $d=-2\lambda <{3}/{2}$. Hence $\tilde S<\infty$ a.s. Note that the constraint for $(\tilde V_s-\tilde W_s)/(\tilde W_s-\tilde O_s)$ to be small will be satisfied for some $s$ because the process $(\tilde V_s-\tilde W_s)/(\tilde V_s-\tilde O_s)$ oscillates between $[0,1]$ and has positive probability to get to any neighborhood of $0$.

Note that $\tilde S<\infty$ corresponds to the case (i) and (ii) at the end of \S\,\ref{sec:girsanov} for the original curve $\gamma$ and the case $\tilde S=\infty$ corresponds to the case (iii) in which $\gamma$ hits the point $-1$.
Hence for $\lambda,\mu,\nu$ in (\ref{first-condition}), we are in case (i) or (ii) and the curve hSLE$_{8/3}(\lambda,\mu,\nu)$ does not hit the point $-1$ a.s.

Let us now make a remark on the degenerate case.
\begin{remark}
If we allow $b=c$ in (\ref{first-condition}), then this corresponds to the degenerate case in \S\, \ref{degenerated case}, and we have
\begin{align*}
\tilde J(\tilde V_s-\tilde W_s, \tilde W_s-\tilde O_s)=\frac{6-\kappa(\lambda+3)}{\tilde V_s-\tilde W_s}+\frac{\kappa\nu}{\tilde W_s-\tilde O_s}.
\end{align*}
In this case $\tilde \gamma$ is a SLE$_\kappa(\rho_1,\rho_2)$ with $\rho_1=\kappa(\lambda+3)-6, \rho_2=\kappa\nu$. Since $\kappa=8/3, \lambda>-3/4$, we have $\rho_1>0, \rho_2>-2$. This means $\tilde S=\infty$.
We are in case (iii)  mentioned at the end of \S\,\ref{sec:girsanov}. The original curve $\gamma$ hits $-1$ a.s.
\end{remark}

\subsubsection{Evolution after $T$ and conclusion}\label{sect:sle-concl}
Under condition (\ref{first-condition}), we are either in case (i) or (ii) mentioned after Lemma \ref{lem:nu}. Therefore
$({W_t-O_t})/({O_t-V_t})\to\infty \mbox{ as } t\to T.$ By the expansion (\ref{q2}),  $J(W_t-O_t, W_t-V_t)$ behaves approximately like  ${\kappa\lambda}/({W_t-O_t})$. By redoing the Girsanov transformation as in \S\,\ref{sec:girsanov}, we have that as $t\to T$, the hSLE curve $\gamma$ is locally absolutely continuous w.r.t. a SLE$_\kappa(\kappa\lambda)$ process. Similarly to Lemma \ref{lem:nu}, we have the following.
\begin{lemma}\label{lem:lambda}
In the $\kappa=8/3$ case, under condition (\ref{first-condition}),
\begin{itemize}
\item
 $T=\infty$ and $\lim_{t\to T}\gamma_t =\infty$ if and only if $\lambda\ge-1/4$.
 \item
 $T<\infty$ and $\lim_{t\to T}\gamma_t \in (-\infty,-1)$ if and only if $-3/4 <\lambda<-1/4$.
 \end{itemize}
\end{lemma}

In the $T<\infty$ case, for $t>T$, $V_t,O_t$ will stick to each other and the function $J$ degenerates to a rational function. The rest of the curve is just a $SLE_\kappa(\kappa\lambda)$ with marked point $V_t=O_t$.

Now we can summarize the above-discussed properties of hSLE processes.
\begin{proposition}\label{prop:hsle}
The hSLE$_{8/3}(\lambda,\mu,\nu)$ process is well defined for $\lambda,\mu,\nu$ inside (\ref{first-condition}). It is almost surely a simple curve and moreover
\begin{itemize}
\item[-] it does not hit the point $-1$;
\item[-] it hits $(-1,0)$ if and only if $\nu<-1/4$;
\item[-] it hits $(-\infty, -1)$ if and only if $\lambda<-1/4$.
\end{itemize}
When $\lambda,\mu,\nu$ are such that $\lambda>-3/4, \mu\ge-1/4, \nu>-3/4, \mu=\nu+\lambda+3/2$, 
the hSLE$_{8/3}(\lambda,\mu,\nu)$ defined by (\ref{hsle}) where $J$ is given by (\ref{J}) is actually an SLE$_\kappa(\rho_1,\rho_2)$ process that hits the point $-1$. It is unable to continue thereafter.
\end{proposition}

\section{One-sided restriction}\label{s5}

As in the case of the chordal conformal restriction, it will be useful in the trichordal case to first study the following larger family of random sets, that satisfy conformal restriction on just one side: 
In the $(\H,-1,0,\infty)$ setting, let $\Omega^1(\H,-1,0,\infty)$ be the collection of simply connected and closed $K\subset\overline\H$ such that $K\cap\R=(-\infty,0]$ and $K$ is unbounded.
 A probability measure $\P^1_{\H,-1,0,\infty}$ on $\Omega^1(\H,-1,0,\infty)$ is said to verify \emph{one-sided trichordal conformal restriction for the side $(0,\infty)$} if for all $A\in\mathcal{Q}^{(0,\infty)}$, the law of $\varphi_A(K)$ conditioned on $\{K\cap A=\emptyset\}$ is the same as the law of $K$.

Clearly, if a random set satisfies (three-sided) trichordal conformal restriction, it also defines a set that satisfies trichordal one-sided conformal restriction (just add the connected components of the complement of the restriction sample, that have part of the negative half-line on their boundary). 

The argument given in \S\,\ref{Characterization} shows also that there exist three constants $\alpha,\beta,\gamma\in\R$ such that for all $A\in\mathcal{Q}^{(0,\infty)}$,
\begin{align*}
\P[K\cap A=\emptyset]=\phi_A'(\infty)^\alpha\phi_A'(-1)^{\beta}\phi_A'(0)^\gamma
\end{align*}
and conversely, it is easy to see there exists at most one one-sided restriction measure, that we denote by $\P^1 (\alpha, \beta, \gamma)$, that satisfies this identity. Our goal is now to determine the range of accessible values 
for these three parameters.

In the following subsections, we are first going to construct the right boundaries of the $\P^1(\alpha,\beta,\gamma)$ samples using hSLE curves, for a certain range of $\alpha,\beta,\gamma$. Then we will construct some limiting cases
that are not covered by this first construction, using SLE$(\rho)$ curves.
We will then use Poisson point process of Brownian excursions to construct yet another sub-family of $\P^1(\alpha,\beta,\gamma)$ distributions. In the end, we will prove that these three constructions combined together will give the full range of possible values for $\alpha,\beta,\gamma$.

\subsection{Construction via hSLE curves}\label{sec:construction-by-hsle}
In this section, we are going to construct one-sided restriction measures by constructing their right boundaries that are actually the hSLEs defined in the previous section. 
To this purpose, we will adopt the same strategy as the one used by Lawler, Schramm and Werner in the chordal case \cite{MR1992830}: 
We will first find out an explicit  martingale $(M_t, t\ge 0)$, which we will then show to correspond to the quantity $\P(\gamma\cap A=\emptyset | \gamma[0,t])$ 
by inspecting its $L^1$ convergence.
Let $\gamma$ be a hSLE$_{8/3}(\lambda,\mu,\nu)$ such that $\lambda, \mu, \nu$ are in (\ref{first-condition}).
Let $A\in\mathcal{Q}^{(0,\infty)}$.  Let $A_t:=g_t(A)$ and $h_t:=g_{A_t}$.  
 \begin{figure}[h!]
    \centering
    \includegraphics[width=0.86\textwidth]{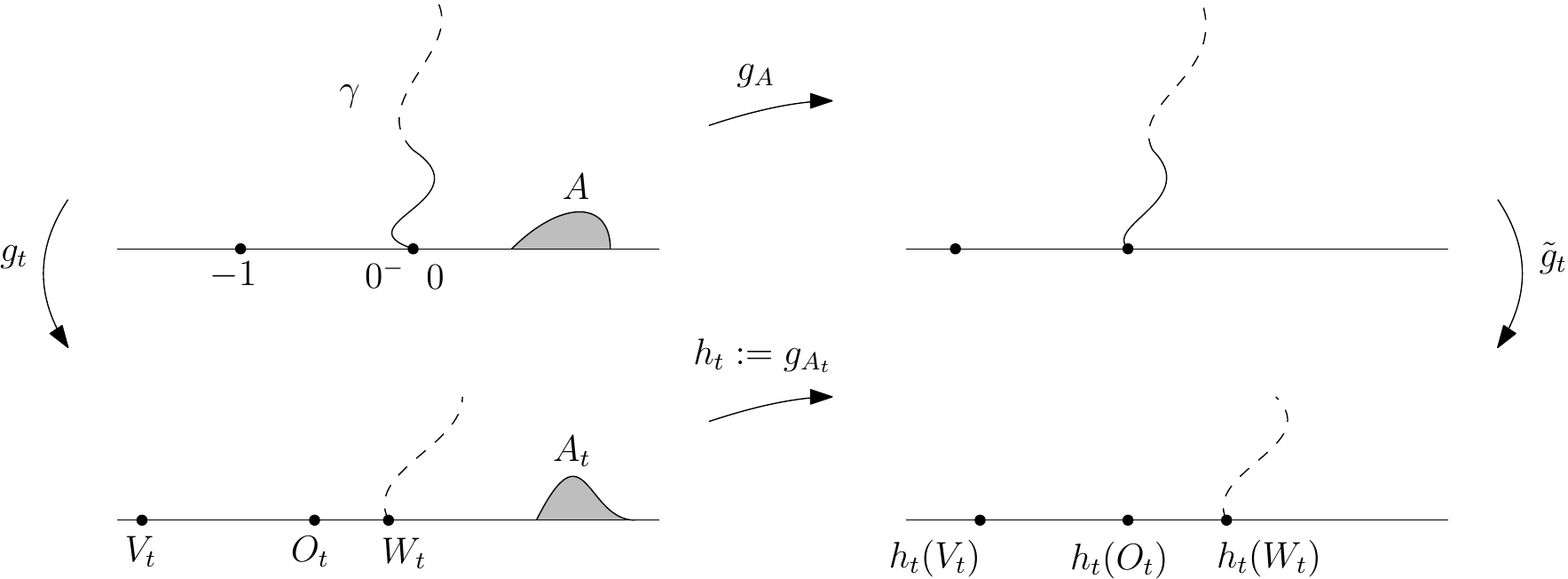}
    \caption{Commutative diagram}
    \label{fig:commutative-diagram}
\end{figure}

For all $t>0$, when neither $W_t-O_t$ nor $O_t-V_t$ is $0$, 
let us define
\begin{align*}
M_t: =h_t'(W_t)^{5/8}h_t'(O_t)^{n}h_t'(V_t)^m
\left(\frac{h_t(O_t)-h_t(V_t)}{O_t-V_t}\right)^{l-m-n+\lambda}\\
{G\left(\frac{h_t(W_t)-h_t(O_t)}{h_t(O_t)-h_t(V_t)}\right)} \bigg/ {G\left(\frac{W_t-O_t}{O_t-V_t}\right)}.
\end{align*}
Knowing that as $t\to 0$
\begin{align*}
\frac{h_t(W_t)-h_t(O_t)}{h_t(O_t)-h_t(V_t)}\to 0, \, \frac{W_t-O_t}{O_t-V_t}\to 0,\, \frac{h_t(W_t)-h_t(O_t)}{h_t(O_t)-h_t(V_t)} \bigg/  \frac{W_t-O_t}{O_t-V_t}\to \frac{g_A'(0)}{g_A(0)-g_A(-1)}
\end{align*}
and knowing
 the expansion (\ref{z-to-0}), we define by continuity
\begin{align}\label{M0}
M_0:=\lim_{t\to 0} M_t= g_A'(0)^{5/8+n+\nu} g_A'(-1)^{m} \left( g_A(0)-g_A(-1) \right)^{l-m-n+\lambda-\nu}.
\end{align}
If $\gamma$ hits $(-1,0]$ at a time $T_1<\infty$, then $V_{T_1}<O_{T_1}=W_{T_1}$. Define by continuity

$$M_{T_1}:=\lim_{t\to {T_1}}M_t=h_{T_1}'(W_{T_1})^{5/8+n+\nu}h_{T_1}'(V_{T_1})^m\left(\frac{h_{T_1}(O_{T_1})-h_t(V_{T_1})}{O_{T_1}-V_{T_1}}\right)^{l-m-n+\lambda-\nu}.$$ 
If $\gamma$ hits $(-\infty,-1)$ at a time ${T_2}<\infty$, then $V_{T_2}=O_{T_2}=W_{T_2}$. Moreover, 
$$\frac{h_t(W_t)-h_t(O_t)}{h_t(O_t)-h_t(V_t)}\underset{t\to {T_2}}{\longrightarrow}\infty.$$  
Let us define by continuity
$$M_{T_2}:=\lim_{t\to {T_2}}M_t=h_{T_2}'(W_{T_2})^{5/8+l+\lambda}.$$
After the first time ${T_2}$ that $\gamma$ hits $(-\infty,-1)$, the points $V_t$ and $O_t$ will stick together. Thereafter let 
$$M_t=h_t'(W_t)^{5/8}h_t'(O_t)^{l+m+\lambda-\nu-\mu} \left(\frac{h_t(W_t)-h_t(O_t)}{W_t-O_t}\right)^{\nu+\mu-m}.$$

\begin{remark}
Note that $M_t$ is well defined only for $t\le T_A=\inf\{t: \gamma(0,t)\cap A=\emptyset\}$. Note that this $T_A$ should not be confused with the $T$ of the previous section; we will keep this definition of $T_A$ until the end of the paper.
\end{remark}

\begin{remark}
As mentioned at the beginning of \S\,\ref{sec:hsle}, conditionally on $\gamma[0,t]$, the remainder of the curve should satisfy a kind of four-point restriction property, 
and the martingale $M_t$ indeed take the form of four-point restriction probabilities as suggested by (\ref{dubedat}).
\end{remark}

\begin{lemma}\label{martingale-1-side} The process
$(M_t, t\ge 0)$ is a local martingale. Moreover for all $A\in\mathcal{Q}^{(0,\infty)}$, there exist a constant $C$, possibly dependent on $\lambda,\mu,\nu$ and $A$, such that $0\le M_t\le C$ for all $0\le t\le T_A$ a.s.
\end{lemma}

\begin{proof}
The fact that $M$ is a local martingale can be shown by straightforward but {longish} It\^o computation that we postpone to the \S\,\ref{appendix}.

Without loss of generality, we can assume that $\partial A$ is a simple curve, since the semigroup generated by such $A$ will be dense in $\mathcal{Q}^{(0,\infty)}$.
For a fixed time $t$,  let $S_t:=\capacity(A_t)$.
Let $\beta_t:[0,S_t]\to\overline\H$ be the simple curve $\overline{\partial A_t\cap\H}$ starting from $\beta_t(0)=\inf(A_t\cap\R)$ and parametrized by half plane capacity. 

Note that $S_t=\capacity(A_t)\le \capacity(A)$ for all $t\ge 0$ because
that $$
\capacity(\gamma(0,t))+\capacity(A_t)=\capacity(A\cup \gamma(0,t)) \le \capacity(A)+\capacity(\gamma(0,t))
$$
by the well-known subadditivity property of the half-plane capacity.
Set $\hat g_s=g_{\beta_t[0,s]}$ and $x_s=\hat g_s(\beta_t(x))$. By chordal version of Loewner's theorem, $\partial_s\hat g_s(z)=2/(\hat g_s(z)-x_s)$. We have $\hat g_{S_t}=h_t$ since both are the normalized map from $\H\setminus A_t$ onto $\H$. Since $\partial_s\hat g_s'(z)=-2\hat g_s'(z)/(\hat g_s(z)-x_s)^2$, it follows that
\begin{align*}
\partial_s \log \hat g'_s(z)=\frac{-2}{(\hat g_s(z)-x_s)^2}.
\end{align*}
Note that $V_t\le O_t\le W_t \le x_0$ because $A_t\in\mathcal{Q}^{(0,\infty)}$. Let $v_s=\hat g_s(V_t), o_s=\hat g_s(O_t), w_s=\hat g_s(W_t)$. Then $v_s\le o_s\le w_s \le x_s$.  We  have
\begin{align*}
\partial_s\log(w_s-v_s)=\frac{-2}{(x_s-w_s)(x_s-v_s)}.
\end{align*}
Let
\begin{align}\label{qw}
q(z):=-zw'(-z)/w(-z),
\end{align}
then
\begin{align*}
\partial_s\log w\left(\frac{w_s-o_s}{o_s-v_s}\right)=-2q\left(\frac{w_s-o_s}{o_s-v_s}\right)\left(\frac{1}{(x_s-w_s)(x_s-o_s)}-\frac{1}{(x_s-o_s)(x_s-v_s)} \right).
\end{align*}
Therefore, for 
\begin{align}\label{P_lmn}
P_{\lambda,\mu,\nu}(x,y):=nx^2+(l-m-n+\lambda-\nu-\mu)xy+my^2+\nu x+\mu y+\frac58,
\end{align}
 we have
\begin{equation}\label{exp-M}
\begin{split}
M_t=\exp \int_0^{S_t} -2&\left[ P_{\lambda,\mu,\nu}\left(\frac{x_s-w_s}{x_s-o_s},\frac{x_s-w_s}{x_s-v_s}\right)\right.\\
&\left.+q\left(\frac{w_s-o_s}{o_s-v_s}\right)\left(\frac{1}{(x_s-w_s)(x_s-o_s)}-\frac{1}{(x_s-o_s)(x_s-v_s)} \right) \right] ds.
\end{split}
\end{equation}
Since $S_t$ is uniformly bounded, it is enough to prove that the quantity
\begin{align*}
f(x,y):=P_{\lambda,\mu,\nu}(x,y)+q\left(\frac{x-1}{y-x}y \right)x(1-y)
\end{align*}
is uniformly bounded from below for all $0\le y\le x\le 1$. It is true because $f$ is a continuous function and $0\le y\le x\le 1$ is a compact set.
\end{proof}

\begin{proposition}\label{prop:one-sided}
For $\lambda,\mu,\nu$ in (\ref{first-condition}), let $\gamma$ be a hSLE$_{8/3}(\lambda,\mu,\nu)$  from $0$ to $\infty$, having $-1,0^-$ as marked points. The closure of the connected component of $\H\setminus \gamma$ at the left of $\gamma$ has the  law  $\P^1(\alpha,\beta,\gamma)$.
\end{proposition}
Recall that $\alpha,\beta,\gamma$ are given by (\ref{alphabetagamma}) hence
\begin{align*}
\alpha=\frac23\lambda^2+\frac{4}{3}\lambda+\frac{5}{8}, \quad
\beta=\frac{2}{3}\mu^2+\frac{1}{3}\mu, \quad
\gamma=\frac{2}{3}\nu^2+\frac{4}{3}\nu+\frac{5}{8}.
\end{align*}

\begin{proof}
By similar arguments as in \cite[Lemma 6.2, Lemma 6.3]{MR1992830}, we can show that
\begin{itemize}
\item[-] If $\gamma\cap A=\emptyset$, then ${T_A}=\infty$. As $t\to {T_A}$,
\begin{align*}
&h_t'(W_t)\to 1, \, h_t'(O_t)\to 1, \, h_t'(V_t)\to 1, \, \frac{h_t(O_t)-h_t(V_t)}{O_t-V_t}\to 1,\\
&\frac{h_t(W_t)-h_t(O_t)}{h_t(O_t)-h_t(V_t)}\to\infty,\, \frac{W_t-O_t}{O_t-V_t}\to\infty
\end{align*}
and
\begin{align*}
\frac{h_t(W_t)-h_t(O_t)}{h_t(O_t)-h_t(V_t)} \bigg/ \frac{W_t-O_t}{O_t-V_t} \to 1.
\end{align*}
By the expansion (\ref{z-to-infty}), we have $M_t\to 1$.

\item[-] If $\gamma\cap A\not=\emptyset$, then ${T_A}<\infty$. As $t\to {T_A}$,
$
h_t'(W_t)\to 0
$
and the quantities $$h_t'(O_t), \, h_t'(V_t), \, \frac{h_t(O_t)-h_t(V_t)}{O_t-V_t},\,\frac{h_t(W_t)-h_t(O_t)}{h_t(O_t)-h_t(V_t)},\, \frac{W_t-O_t}{O_t-V_t}$$ all go to  finite limits.
Hence $M_t\to 0$.
\end{itemize}
Therefore $M_t\to\mathbf{1}_{\gamma\cap A=\emptyset}$ as $t\to {T_A}$. Since Lemma \ref{martingale-1-side} says that $(M_t, t\ge 0)$ is a bounded martingale, we get by martingale convergence theorem and (\ref{M0}) that
\begin{align*}
\P({\gamma\cap A=\emptyset})&=\E(\mathbf{1}_{\gamma\cap A=\emptyset})=\E(M_0)=g_A'(0)^{5/8+n+\nu} g_A'(-1)^{m} \left( g_A(0)-g_A(-1) \right)^{l-m-n+\lambda-\nu}\\
&=\varphi_A'(0)^{(2/3)\nu^2+(4/3)\nu+5/8} \varphi_A'(-1)^{(2/3)\mu^2+(1/3)\mu} \varphi_A'(\infty)^{(2/3)\lambda^2+(4/3)\lambda+5/8}.
\end{align*}
\end{proof}

\begin{remark}
For $\lambda,\mu,\nu$ in (\ref{first-condition}), using the family of hSLE$(\lambda, \mu, \nu)$ curves, we can construct $\P^1(\alpha,\beta,\gamma)$ for all $(\alpha,\beta,\gamma)\in\R^3$ such that
\begin{align}\label{range:1}
 \alpha> 0,\quad\gamma>0,\quad -\frac{1}{24} \le \beta <  \tilde{\xi}(\alpha,\gamma) .
\end{align}
\end{remark}

\subsection{A limiting case: construction via SLE$_{8/3}(\rho_1, \rho_2)$ curves}\label{sec:limit}
In this section, we want to construct the case $\alpha>0, \gamma>0, \beta=\tilde\xi(\alpha,\gamma)$. This corresponds to the degenerate case in \S\,\ref{degenerated case}, in which we have
\begin{align}\label{case b=c}
\lambda>-\frac{3}{4}, \nu>-\frac{3}{4}, \mu=\nu+\lambda+\frac32.
\end{align}
{By Proposition \ref{prop:hsle}, an hSLE with such limiting parameters is actually an SLE$(\rho_1, \rho_2)$.} We will construct a family of random sets $K$ using SLE$_{8/3}(\rho_1,\rho_2)$ processes and later prove that they satisfy one-sided trichordal restriction.
The construction goes as the following.

\begin{itemize}
\item[(i)] For $\rho_1=\kappa\nu, \rho_2=\kappa\lambda+2$, let $\gamma_0$ be a SLE$_{8/3}(\rho_1,\rho_2)$ in $\mathbb{H}$ from $0$ to $\infty$ having $0^-,1$ as marked points. 
Let $K_0$ be the closure of the connected component of $\H\setminus\gamma_0$ which is to the left of $\gamma_0$.
\item[(ii)]
Let $\tilde K$ be a one-sided chordal restriction measure of exponent $\alpha(\lambda)$ in configuration $(\H, 0,\infty)$ for the side $(-\infty,0)$, independent of $\gamma_0$. 
Let $\phi$ be the conformal map from $\H$ to $\H\setminus K_0$ that sends the boundary $[0,\infty]$ to $[1,\infty]$.
Then $K_1:=\phi(\tilde K)$ is a one-sided chordal restriction measure of exponent $\alpha(\lambda)$ in the configuration $(\H\setminus K_0, 1, \infty)$ for the side $\gamma_0\cup(0,1)$.
\item[(iii)] Let $K:=f( K_0\cup K_1)$ where $f: z\mapsto {z}/({1-z})$ is the M\"obius transformation from $\H$ to itself that sends $\infty,0,1$ to $-1,0,\infty$.
\end{itemize}

\begin{figure}[h!]
    \centering
    \includegraphics[width=0.86\textwidth]{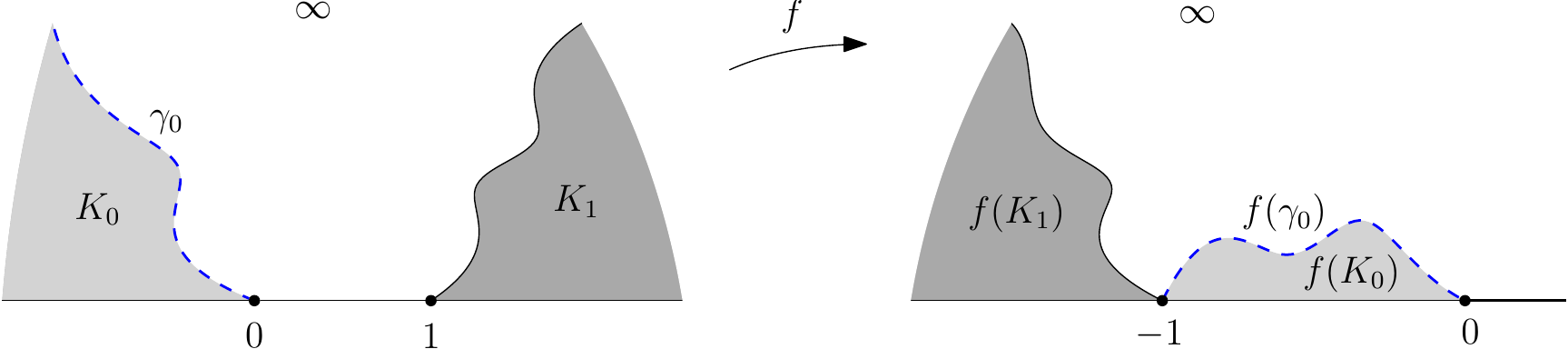}
    \caption{Construction of $K$.}
    \label{fig:limit-case}
\end{figure}

We will again use a proper martingale in order to prove the restriction property of $K$.
For all $A\in\mathcal{Q}^{(0,1)}$, let $W_t, O_t, V_t$ 
be associated to the Loewner curve $\gamma_0$ in a way that
$W_t$ is the driving function that generates $\gamma$ and $O_t=g_t(0), V_t=g_t(1)$. Define $h_t$
in the same way as in \S\, \ref{sec:construction-by-hsle}. 

\begin{lemma}
The following is a bounded martingale:
\begin{align*}
N_t&:=h_t'(W_t)^{5/8}h_t'(O_t)^{a_1} h_t'(V_t)^{a_2} \\
&\left(\frac{h_t(W_t)-h_t(O_t)}{W_t-O_t}\right)^{b_{01}}\left(\frac{h_t(W_t)-h_t(V_t)}{W_t-V_t}\right)^{b_{02}}\left(\frac{h_t(V_t)-h_t(O_t)}{V_t-O_t}\right)^{b_{12}},
\end{align*}
where
$$
a_1=\frac{\rho_1(4+3\rho_1)}{32}, a_2=\frac{\rho_2(4+3\rho_2)}{32}, b_{01}=\frac{3}{8}\rho_1, b_{02}=\frac{3}{8}\rho_2 , b_{12}=\frac{3}{16}\rho_1\rho_2.
$$
\end{lemma}
\begin{proof}
We can first prove by It\^o computation that $N$ is a local martingale (for example see \cite{MR1992830}).
Then we should write $N$ in the form of (\ref{exp-M}). Using the same notations as in the proof of Lemma \ref{martingale-1-side}, we have
\begin{align*}
N_t=\exp \int_0^{S_t} -2 R \left( \frac{w_s-x_s}{o_s-x_s}, \frac{w_s-x_s}{v_s-x_s} \right) ds,
\end{align*} 
where
\begin{align*}
R(x,y)=a_1 x^2 +a_2 y^2+b_{12} xy+ b_{01} x+b_{02} y+\frac58.
\end{align*}
Note that in our case $o_s\le w_s\le x_s\le v_s$. Hence it is enough to prove that $R(x,y)$ is bounded from below for $0\le x\le 1, y\le 0$. 
This is true because as $y\to-\infty$, the leading term of $R(x,y)$ is $a_2 y^2$ where $a_2>0$.
\end{proof}

\begin{proposition}\label{prop:limit}
The random set $K$ has law $\P^1\left(\alpha, \tilde\xi(\alpha,\gamma) ,\gamma\right)$.
\end{proposition}

\begin{proof}
{Let $T=\inf \{t: \gamma([0,t]) \cap {A}\not=\emptyset\}$.}
Similarly to the proof of Proposition \ref{prop:one-sided},
note that
\begin{itemize}
\item[-] If $\gamma\cap A=\emptyset$, then $T=\infty$. As $t\to \infty$, the  terms
\begin{align*}
h_t'(W_t), \, h_t'(O_t), \, \frac{h_t(V_t)-h_t(O_t)}{V_t-O_t},\,
\frac{h_t(W_t)-h_t(O_t)}{W_t-O_t},\, \frac{h_t(W_t)-h_t(V_t)}{W_t-V_t}
\end{align*}
all tend to $1$ almost surely. However, the term $h_t'(V_t)$ does not tend to $0$. In the same way as explained in \cite[\S\, 5.2]{MR2060031}, noting that $a_1=\alpha$, we have
\begin{align*}
h_t'(V_t)^{a_1}\underset{t\to\infty}{\longrightarrow} \P\left( K_1 \cap A=\emptyset \, | \, \gamma_0\right).
\end{align*}
\item[-] If $\gamma\cap A\not=\emptyset$, then $T<\infty$. As $t\to T$, $h_t'(W_t)\to 0$ and all the other terms in $N_t$ tend to a finite limit. Hence $N_T= 0$.
\end{itemize}
We have thus proved that
$$N_T= \mathbf{1}_{\gamma_0\cap A=\emptyset}\, \P\left( K_1 \cap A=\emptyset \,|\, \gamma_0 \right)=\P\left( (K_0\cup K_1)\cap A=\emptyset\,|\, \gamma_0 \right).$$
By the martingale convergence theorem, we have
\begin{align*}
\P\left( (K_0\cup K_1)\cap A=\emptyset\right)=\E(N_{T})=\E(N_0)&=
g_A'(0)^{5/8+a_1+b_{01}} g_A'(1)^{a_2} \left( g_A(1)-g_A(0) \right)^{b_{02}+b_{12}}\\
&=\varphi_A'(0)^{\gamma(\nu)}\varphi_A'(1)^{\alpha(\lambda)}\varphi_A'(\infty)^{\tilde\xi(\alpha,\gamma)}.
\end{align*}
Then it is easy to see that
\begin{align*}
\P(K\cap f(A)=\emptyset)=\varphi_{f(A)}'(0)^{\gamma(\nu)}\varphi_{f(A)}'(\infty)^{\alpha(\lambda)}\varphi_{f(A)}'(-1)^{\tilde\xi(\alpha,\gamma)}.
\end{align*}
The desired proposition follows.
\end{proof}
\begin{remark}
Proposition \ref{prop:limit} constructs $\P^1(\alpha,\beta,\gamma)$ for all $(\alpha,\beta,\gamma)\in\R^3$ such that
\begin{align*}
\alpha>0,\quad \gamma>0,\quad, \beta=\tilde \xi(\alpha,\gamma).
\end{align*}
\end{remark}
The following lemmas are  obvious from the construction.
\begin{lemma}
Conditionally on $f(K_0)$, $f(K_1)$ is a one-sided chordal restriction sample in $\H\setminus f(K_0)$.
\end{lemma}
\begin{lemma}\label{lem:-1}
The right boundary of the random set $K$ with law $\P^1(\alpha, \tilde\xi(\alpha,\gamma),\gamma)$ intersects $-1$ almost surely.
\end{lemma}

\subsection{Construction by Poisson point process of Brownian excursions}

Similarly to what we mentioned in \S\,\ref{sec:brownian-excursions} (or see \cite{MR2178043}), we can also construct one-sided trichordal measures by taking the filling of a Poisson point process of Brownian excursions.

More precisely, let $\mu_\H$ be the half-plane Brownian excursion measure as in (\ref{brownian-excursion-measure}). For 
 $\theta\ge0$, let $\mathcal{P}_{0}(\theta)$ be a Poisson point process of Brownian excursions of intensity $\theta \mu_\H$ but restricted to the set of excursions that start in $(-\infty,-1)$ and end in $(-1,0)$. 
\begin{lemma}\label{sub-family-1}
There exist $k>0$ such that
\begin{align*}
\P(\mathcal{P}_{0}(\theta)\cap A=\emptyset)=\varphi_A'(0)^{-k\theta}.
\end{align*} 
Hence $\mathcal{F}(\mathcal{P}_{0}(\theta))$ has the law $\P^1(0,-k\theta,0)$.
\end{lemma}

\begin{proof}
The conformal invariance and restriction property of $\mathcal{P}_{0}(\theta)$ follows from the conformal invariance and restriction property of $\mu_\H$. Hence there exist $\alpha, \beta, \gamma$ such that
\begin{align*}
\P(\mathcal{P}_{0}(\theta)\cap A=\emptyset)=g_A'(-1)^\beta g_A'(0)^\gamma (g_A(0)-g_A(-1))^{\alpha}.
\end{align*}
Let $A_\eps$ be the vertical slit from $\eps$ to $\eps+i$, then
\begin{align*}
\lim_{\eps\to 0}\P(\mathcal{P}_{0}(\theta)\cap A_\eps=\emptyset)= \P(\mathcal{P}_{0}(\theta)\cap A_0=\emptyset)>0
\end{align*}
because there will be only finitely many excursions that are near the point $0$. However 
$
g_{A_\eps}'(0)\to 0,
$
as $\eps\to 0$,
so we must have $\gamma=0$. 
By symmetry, we also have $\alpha=0$. Hence
\begin{align*}
\P(\mathcal{P}_{0}(\theta)\cap A=\emptyset)=\varphi_A'(0)^{\beta}.
\end{align*} 
For all $\theta_1,\theta_2>0$ and all $A\in\mathcal{Q}^{(0,\infty)}$, we have
\begin{align*}
\P(\mathcal{P}_{0}(\theta_1+\theta_2)\cap A=\emptyset)=\P(\mathcal{P}_{0}(\theta_1)\cap A=\emptyset)\P(\mathcal{P}_{0}(\theta_2)\cap A=\emptyset).
\end{align*}
Therefore 
$\beta$ must be linear with respect to $\theta$.
Note that $\varphi_A'(0)>1$ so we must have $\beta<0$ so that $\varphi_A'(0)^{\beta}$ can be a probability.
Hence there exist $k>0$ such that $\beta= -k\theta$.
\end{proof}

Since the filling of the union of independent restriction samples still satisfy restriction, we have many ways to play with it.
For example, let $\mathcal{P}_1$ be a Poisson point process of Brownian excursions that start and end on $(-\infty,-1)$ which corresponds to a chordal one-sided restriction measure of exponent $\alpha\ge 0$.
Let $\mathcal{P}_2$ be an Poisson point process of Brownian excursions that start and end on $(-1,0)$ which corresponds to a chordal one-sided restriction measure of exponent $\gamma\ge 0$.
Let $\mathcal{P}_0$ as in Lemma \ref{sub-family-1} which has law $\P^1(0,-k\theta,0)$ for $k\theta\ge 0$. Let $\mathcal{P}_0,\mathcal{P}_1,\mathcal{P}_2$ be independent.
It is easy to see that $\mathcal{F}\left( \mathcal{P}_{0} \cup \mathcal{P}_1 \cup\mathcal{P}_2 \right)$ has the law $\P^1(\alpha,\alpha-k\theta+\gamma,\gamma)$. Thus we have constructed using Poisson point process of Brownian excursions all measures $\P^1(\alpha, \beta,\gamma)$ for $(\alpha,\beta,\gamma)\in\R^3$ such that
\begin{align}
\alpha\ge0,\quad \gamma\ge0, \quad\beta\le \alpha+\gamma.
\end{align}
This construction requires no SLE knowledge and can be seen as one of the first evidences why trichordal restriction measures form a three-parameter family. However this construction does not give the entire family of desired restriction measures.

We can also take the union of restriction samples with law $\P^1\left( \alpha, \tilde\xi(\alpha,\gamma), \gamma \right)$ as constructed in \S\,\ref{sec:limit} and restriction samples with law  $\mathcal{P}_0(\theta)$ as in Lemma \ref{sub-family-1}.
This constructs all measures $\P^1(\alpha,\beta,\gamma)$ for $(\alpha,\beta,\gamma)\in\R^3$ such that
\begin{align}\label{one-sided-final-range}
\alpha \ge 0,\quad \gamma\ge 0, \quad \beta\le \tilde\xi(\alpha,\gamma),
\end{align}
which is a larger range than that of (\ref{range:1}).

Note that this construction does not use hSLE processes. However these processes will turn out to be necessary in later constructions of two-sided and three-sided restriction measures.

Now we can make a remark about the existence of cut points for a sample $K$ with law $\P^1(\alpha,\beta,\gamma)$, which is an immediate consequence of the construction above and of cut-point properties for chordal restriction measures in \cite{MR1992830}.
\begin{proposition}\label{prop:geom-one-sided}
Let $K$ be a sample which has law  $\P^1(\alpha,\beta,\gamma)$ where $\alpha,\beta,\gamma$ are in range (\ref{one-sided-final-range}).
\begin{itemize}
\item[-] If $0<\gamma<1/3$, then $K$ has infinitely many cut points in any neighborhood of $0$.
\item[-] If $\gamma\ge1/3$, then $K$ has no cut point on the line segment $(-1,0]$.
\item[-] If $\gamma=0$, then $0\not\in\overline{K\setminus [-1,0]}$ and $\overline{K\setminus [-1,0]}$ has no cut point on $(-1,0]$.
\item[-] The point $-1$ is a cut point of $K$ if and only if $\beta=\tilde\xi(\alpha,\gamma)$.
\end{itemize}
\end{proposition}
Due to symmetry, the situation in the neighborhood of $\infty$ is the same as in the neighborhood of $0$ if we consider $\alpha$ in stead of $\gamma$.
In Figure \ref{fig:one-sided} we make some illustrations of the above-mentioned geometric properties for some one-sided restriction measures in the unit disk (which is a more symmetric domain). The points $a,b,c$ have respectively restriction exponents $\alpha,\beta,\gamma$.

 \begin{figure}[h!]
    \centering
    \begin{subfigure}[t]{0.3\textwidth}
\centering
    \includegraphics[width=0.8\textwidth]{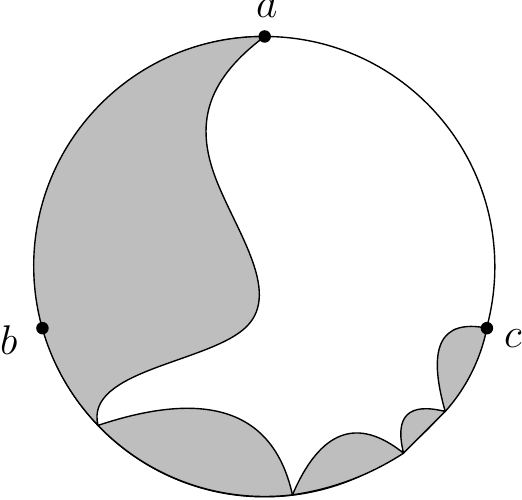}
    \caption{$\alpha\ge 1/3, \gamma<1/3, \beta<\tilde\xi(\alpha,\gamma)$. $K$ has infinitely many cut points in any neighborhood of $c$.}
    \label{fig:cut-point}
    \end{subfigure}
    \quad
    \begin{subfigure}[t]{0.3\textwidth}
    \centering
    \includegraphics[width=0.8\textwidth]{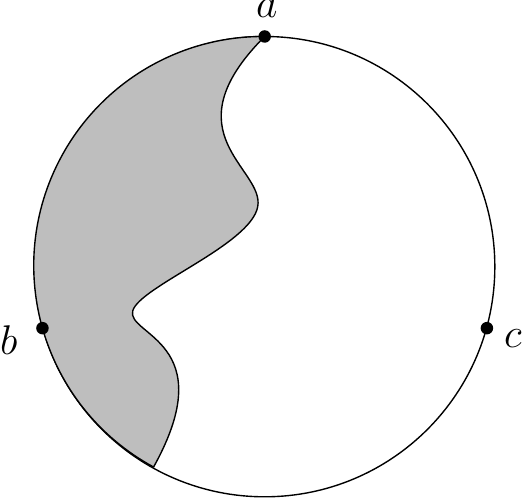}
    \caption{$\alpha\ge 1/3, \beta<\alpha, \gamma=0$. The closure of the interior of $K$ has no cut point and does not contain the point $c$.}
    \label{fig:cut-point}
    \end{subfigure}
    \quad
     \begin{subfigure}[t]{0.3\textwidth}
     \centering
    \includegraphics[width=0.8\textwidth]{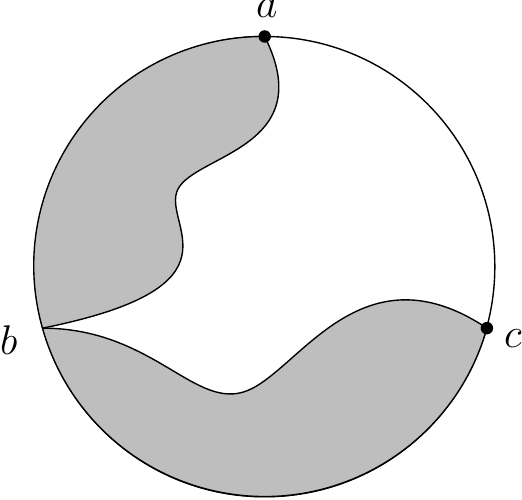}
    \caption{$\alpha,\gamma\ge 1/3, \beta=\tilde\xi(\alpha,\gamma)$. The point $b$ is the unique cut point of $K$.}
    \label{fig:cut-point}
    \end{subfigure}
    \caption{Some one-sided restriction samples in the unit disk.}
    \label{fig:one-sided}
\end{figure}

\subsection{Range of existence for one-sided restriction measures}
We will now show that (\ref{one-sided-final-range}) is the full range of values for which $\P^1(\alpha,\beta,\gamma)$ exist.
\begin{proposition}
$\P^1(\alpha,\beta,\gamma)$ does not exist for $(\alpha,\beta,\gamma)\not\in R_1$ where
\begin{align*}
R_1:= \left\{\left(\alpha,\beta,\gamma\right)\in\R^3;\, \alpha\ge 0,\, \gamma\ge 0,\, \beta\le \tilde\xi(\alpha,\gamma)\right\}.
\end{align*}
\end{proposition}

\begin{proof}
Let $A\in\mathcal{Q}^{(0,\infty)}$.
For $n\ge 1$, the scaled hull $A/n\in\mathcal{Q}^{(0,\infty)}$.
Note that $$g_{A/n}'(-1)\underset{n\to\infty}{\longrightarrow} 1, \quad  g_{A/n}(0)-g_{A/n}(-1)\underset{n\to\infty}{\longrightarrow} 1, \quad g_A'(0)=g_{A/n}'(0) \mbox{ for all } n.$$
If $K$ has the law of $\P^1(\alpha,\beta,\gamma)$ for some $\gamma<0$, then
\begin{align*}
\P(K\cap A/n=\emptyset)=g_{A/n}'(-1)^\beta g_{A/n}'(0)^\gamma \left( g_{A/n}(0)-g_{A/n}(-1) \right)^{\alpha-\beta-\gamma}\underset{n\to\infty}{\longrightarrow} g_A'(0)^\gamma>1
\end{align*}
since $g_A'(0)<1$. This is impossible.
The same argument applied to the sequence of hulls $(nA)_{n\ge 0}$ shows that $\P^1(\alpha,\beta,\gamma)$ does not exist either for $\alpha<0$.

Now it is left to show that $\P^1(\alpha,\beta,\gamma)$ does not exist for $\beta> \tilde\xi(\alpha,\gamma)$.
Suppose that $\P^1(\alpha,\beta,\gamma)$ exists for some $\alpha\ge 0, \gamma\ge 0, \beta> \tilde\xi(\alpha,\gamma)$ and let $K_1$ be a sample.
Let $K_0$ be an independent restriction sample with law $\P^1(0,\tilde\xi(\alpha,\gamma)-\beta, 0)$ which exists according to Lemma \ref{sub-family-1}. 
Then the fill-in of $K_1\cup K_0$ has the law $\P^1(\alpha,\tilde\xi(\alpha,\gamma), \gamma)$. The right boundary of $K_0$ does not intersect $\{-1\}$ almost surely, however according to Lemma \ref{lem:-1}, the right boundary of $K_1\cup K_0$ intersects $\{-1\}$ almost surely since it has law $\P^1(\alpha,\tilde\xi(\alpha,\gamma), \gamma)$. This is a contradiction.
\end{proof}

Now we are able to state a corollary that justifies the fact that we did impose the condition that $K$ does hit all three points $a$, $b$ and $c$ in our definition of 
the trichordal restriction property. 
\begin{corollary}\label{corollary}
Let $K$ be a random set that satisfies all axioms of the three-sided trichordal restriction sample except that we replace the condition $K\cap\partial D =\{a,b,c\}$ by the condition $K\cap\partial D\subset\{a,b,c\}$. Then, 
\begin{itemize}
\item[(i)] Almost surely,  $K\cap\partial D$ has at least two points.
\item[(ii)]
If $K\cap\partial D$ contains exactly two points with positive probability, then the law of the set $K$ is exactly a chordal restriction measure.
\end{itemize}
\end{corollary}
\begin{proof}
The statement (i) is due to the fact that there is no one-sided trichordal restriction sample $K$ such that the closure of the interior of $K$ intersects $\partial D$ at only one point.

For the statement (ii), note that one-sided trichordal restriction sample $J$ such that the closure of the interior of $J$ intersects only one arc of $\partial D$  are just one-sided chordal restriction samples. 
If $K\cap\partial D=\{a,b\}$, then the right boundary of $K$ is a one-sided chordal restriction sample.
Hence for all $A\subset D$ such that $D\setminus A$ is simply connected and $\overline A\cap \partial D\subset (bc)$, there is a $\alpha>0$ such that 
\begin{align}\label{corollary_align}
\P(K\cap A=\emptyset)=\left(\phi_A'(a)\phi_A'(b)\right)^\alpha.
\end{align}
By Lemma \ref{lem:123}, we know that (\ref{corollary_align}) is true for all $A\subset D$ such that $D\setminus A$ is simply connected and $a,b,c\not\in\overline A\cap \partial D$.
\end{proof}

\section{Two-sided restriction}\label{sec:two-sided}

We now turn our attention to the two-sided case: 
In the $(\H,-1,0,\infty)$ setting, let $\Omega^2(\H,-1,0,\infty)$ be the collection of simply connected and closed $K\subset\overline\H$ such that
$K\cap \R=[-1,0]$ and $K$ is unbounded.
 A probability measure $\P^2_{\H,-1,0,\infty}$ on $\Omega^2(\H,-1,0,\infty)$ is said to satisfy \emph{two-sided trichordal conformal restriction for the sides $(-\infty,-1)$ and $(0,\infty)$} if for all $A\in\mathcal{Q}^{(-\infty,-1)\cup(0,\infty)}$, the law of $\varphi_A(K)$ conditioned on $\{K\cap A=\emptyset\}$ is the same as the law of $K$.

Again, the same argument as before shows that there exist three constants $\alpha,\beta,\gamma\in\R$ such that for all $A\in\mathcal{Q}^{(-\infty,-1)\cup(0,\infty)}$,
\begin{align*}
\P[K\cap A=\emptyset]=\phi_A'(\infty)^\alpha\phi_A'(-1)^{\beta}\phi_A'(0)^\gamma
\end{align*}
and that conversely there exists at most one two-sided restriction measure, that we denote by $\P^2(\alpha, \beta, \gamma)$ that satisfies this. Our goal is now to determine the range of accessible values 
for these three parameters.

In the following subsections, we are first going to construct a big class of $\P^2(\alpha,\beta,\gamma)$ samples using hSLE curves and the one-sided restrictions samples that have been constructed in the previous section.
Then we will construct a class of limiting cases using a mixture of chordal restriction samples and Poisson point process of Brownian motions. 
In the end we will prove that these two constructions together will give the full range of accessible values for $\alpha,\beta,\gamma$.

\subsection{Construction by hSLE curves and one-sided samples}\label{sec:construction-two-sided}
In the present section, our goal is to construct a two-sided restriction sample by first constructing its right boundary by a hSLE curve and then putting a one-sided restriction sample to the left of the first curve. To this end, we are going to use the same local martingale $(M_t, t\ge 0)$ as before, but for $A\in\mathcal{Q}^{(-\infty,-1)\cup(0,\infty)}$.

In order for  $(M_t, t\ge 0)$ to be a bounded martingale, we need to restrict $\lambda,\mu,\nu$ to a range which is smaller than (\ref{first-condition}).
This is because two-sided restriction implies one-sided restriction but is harder to achieve.
Let us give the range right now.
\[
\left|\,
\begin{array}{lllllll}
\displaystyle{\lambda\ge 0} &\iff& a-b\ge 2 &\iff& \displaystyle{\alpha\ge\frac58} &\iff& l\ge 0\\[2mm]
\displaystyle{\mu\ge 0} &\iff & \displaystyle{a+b\ge c+\frac12} &\iff&  \displaystyle{\beta\ge0}&\iff& m\ge 0\\[2mm]
\displaystyle{\nu>-\frac34} &\iff & c>0 &\iff& \gamma>0 &\iff& \displaystyle{n\ge-\frac1{24}} \\[2mm]
\displaystyle{\mu<\lambda+\nu+\frac32 }& \iff & b<c &\iff& \beta<\tilde\xi(\alpha,\gamma)&\iff& (1) \\[2mm]
\displaystyle{\nu\le \lambda+\mu}&\iff& \displaystyle{\c\le \a-\frac12} &\iff& \gamma\le\tilde\xi(\alpha,\beta)&\iff& (2)\\[3mm]
\displaystyle{\lambda+\mu+\nu \ge -\frac12} &\iff& \displaystyle{\a\ge \frac32} &\iff& \tilde\xi(\alpha,\beta,\gamma)\ge1&\iff& (3).
\end{array}
\label{second-condition}\tag{Range 2}
\right.
\]
The equivalence here is to be understood in the sense that we assume above all $\lambda\ge0, \mu\ge 0, \nu>-3/4$. In this case, (\ref{align:bijection}) is true and moreover we have $\lambda=U(l), \mu=U(m)$. However, note that we have $\nu=U(n)$ only if $\nu\ge-1/4$, otherwise $\nu=-U(n)-1/2$.
Therefore 
if $\nu\ge-1/4$, then we have
\begin{align*}
(2)&\iff U(n)\le U(l)+U(m) \iff n\le\tilde\xi(l,m),\\
(3)&\iff U(l)+U(m)+U(n)\ge-1/2;
\end{align*}
if $\nu<-1/4$, then we have
\begin{align*}
(2)&\iff -U(n)-\frac12\le U(l)+U(m),\\
(3)&\iff U(l)+U(m)-U(n)-\frac12\ge-\frac12\iff n\le\tilde\xi(l,m).
\end{align*}
Note that $U(l)+U(m)+U(n)\ge-1/2$ is always true and hence we can say
\begin{align}\label{2+3}
(2)+(3)\iff n\le\tilde\xi(l,m).
\end{align}

Before proving the boundedness of the local martingale $M$, let us first prove a lemma which gives some estimation about the function $w$.
\begin{lemma}\label{auxiliary lemma}
For $a,b,c$ in (\ref{second-condition})  that satisfy in addition $b<0$ and $1-c\le-b$, we have 
\begin{align*}
p(z)=z\frac{w'(z)}{w(z)}<-b \quad \text{for all } z\in(-\infty,0).
\end{align*}

\end{lemma}
\begin{proof}
Note that $w'(0)=ab/c<0$ and $w(0)=1$. Hence $p(0)=0$ and there is a $M>0$ such that $p(z)>0$ for $z\in(-M,0)$.
We can therefore look at the function $\log p$ on $(-M,0)$ and compute
\begin{align*}
\left( \log p(z) \right)'=\frac1z+\frac{w''(z)}{w'(z)}-\frac{w'(z)}{w(z)}.
\end{align*}
Note that $w$ is a solution of (\ref{euler}), hence
\begin{align*}
\left( \log p(z) \right)'=\frac{1}{1-z} \left( \frac{(p(z)+a)(p(z)+b)}{p(z)} +\frac{1-c-p(z)}{z} \right).
\end{align*}
Observe that
\begin{align*}
\left( \log p(z) \right)'<0 \text{ when } 0<p(z)\le 1-c; \quad \left( \log p(z) \right)'>0 \text{ when } p(z)\ge -b.
\end{align*}
Therefore we must have $p(z)<-b$ for all $z\in(-\infty, 0)$.
\end{proof}

Now we are ready to prove our important lemma.
\begin{lemma}\label{martingale2}
Let   $\lambda,\mu,\nu$ be in (\ref{second-condition}).
For all  $A\in\mathcal{Q}^{(-\infty,-1)\cup (0,\infty)}$, there exist a constant $C$, possibly dependent on $\lambda,\mu,\nu$ and $A$, such that $0\le M_t\le C$ for all $t\ge 0$ a.s.
\end{lemma}
\begin{proof}
Since (\ref{second-condition}) is included in (\ref{first-condition}), for $A\in\mathcal{Q}^{(0,\infty)}$ the lemma follows from Lemma \ref{martingale-1-side}. Hence it is enough to prove the lemma for $A\in\mathcal{Q}^{(-\infty,-1)}$.
Let us keep the notations from Lemma \ref{martingale-1-side}. 
Recall that $P_{\lambda,\mu,\nu}$ and $q$ were defined by (\ref{P_lmn}) and (\ref{qw}) as
\begin{align*}
P_{\lambda,\mu,\nu}(x,y)&=nx^2+(l-m-n+\lambda-\nu-\mu)xy+my^2+\nu x+\mu y+\frac58,\\
q(z)&=-zw'(-z)/w(-z).
\end{align*}
Here $x$ and $y$ actually represent the quantities $(x_s-w_s)/(x_s-o_s)$ and $(x_s-w_s)/(x_s-v_s)$.

The only modification from the proof of Lemma \ref{martingale-1-side}
is: For $A\in\mathcal{Q}^{(-\infty,-1)}$, we have $x_s\le v_s\le o_s\le w_s$.
We need to prove that the quantity
\begin{align*}
f(x,y)=P_{\lambda,\mu,\nu}(x,y)+q\left(\frac{x-1}{y-x}y \right)x(1-y)
\end{align*}
is bounded from below for $y\ge x\ge 1$.
Since $f$ is continuous, it is enough to look at $f$ when either $x$ or $y$ goes to $\infty$.

\underline{Let us first consider the case $m=0$.} In this case we also have $\mu=0$ and
\begin{align*}
f(x,y)=nx^2+(l-n+\lambda-\nu)xy+\nu x+\frac58 +q\left(\frac{x-1}{y-x}y \right)x(1-y).
\end{align*}
The condition $m=0$ and (\ref{second-condition}) also imply
\begin{align*}
\lambda\ge\nu, l\ge n, \lambda+\nu\ge-\frac12.
\end{align*}
Note that in this case we have $b\le0$ and $1-c\le-b$. We can apply Lemma \ref{auxiliary lemma}:
 $$q\left(\frac{x-1}{y-x}y \right)\le-b.$$
Since $x(1-y)\le 0$, we have
 \begin{align*}
f(x,y)&>nx^2+(l-n+\lambda-\nu)xy+\nu x+\frac58 -b x(1-y)=nx^2+(l-n) xy +\lambda x +\frac58 \\
&\ge n x^2+(l-n) x^2 +\lambda x+\frac58=l x^2 +\lambda x+\frac58 \ge 0.
\end{align*}

\underline{Now let us consider the case $m\not=0$.} In fact $m>0$.

If $x$ stays bounded and $y$ goes to $\infty$, then $f$ behaves  asymptotically as a polynomial in $y$ with the leading term $my^2$ where $m> 0$. It is obviously bounded from below.

If $x$ goes to $\infty$ then $y$ must also go to $\infty$.
By (\ref{lim2}) and knowing that $a-b\ge 2$, we have as $z\to\infty$, 
\begin{align*}
q(z)=-b+O(z^{-1}).
\end{align*}
Hence as $x,y\to\infty$, we have
\begin{align*}
q\left(\frac{x-1}{y-x}y \right) x(1-y) =-bx+ bxy+O(y-x).
\end{align*}
It follows that
\begin{align}\label{estim6}
f(x,y)=nx^2+(l-m-n)xy+my^2+(\lambda-\mu) x +\mu y +O(y-x).
\end{align}
Let us make a change of variable $x=1+c_1, y=1+c_1+c_2$ for $c_1,c_2\ge 0$ and get
\begin{align}\label{inequality-f}
f(x,y)\ge l c_1^2+(l+m-n) c_1c_2 + m c_2^2 + (2l+\lambda) c_1 + O(c_2).
\end{align}
Note that $x,y\to\infty$ is equivalent to $c_1\to\infty$. So the $O$ in the equation above is to be understood in the $c_1\to\infty$ limit.

\begin{itemize}
\item
If $l=0$ then $\lambda=0$ and (\ref{inequality-f}) becomes
\begin{align*}
f(x,y)=(m-n) c_1c_2 +m c_2^2+ O(c_2).
\end{align*}
Since in this case we also have $m\ge n$ and $m>0$, it follows that $f$ is bounded from below as $c_1\to\infty$.
\item
If $l>0$ (we also have $m>0$), then if ever $f$ is not bounded from below then we must have $l+m-n<0$ and the discriminant
$\Delta:=(l+m-n)^2-4ml >0$. Note that  in this case on the one hand $\Delta$ is increasing in $n$, on the other hand
by (\ref{2+3}) we have
\begin{align*}
n\le \tilde\xi(l,m)=U^{-1}(U(l)+U(m))=U^{-1}(\lambda+\mu).
\end{align*}
Hence
\begin{align*}
\Delta\le\left( l+m-U^{-1}(\lambda+\mu) \right)^2-4ml=-\frac{4}{9}\lambda\mu(2\lambda+2\mu+1)\le 0.
\end{align*}
which leads to a contradiction. Therefore $f$ is bounded from below.
\end{itemize}
\end{proof}

For $\lambda,\mu,\nu$ in
(\ref{second-condition}), we are going to construct  two-sided restriction samples $K$  by the following steps, see Figure \ref{fig:construction-two-sided}. 
\begin{construction}\label{cons:two-sided}
\begin{itemize}
\item[(i)] Let $\gamma_0$ be a hSLE$(\lambda,\mu,\nu)$ from $0$ to $\infty$.  Let $X:=\inf \gamma_0([0,\infty])\cap \R$. According to Proposition \ref{prop:hsle}, we have that $X=0$ when $\nu\ge-1/4$ and $X\in(-1,0)$ when $-3/4<\nu<-1/4$.
Let $\H^-$ be the connected component of $\H\setminus\gamma_0$ which has $(-\infty, X)$ as a part of its boundary.
\item[(ii)] Let $\tilde K$ be an independent sample $\P^1(m,n,l)$. 
Note that (\ref{second-condition}) together with (\ref{2+3}) ensures that $(m,n,l)\in R_1$ and $\P^1(m,n,l)$ exists.
Let $\phi$ be the conformal map from $\H$ to $\H^-$ that sends the boundary points $-1,0,\infty$ to $X,\infty,-1$.
\item[(iii)] Let $K:=\mathcal{F}\left(\phi(\tilde K)\cup \gamma_0\right)$.
\end{itemize}
\end{construction}

\begin{figure}[h]
\centering
\includegraphics[width=0.78\textwidth]{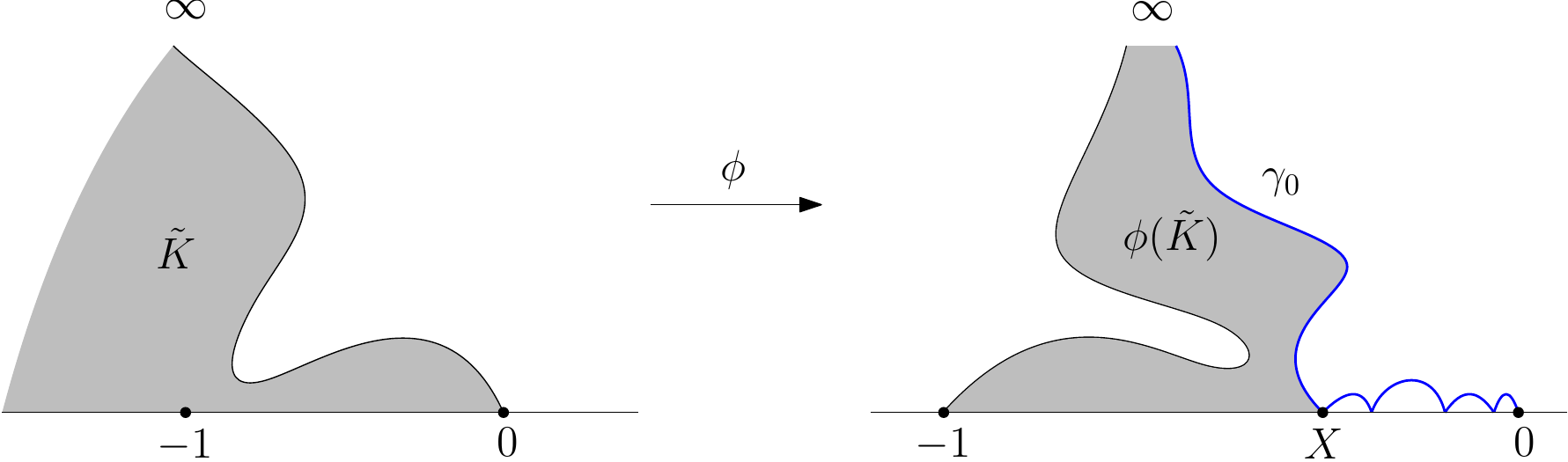}
\caption{Construction of two-sided restriction measures.}
\label{fig:construction-two-sided}
\end{figure}

\begin{proposition}\label{prop:two-sided}
For $\lambda,\mu,\nu$ in the range of Lemma \ref{martingale2}, the law of the set $K$
constructed above is $\P^2(\alpha,\beta,\gamma)$.
\end{proposition}
Recall again that
\begin{align*}
\alpha=\frac{2}{3}\lambda^2+\frac{4}{3}\lambda+\frac{5}{8}, \quad
\beta=\frac{2}{3}\mu^2+\frac{1}{3}\mu, \quad
\gamma=\frac{2}{3}\nu^2+\frac{4}{3}\nu+\frac{5}{8}.
\end{align*}
\begin{proof}
For $A\in\mathcal{Q}^{(0,\infty)}$,  this proposition follows from Proposition \ref{prop:one-sided}.
It is thus enough to consider the case $A\in\mathcal{Q}^{(-\infty,-1)}$.

Recall $T_A=\inf \{t: \gamma_0([0,t]) \cap {A}\not=\emptyset\}$.
Similarly to the proof of Proposition \ref{prop:one-sided} and Proposition \ref{prop:limit},
note that
\begin{itemize}
\item[-] If $\gamma_0\cap A=\emptyset$, then ${T_A}=\infty$. As ${T_A}\to\infty$, we have
\begin{align*}
h_t'(W_t)\to 1, \, \frac{h_t(W_t)-h_t(O_t)}{W_t-O_t} \to 1,\,
\frac{h_t(W_t)-h_t(O_t)}{h_t(O_t)-h_t(V_t)}\to\infty,\, \frac{W_t-O_t}{O_t-V_t} \to\infty.
\end{align*}
By the expansion (\ref{z-to-infty}), we have
\begin{align*}
\lim_{t\to\infty}M_t=\lim_{t\to \infty} h_t'(O_t)^{n} h_t'(V_t)^{m}\left(\frac{h_t(O_t)-h_t(V_t)}{O_t-V_t}\right)^{l-m-n}=\P\left( K\cap A=\emptyset\, |\, \gamma \right).
\end{align*}

\item[-] If $\gamma_0\cap A\not=\emptyset$, then ${T_A}<\infty$. As $t\to {T_A}$, $h_t'(W_t)\to 0$ and other terms in $M_t$ tend to a finite limit. Hence $M_{T_A}=0$.
\end{itemize}
We have thus proved that
\begin{align*}
M_{T_A}=\mathbf{1}_{\gamma_0\cap A=\emptyset} \,\P\left( K\cap A=\emptyset\,|\, \gamma_0 \right) = \P\left( K\cap A=\emptyset\,|\, \gamma_0 \right).
\end{align*}
By the martingale convergence theorem, we have
\begin{align*}
\P\left( K\cap A=\emptyset \right)
& =\E(M_{T_A})=\E(M_0) \\
&=\varphi_A'(0)^{(2\nu^2/3)+(4\nu/3)+5/8} \varphi_A'(-1)^{(2\mu^2/3)+(\mu/3)} \varphi_A'(\infty)^{(2\lambda^2/3)+(4\lambda/3)+5/8}.
\end{align*}
\end{proof}
\begin{remark}
According to (\ref{second-condition}), we have constructed the measures $\P^2(\alpha,\beta,\gamma)$ for $(\alpha,\beta,\gamma)\in\R^3$ such that
\begin{align}\label{range:two-sided}
\alpha\ge\frac58,\quad\beta\ge 0,\quad \gamma>0,\quad \beta<\tilde\xi(\alpha,\gamma),\quad \gamma\le\tilde\xi(\alpha,\beta), \quad \tilde\xi(\alpha,\beta,\gamma)\ge 1.
\end{align}
\end{remark}

The Construction \ref{cons:two-sided} together with the geometric properties of one-sided restriction measures mentioned in Proposition \ref{prop:geom-one-sided} immediately imply some geometric properties of two-sided restriction samples. We list the two following properties without mentioning other obvious ones.
\begin{proposition}\label{lem:goemetry2}
Let $K$ be a sample which has law $\P^2(\alpha,\beta,\gamma)$ where $\alpha,\beta,\gamma$ are within the range (\ref{range:two-sided}), see Figure \ref{fig:construction-two-sided}. Then:
\begin{itemize}
\item[-]
The point $0$ is a cut-point of $K$ if and only if $\gamma=\tilde\xi(\alpha,\beta)$.
\item[-]
There exist a point $X\in(-1,0)$, which is the unique triple disconnecting point of $K$, if and only if $\tilde\xi(\alpha,\beta,\gamma)=1$. 
\end{itemize}
\end{proposition}
\begin{proof}
When $\gamma=\tilde\xi(\alpha,\beta)$, we have $\gamma\ge 5/8$ and $\gamma_0$ intersects $\R$ only at $0$. Moreover, $\gamma=\tilde\xi(\alpha,\beta)$ implies $n=\tilde\xi(m,l)$, hence $-1$ is on the right boundary of $\tilde K$, consequently $0$ is on the left boundary of $\phi(\tilde K)$, hence must be a cut point.

When  $\tilde\xi(\alpha,\beta,\gamma)=1$, since $\alpha\ge5/8$ we have $\gamma< 1/3$ hence $\gamma_0$ is boundary touching and $X\in(0,1)$. $X$ is obviously on the right boundary of $K$. When $\tilde\xi(\alpha,\beta,\gamma)=1$, we actually have
$n=\tilde\xi(m,l)$. Hence $-1$ is on the right boundary of $\tilde K$. Hence $X$ is on the left boundary of $\phi(\tilde K)$ hence it is on both the left and right boundaries of $K$, hence is a triple disconnecting point. {It is obviously unique.}
\end{proof}

In Figure \ref{fig:two-sided-geom} we illustrate some of the above-mentioned geometric properties for some two-sided restriction measures in the unit disk (which is a more symmetric domain). The points $a,b,c$ have respectively restriction exponents $\alpha,\beta,\gamma$.

\begin{figure}[h]
\centering
\begin{subfigure}[t]{0.3\textwidth}
\centering
\includegraphics[width=0.8\textwidth]{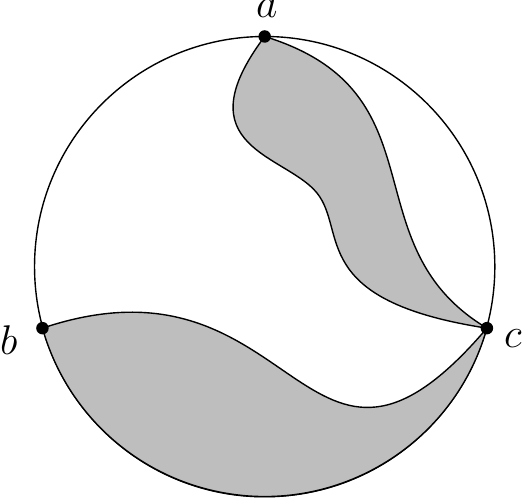}
\caption{The case $\gamma=\tilde \xi (\alpha,\beta)$. 
The point $c$ is a cut-point.
}
\label{fig:two-sided2}
\end{subfigure}
\quad
\begin{subfigure}[t]{0.3\textwidth}
\centering
\includegraphics[width=0.8\textwidth]{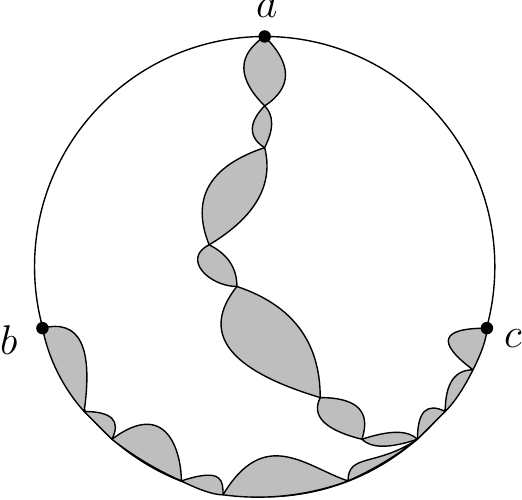}
\caption{The case $\tilde\xi(\alpha,\beta,\gamma)=1,\, \beta,\gamma>0$.
We always have $\alpha\le1$ and $\beta,\gamma< 1/3$,  hence each branch has infinitely many cut-points.}
\label{fig:two-sided1}
\end{subfigure}
\quad
\begin{subfigure}[t]{0.3\textwidth}
\centering
\includegraphics[width=0.8\textwidth]{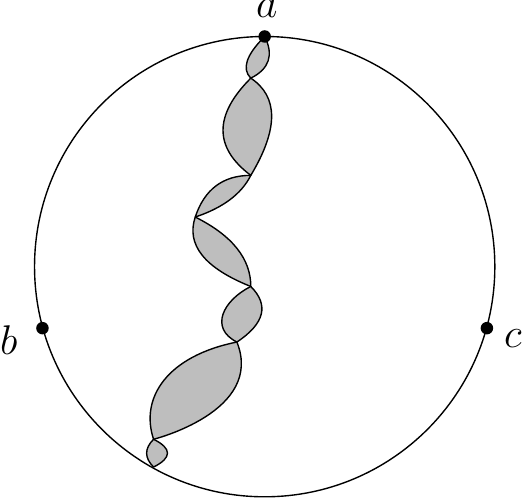}
\caption{The case $\alpha=1, \beta=\gamma=0$, which is not in range (\ref{range:two-sided}).
It is the filling of a Brownian excursion which starts at $a$ and exits at the arc $(bc)$.}
\label{fig:two-sided3}
\end{subfigure}
\caption{Some two-sided restriction samples in the unit disk.}
\label{fig:two-sided-geom}
\end{figure}

\subsection{Construction of limiting cases}
We can see that (\ref{range:two-sided}) {does not seem to be the full range}, because some limiting cases are missing, such as $\gamma=0$ and $\beta=\tilde\xi(\alpha,\gamma)$.

When $\gamma=0$, the right-boundary of the restriction sample does not start at $0$ and thus cannot be given by a hSLE curve which is defined to start at $0$  .
However since the case $\beta=0, \gamma>0$ is included in (\ref{range:two-sided}), by symmetry between $\beta$ and $\gamma$, the case $\beta>0, \gamma=0$ should also exist and can be constructed in a symmetric way, i.e. by first constructing its left boundary which is a hSLE curve starting at $-1$.

For the $\beta=\tilde\xi(\alpha,\gamma)$ case, we can either use again the SLE$_{8/3}(\rho_1,\rho_2)$ curves as in \S\,\ref{sec:limit}, or say it is symmetric to the $\gamma=\tilde\xi(\alpha,\beta)$ case.

The $\beta=\gamma=0$ case is still not constructed. However this case should exist. Consider a Brownian excursion in $\H$ starting from $\infty$ and exiting at a uniformly chosen point of the segment $[-1,0]$, then its outer boundary should satisfy two-sided restriction with exponents $(1,0,0)$ 
(see Figure \ref{fig:two-sided3}).
The union of $n$ independent such Brownian excursions should satisfy restriction with exponents $(n,0,0)$.

Now let us construct $\P^2(\alpha,0,0)$ samples for all $\alpha>1$. As illustrated in Figure \ref{fig:alpha00}:
For $\tau\in(-1,0)$, let $I_\tau$ be a chordal restriction sample of exponent $\alpha> 1$ in $(\H,\tau,\infty)$. Let $\mathcal{P}_\tau(\theta)$ be a Poisson point process of Brownian excursions of intensity 
$\theta \nu_\H$ but restricted to the excursions which start in $(-1,\tau)$ and end in $(\tau,0)$, so that it satisfies restriction in $(\H,-1,\tau,0)$ for exponent $(0,1-\alpha,0)$. This is possible because $1-\alpha<0$ and we have Lemma \ref{sub-family-1}. Let $K_\tau:=\mathcal{F}\left( I_\tau \cup \mathcal{P}_\tau(\theta) \right)$  be the filling of the union of $I_\tau$ and $ \mathcal{P}_\tau(\theta)$.
Let $\rho$ be the probability density function supported on $[-1,0]$ defined by
\begin{align*}
\rho(\tau)=\frac{\tau^{\alpha-1} (\tau+1)^{\alpha-1}}{M} \quad\text{with}\quad 
M=\int_{-1}^0 \tau^{\alpha-1} (\tau+1)^{\alpha-1} d\tau.
\end{align*}
Now let $K$ be the random set constructed by first choosing $T$ according to the density $\rho$ and then a (conditionally) independent $K_\tau$ for $\tau=T$.

\begin{figure}[h]
\centering
\includegraphics[width=0.39\textwidth]{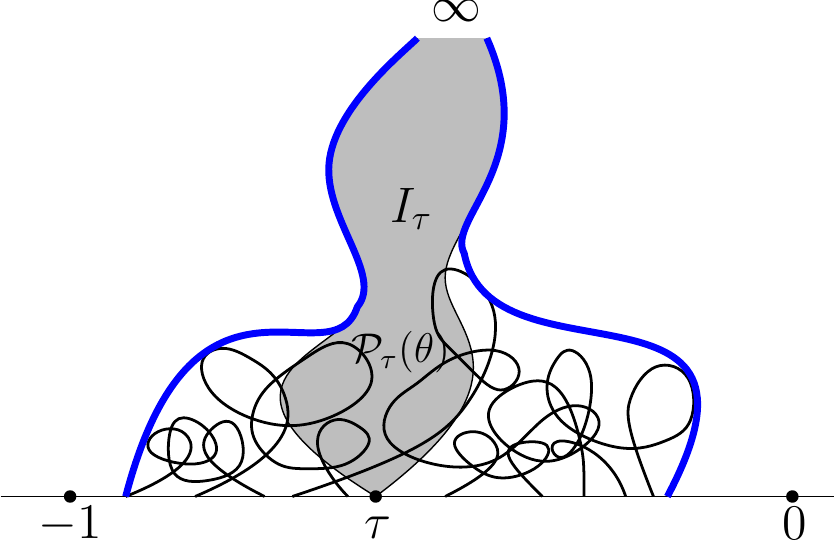}
\caption{The random set $K_\tau$.}
\label{fig:alpha00}
\end{figure}

\begin{lemma}
The law of this set $K$ is $\P^2(\alpha,0,0)$ (see Figure \ref{fig:alpha00}).
\end{lemma}

\begin{proof}
Note that the map $z\mapsto \varphi_A(z)-\varphi_A(\tau)$ is a conformal map from $\H\setminus A$ to $\H$ that fixes $\tau$ and $\infty$. Hence
\begin{align*}
\P\left( I_\tau \cap A=\emptyset \right)=\varphi_A'(\tau)^\alpha \varphi_A'(\infty)^\alpha.
\end{align*}
Note that the conformal map from $\H\setminus A$ to $A$ that fixes $-1,\tau,0$ is given by
\begin{align*}
f_{\tau,A}(z)=\frac{\tau(\varphi_A(\tau)+1)}{\varphi_A(\tau)-\tau}-\frac{\tau\varphi_A(\tau)(\tau+1)(\varphi_A(\tau)+1)}{(\varphi_A(\tau)-\tau)^2\varphi_A(z)+\varphi_A(\tau)(\varphi_A(\tau)-\tau)(\tau+1)}.
\end{align*}
It yields that
\begin{align*}
f_{\tau,A}'(\tau)=\frac{\tau(\tau+1)}{\varphi_A(\tau)\left( \varphi_A(\tau)+1 \right)} \varphi_A'(\tau).
\end{align*}
For all $A\in\mathcal{Q}^{(-\infty,-1)\cup(0,\infty)}$, we have
\begin{align*}
\P( \mathcal{P}_\tau(\theta) \cap A=\emptyset)=f_{\tau,A}'(\tau)^{1-\alpha}.
\end{align*}
Since $ \mathcal{P}_\tau(\theta)$ and $I_\tau$ are independent, we further have
\begin{align*}
\P\left( K_\tau \cap A=\emptyset \right)&=\P\left( I_\tau \cap A=\emptyset \right) \P( \mathcal{P}_\tau(\theta) \cap A=\emptyset)\\
&=\varphi_A'(\infty)^\alpha \varphi_A'(\tau) \left( \frac{\tau(\tau+1)}{\varphi_A(\tau)\left( \varphi_A(\tau)+1 \right)} \right)^{1-\alpha}.
\end{align*}
Hence
\begin{align*}
\P\left( K\cap A=\emptyset \right)&= \int_{-1}^0 \P\left( K_\tau \cap A=\emptyset \right) \rho(\tau) d\tau \\
&=\frac1M \int_{-1}^0 \varphi_A'(\infty)^\alpha \varphi_A'(\tau) \left( \frac{\tau(\tau+1)}{\varphi_A(\tau)\left( \varphi_A(\tau)+1 \right)} \right)^{1-\alpha} \tau^{\alpha-1} (\tau+1)^{\alpha-1} d\tau \\
&= \frac1M \varphi_A'(\infty)^\alpha  \int_{-1}^0 \left(\varphi_A(\tau)\left( \varphi_A(\tau)+1 \right)\right)^{\alpha-1} \varphi_A'(\tau) d\tau.
\end{align*}
Since $\varphi_A$ sends $-1,0$ to themselves, we can view $\tau\mapsto \varphi_A(\tau)$ as a change of variable on $[-1,0]$. Hence
\begin{align*}
\P\left( K\cap A=\emptyset \right)&= \frac1M \varphi_A'(\infty)^\alpha  \int_{-1}^0 \left( \tau(\tau+1) \right) ^{\alpha-1} d\tau
= \varphi_A'(\infty)^\alpha.
\end{align*}
\end{proof}

\begin{remark}
By adding these limiting cases to what have been constructed in the previous section, we have constructed $\P^2(\alpha,\beta,\gamma)$ for all $(\alpha,\beta,\gamma)\in\R^3$ in the following domain (which is an extension of (\ref{range:two-sided}))
\begin{align}\label{two-sided-final-range}
\alpha\ge\frac58,\quad\beta\ge 0,\quad \gamma\ge0,\quad \beta\le\tilde\xi(\alpha,\gamma),\quad \gamma\le\tilde\xi(\alpha,\beta), \quad \tilde\xi(\alpha,\beta,\gamma)\ge 1.
\end{align}
\end{remark}

\subsection{Range of existence for two-sided measures}
In this section we are going to prove that (\ref{two-sided-final-range}) is actually the largest range for which $\P^2(\alpha,\beta,\gamma)$ can exist.

The conditions $\beta\ge 0, \gamma\ge0,\beta\le\tilde\xi(\alpha,\gamma), \gamma\le\tilde\xi(\alpha,\beta)$ are inherited from the one-sided range (\ref{one-sided-final-range}), because the left and right boundaries of a two-sided restriction measure should both satisfy one-sided restriction.

The condition $\alpha\ge {5}/{8}$ is inherited from the chordal two-sided restriction measures. 
If $K$ is a sample with law $\P^1(\alpha,\beta,\gamma)$, then the sequence $(K/n)_{n\ge 0}$ ($K$ rescaled by the factor $1/n$) admits a subsequence that converges in law to a two-sided chordal restriction measure $K_\infty$, due to some standard facts about tightness and weak convergence of {measures}. Moreover , for $A\in\mathcal{Q}^{(0,\infty)}$, by approaching the event $\{K_\infty \cap A=\emptyset\}$ by open and closed sets  (for example see  \S\, 4.2 \cite{MR3293294}), we have 
\begin{equation}\label{kacvg}
\begin{split}
\P(K_\infty\cap A=\emptyset)&=\P\left(\frac1nK\cap A=\emptyset\right)\\
&=g_A'\left(-\frac1n\right)^\beta g_A'(0)^\gamma \left(g_{A}(0)-g_{A}\left(-\frac1n\right)\right)^{\alpha-\beta-\gamma}
\underset{n\to \infty}{\longrightarrow} g_A'(0)^\alpha.
\end{split}
\end{equation}
Hence we must have $\alpha\ge5/8$.

Now it is only remains to prove the following.
\begin{lemma}
The measure
$\P^2(\alpha,\beta,\gamma)$ does not exist if $\tilde\xi(\alpha,\beta,\gamma)<1$.
\end{lemma}
\begin{proof}
If there exist a sample $K_1$ with law $\P^2(\alpha,\beta,\gamma)$ for some $(\alpha,\beta,\gamma)$ such that 
\begin{align*}
\alpha\ge\frac58,\quad\beta\ge 0,\quad \gamma\ge0,\quad \beta\le\tilde\xi(\alpha,\gamma),\quad \gamma\le\tilde\xi(\alpha,\beta),
\end{align*}
but
$\tilde\xi(\alpha,\beta,\gamma)<1$, then since $\tilde\xi$ is continuous and strictly increasing in its arguments, there exist $\delta>0$ such that $\tilde\xi(\alpha,\beta+\delta,\gamma+\delta)=1$. Let $K_0$ be an independent one-sided chordal restriction sample of exponent $\delta$ in $(\H,-1,0)$. Then $K:=\mathcal{F}(K_1\cup K_0)$ has law $\P^2(\alpha,\beta+\delta,\gamma+\delta)$. 
This construction shows that the left and right boundaries of $K$ have no common point on $[-1,0]$. However $(\alpha,\beta+\delta,\gamma+\delta)$ is in the range (\ref{range:two-sided}), and according to Proposition
\ref{lem:goemetry2} $K$ should have a triple disconnecting point on $[-1,0]$. This is a contradiction.
\end{proof}

\section{Three-sided restriction}\label{sec:three-sided}
We are now finally going to be able to determine which of the three-sided restriction measures exist described via Proposition \ref {prop:charac} do indeed exist. In order to distinguish them from the two-sided and one-sided ones, we will refer to them as $\P^3 ( \alpha, \beta, \gamma)$  instead of $\P( \alpha, \beta, \gamma)$ in the present section. 

\subsection{Construction by hSLE curves and two-sided samples}\label{sec:construction3}
We still consider the same martingale $(M_t, t\ge 0)$ as before, but this time for all $A\in\mathcal{Q}^{**}$.

In order for  $(M_t, t\ge 0)$ to be a bounded martingale, we need to restrict $\lambda,\mu,\nu$ to a range which is even smaller than (\ref{second-condition}). Let us give the range right now.
\[
\left|
\begin{array}{lllllll}
\lambda\ge 0  &\iff& \a-\b\ge 2 &\iff&\displaystyle{\alpha\ge\frac58} &\iff& l\ge 0\\[2.5mm]
 \displaystyle{\mu\ge \frac34}   &\iff& \a+\b-\c\ge 2 &\iff& \displaystyle{\beta\ge\frac58 }&\iff&\displaystyle{ m\ge \frac58}\\[2.5mm]
\displaystyle{\nu\ge 0 } &\iff& \displaystyle{\c\ge \frac32} &\iff&\displaystyle{ \gamma\ge \frac58} &\iff& n\ge 0\\[2.5mm]
\displaystyle{\mu<\lambda+\nu+\frac32}  &\iff& \b<\c &\iff& \beta<\tilde\xi(\alpha,\gamma) &&\\[2.5mm]
 \nu\le \lambda+\mu  &\iff&\displaystyle{ c\le a-\frac12} &\iff& \gamma\le\tilde\xi(\alpha,\beta)&\iff& n\le\tilde\xi(l,m)\\[2.5mm]
\lambda\le\mu+\nu  &\iff& b\ge 0 &\iff& \alpha\le\tilde\xi(\beta,\gamma) &\iff& l\le\tilde\xi(m,n)\\[2.5mm]
\lambda+\mu+\nu \ge 1 &\iff& \a\ge 3 &\iff& \xi(\alpha,\beta,\gamma)\ge 2 &\iff&\tilde\xi(l,m,n)\ge1.
\end{array}
\label{third-condition}
\tag{Range 3}
\right.
\]
The equivalence here is to be understood in the sense that we assume above all $\lambda\ge0, \mu\ge3/4, \nu\ge 0$. In this case, (\ref{align:bijection}) is true and moreover $\lambda=U(l), \mu=U(m), \nu=U(n)$.

\begin{lemma}\label{lem7.1}
For $a,b,c$ in (\ref{third-condition}), we have $w'(z)>0$ for all $z\in(-\infty,0)$.
\end{lemma}
\begin{proof}
Note that $w'(0)=ab/c>0$. If $w$ is not always strictly increasing on $(-\infty,0)$, then there exist $z_0<0$ such that $w'(z_0)=0$ and $w''(z_0)\ge 0$.
Putting this back to the equation (\ref{euler}) yields
\begin{align*}
z_0(1-z_0)w''(z_0)=abw(z_0).
\end{align*}
The two sides of the equation above should have the same sign, hence we have $w(z_0)\le 0$. This is in contradiction with Lemma \ref{G}.
\end{proof}

\begin{lemma}\label{martingale3}
Let $\lambda,\mu,\nu$ be in (\ref{third-condition}).
For  $A\in\mathcal{Q}^{**}$, there exist a constant $C$, possibly dependent on $\lambda,\mu,\nu$ and $A$, such that $0\le M_t\le C$ for all $t\ge 0$ a.s.
\end{lemma}
\begin{proof}
Since (\ref{third-condition}) is included in (\ref{second-condition}), it is enough to prove the lemma for $A\in\mathcal{Q}^{(-1,0)}$.
Let us keep the notations from the proofs of Lemma \ref{martingale-1-side} and Lemma \ref{martingale2}. 
Recall that $P_{\lambda,\mu,\nu}$ and $q$ were defined by (\ref{P_lmn}) and (\ref{qw}) as
\begin{align*}
P_{\lambda,\mu,\nu}(x,y)&=nx^2+(l-m-n+\lambda-\nu-\mu)xy+my^2+\nu x+\mu y+\frac58,\\
q(z)&=-zw'(-z)/w(-z).
\end{align*}
Here $x$ and $y$ actually represent the quantities $(x_s-w_s)/(x_s-o_s)$ and $(x_s-w_s)/(x_s-v_s)$.

The difference here is that
$v_s\le x_s\le o_s\le w_s$ for $A\in\mathcal{Q}^{(-1,0)}$. 
Hence we need to prove that the quantity
\begin{align*}
f(x,y)=P_{\lambda,\mu,\nu}(x,y)+q\left(\frac{x-1}{y-x}y \right)x(1-y)
\end{align*}
is bounded from below for $x\ge 1, y\le 0$.
Since $f$ is continuous, it is enough to look at $f$ when either $x$ or $y$ goes to $\infty$.

\underline{If $x\to\infty$ and $y$ stays bounded}, then:
\begin{itemize}
\item
If $n>0$,  then $f$ is asymptotically equivalent to the leading term $nx^2$. Hence $f$ is bounded from below. 
\item
If $n=0$, then we also have $\nu=0$ and
\begin{align*}
f(x,y)=(l-m+\lambda-\mu)xy+my^2+\mu y+\frac58 +q\left(\frac{x-1}{y-x}y \right)x(1-y).
\end{align*}
As $x\to\infty$, it is asymptotically equivalent to the leading term
\begin{align*}
\left[(l-m+\lambda-\mu)y +q\left(-y \right)(1-y)\right] x.
\end{align*}
First we can exclude the trivial case where $y=0$ for which $f=5/8$.
Then note that $l-m+\lambda-\mu\le0$ and $y\le 0$. Also note that $q(z)=-zw'(-z)/w(-z)>0$ for all $z>0$ because of Lemma \ref{lem7.1}, Lemma \ref{G}, and $1-y>0$. Hence $$(l-m+\lambda-\mu)y +q\left(-y \right)(1-y)>0.$$ Hence $f$ is bounded from below.
\end{itemize}

\underline{If $x$ stays bounded and $y\to-\infty$}, then $f$ is asymptotically a polynomial in $y$ with leading term $my^2$ where $m> 0$ hence $f$ is bounded from below.

\underline{If $x\to\infty, y\to-\infty$}, then $(x-1)y/(y-x)\to\infty$. In the same way as (\ref{estim6}), we have
\begin{align*}
f(x,y)=nx^2+(l-m-n)xy+my^2+o(xy).
\end{align*}
\begin{itemize}
\item
If $n=0$, then we also have $\nu=0$ and
\begin{align*}
f(x,y)=(l-m)xy+my^2+o(xy).
\end{align*}
Note that in this case we always have $l\le m$.
If $l<m$, then $f$ is obviously bounded from below.
If $l=m$, then we also have $\lambda=\mu$, hence $b=0$, hence $w$ is constantly equal to $1$. Hence $q$ is constantly equal to $0$ and
\begin{align*}
f(x,y)=my^2+\mu y+\frac58.
\end{align*}
This is bounded from below because $m>0$.
\item If $n>0$, then
note that we also have $m>0$. Assume $f$ is not bounded from below, we must have
\begin{align*}
l-m-n>0\\
\Delta':=(l-m-n)^2-4nm>0
\end{align*}
In this case $\Delta'$ is increasing in $l$ hence in $\lambda$.
Since $l\le\tilde\xi(m,n)$, we have
\begin{align*}
l\le U^{-1}\left( U(m)+U(n) \right)=U^{-1}(\mu+\nu).
\end{align*}
This implies that $\Delta'$ gets to its maximum when $l=U^{-1}(\mu+\nu)$. Hence
\begin{align*}
\Delta'\le -\frac49\mu\nu(2\mu+2\nu+1)\le 0
\end{align*}
which leads to a contradiction.
Hence $f$ is bounded from below.
\end{itemize}
\end{proof}

Now we can construct a three-sided restriction sample $K$ by the following steps, see Figure \ref{fig:construction-three-sided}. 
\begin{construction}\label{construction3}
Let $\lambda,\mu,\nu$ be inside  (\ref{third-condition}).
\begin{itemize}
\item[(i)]  Let $\gamma_0$ be a hSLE$(\lambda,\mu,\nu)$ from $0$ to $\infty$.  Note that since $\gamma\ge-1/4$, according to Proposition \ref{prop:hsle} $\gamma_0$ does not hit the real axis except at $0$.
Let $\H^-$ be the connected component of $\H\setminus\gamma_0$ which is to the left of $\gamma_0$.
\item[(ii)] Let $\tilde K$ be an independent sample $\P^2(m,n,l)$. Note that
 (\ref{third-condition}) implies that
 $\tilde K$ exists.
 Let $\phi$ be the conformal map from $\H$ to $\H^-$ that sends the boundary points $0,-1,\infty$ to $\infty,0,-1$.
Let $K:=\phi(\tilde K)$.
\end{itemize}
\end{construction}

\begin{figure}[h]
\centering
\includegraphics[width=0.68\textwidth]{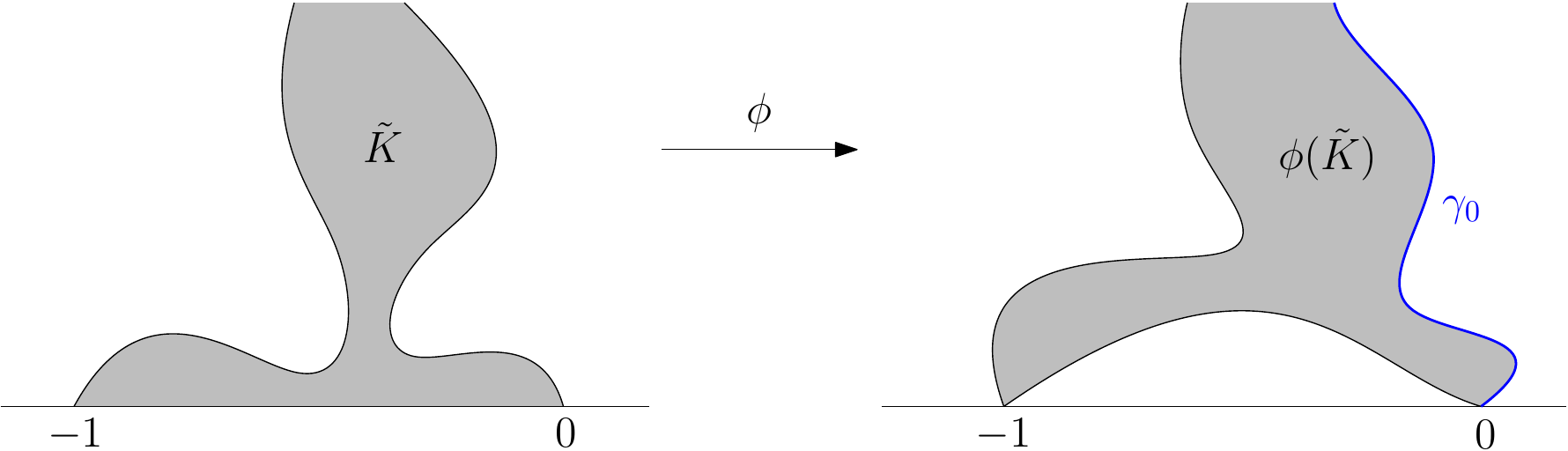}
\caption{Construction of three-sided restriction measures.}
\label{fig:construction-three-sided}
\end{figure}

\begin{proposition}
The above-constructed $K$ has law $\P^3(\alpha,\beta,\gamma)$, where
\begin{align*}
\alpha=\frac{2}{3}\lambda^2+\frac{4}{3}\lambda+\frac{5}{8}, \quad
\beta=\frac{2}{3}\mu^2+\frac{1}{3}\mu, \quad
\gamma=\frac{2}{3}\nu^2+\frac{4}{3}\nu+\frac{5}{8}.
\end{align*}
\end{proposition}
\begin{proof}
The proof is identical to that of Proposition \ref{prop:two-sided}, except that we also need to check for $A\in\mathcal{Q}^{(-1,0)}$.\end{proof}

\begin{remark}
We constructed all three-sided restriction samples of law $\P^3(\alpha,\beta,\gamma)$ for $\alpha,\beta,\gamma$ in (\ref{third-condition}).
Note that the limiting case $\beta=\tilde\xi(\alpha,\gamma)$ can be either constructed using SLE$_{8/3}(\rho_1,\rho_2)$ curves as in \S\,\ref{sec:limit}, or alternatively in a symmetric way as the cases $\gamma=\tilde\xi(\alpha,\beta)$ and $\alpha=\tilde\xi(\beta,\gamma)$. Thus, we have now constructed (and proved the existence of) the measures $\P^3(\alpha,\beta,\gamma)$ for $(\alpha,\beta,\gamma)\in\R^3$ such that
\begin{align}\label{three-sided-final-range}
\alpha\ge\frac58,\quad\beta\ge \frac58,\quad \gamma\ge \frac58,\quad \beta\le\tilde\xi(\alpha,\gamma),\quad \gamma\le\tilde\xi(\alpha,\beta),\quad \alpha\le \tilde\xi(\beta,\gamma), \quad \xi(\alpha,\beta,\gamma)\ge 2.
\end{align}
\end{remark}

The Construction \ref{construction3} together with the geometric properties of two-sided restriction measures mentioned in Proposition \ref{lem:goemetry2}
 immediately imply some geometric properties of three-sided restriction samples. The following proposition is also a part of Theorem \ref{thm:existence}.
 \begin{proposition}
 Let $K$ be a sample which has law  $\P^3(\alpha,\beta,\gamma)$ where $\alpha,\beta,\gamma$ are within range (\ref{three-sided-final-range}).
 \begin{itemize}
 \item[-] $K$ has a triple disconnecting point which is unique if and only if $\xi(\alpha,\beta,\gamma)=2$.
 \item[-] The boundary point $0$ is a cut point of  $K$ if and only if $\gamma=\tilde\xi(\alpha,\beta)$.
 \end{itemize}
 \end{proposition}
We illustrate these geometric properties in Figure \ref{fig:existence} for samples in the unit disk.
 \begin{proof}
 According to Construction \ref{construction3}, $\xi(\alpha,\beta,\gamma)=2$ if and only if $\tilde\xi(m,n,l)=1$. By Proposition \ref{lem:goemetry2}, the set $\tilde K$ has a unique triple disconnecting point $X$ on the boundary $(-1,0)$. This point will be mapped by $\phi$ to a point on $\gamma_0$, which will be the unique triple disconnecting point of $K$.
 
 When $\gamma=\tilde\xi(\alpha,\beta)$, then $m=\tilde\xi(l,n)$ hence $-1$ is a cut point of $\tilde K$, hence $0$ will be a cut point of $K$.
 \end{proof}

\subsection{Range of existence for three-sided restriction measures}\label{sec:range-3}

The range (\ref{three-sided-final-range}) is the same as in Theorem \ref{thm:existence}. In this section, we will prove the following proposition.

\begin{proposition}\label{prop:range-3}
Three-sided trichordal restriction measures $\P^3(\alpha,\beta,\gamma)$ do not exist for $(\alpha,\beta,\gamma)$ outside of range (\ref{three-sided-final-range}).
 \end{proposition}

It is easy to see that, in  range (\ref{three-sided-final-range}), the conditions $\alpha\ge5/8,\,\beta\ge 5/8,\,\gamma\ge 5/8$  are inherited from chordal two-sided restriction measures and the conditions $\beta\le\tilde\xi(\alpha,\gamma),\, \gamma\le\tilde\xi(\alpha,\beta),\, \alpha\le \tilde\xi(\beta,\gamma)$  are inherited from trichordal one-sided restriction measures. Hence it is enough to justify the condition $ \xi(\alpha,\beta,\gamma)\ge 2$. As we are now going to explain, this condition is actually inherited from the two-sided condition $\tilde\xi(\alpha,\beta,\gamma)\ge 1$ in (\ref{two-sided-final-range}). 
For this purpose, we are going to explore the relation between two-sided and three-sided restriction measures. We can already see that if we take $K$ a three-sided restriction measure as in Construction \ref{construction3} and send the right boundary of $K$ to the half-line $(0,\infty)$ by the conformal map $\phi^{-1}$, then the image $\phi^{-1}(K)$ is a two-sided restriction measure.

We would like to prove this transitivity of restriction property in a way that does not rely on the Construction \ref{construction3}. This is because we want to know whether there exists any other $\P^3(\alpha,\beta,\gamma)$ than those constructed in Construction \ref{construction3}.
Let us reformulate the desired property and fix notations for the rest of this section.

Let $\gamma_b$ be the bottom part of $K$ which has $-1$ and $0$ as end points. Let $\varphi_{\gamma_b}$ be the conformal map from the unbounded connected component of $\H\setminus\gamma_b$ to $\H$ which sends $-1,0,\infty$ to themselves. Let $J:=\varphi_{\gamma_b}(K)$. (See Fig \ref{relation-multi}.) We will prove in Lemma \ref{lem:J} that $J$ satisfies two-sided restriction.

\begin{figure}[h]
    \centering
    \includegraphics[width=0.8\textwidth]{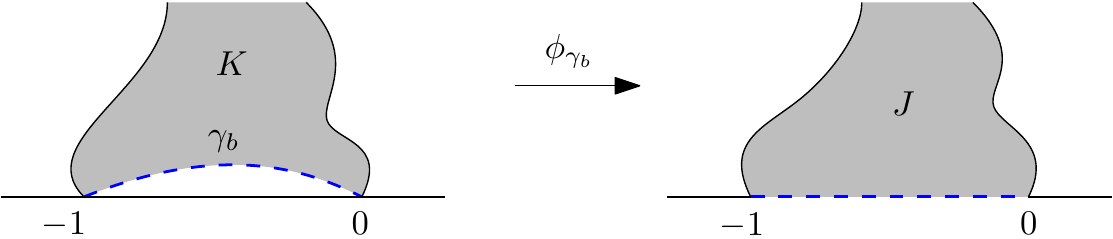}
    \caption{From three-sided restriction to two-sided restriction by applying conformal maps}
    \label{relation-multi}
\end{figure}

\begin{lemma}\label{indep}
 $J$ and $\gamma_b$ are independent. 
\end{lemma}

\begin{proof}
For all $A\in\mathcal{Q}^{(-1,0)}$, because of the restriction property of $K$, we have the following identity in law:
\begin{align}\label{lem: indep}
\varphi_{\gamma_b}(K) | \{K\cap A=\emptyset \}
 = \varphi_{\varphi_A(\gamma_b)}\circ \varphi_A(K) | \{K\cap A=\emptyset\}
=\varphi_{\gamma_b}(K).
\end{align}
Equivalently, (\ref{lem: indep}) says that 
$
J|\{\gamma_b \cap A=\emptyset\} 
$
has the same law as $J$. Since the events $\{\gamma_b \cap A=\emptyset\} $ for all $A\in\mathcal{Q}^{(-1,0)}$ generate the $\sigma$-algebra for the law of $\gamma_b$, we conclude that $J$ is independent of $\gamma_b$.
\end{proof}

Lemma \ref{indep} says that in order to sample $K$, we can first sample $\gamma_b$ and then sample the pre-image by $\varphi_{\gamma_b}$ of an independent sample $J$. 
This gives us an alternative way to describe the law of $J$ : it is the law of $K$ conditioned on $\gamma_b=[-1,0]$. Of course this conditioning must be taken as a limit.
Let $[-1,0]^\eps :=\left\{z\in\overline\H, \dist(z,[-1,0])\le\eps\right\}.$ Let $K_\eps$ be $K$ conditioned on  $\gamma_b\subset[-1,0]^\eps$. 
\begin{lemma}
The sequence $(K_\eps)_{\eps>0}$ converges in law to $J$ as $\eps\to 0$.
\end{lemma}
\begin{proof}
Let $H$ be the space of continuous curves from $-1$ to $0$ that stay in $\H$ except for the two extremities. Let $\Omega^2(\H,-1,0,\infty)$ be as in the definition of the two-sided measures. Let $\Omega(\H,-1,0,\infty)$ be as in the definition of our three-sided measures. Let $\mathcal{G}$ be a function from $H\times \Omega^2(\H,-1,0,\infty)$ to $\Omega(\H,-1,0,\infty)$ that sends $(\gamma, J_0)$ to the pre-image by $\varphi_\gamma$ of $J_0$. 
Let $f$ be a M\"obius transformation from the unit disk $\mathbb{U}$ to $\H$. Endow $H$, $\Omega^2(\H,-1,0,\infty)$ and $\Omega(\H,-1,0,\infty)$ with the metric which is the image by $f$ of the Hausdorff metric on $\mathbb{U}$, then $\G$ is continuous in both of its arguments.

By definition $K_\eps$ has the law of $\mathcal{G}(\tilde\gamma^\eps_b, J)$ where $\tilde\gamma^\eps_b$ is $\gamma_b$ conditioned on $\gamma_b\subset[-1,0]^\eps$ and $J$ is independent of $\tilde\gamma_b$. Since $\tilde\gamma^\eps_b$ converges to the interval $[-1,0]$ almost surely within Hausdorff distance, we have $K_\eps$ converges in law to $\mathcal{G}([-1,0], J)=J$.
\end{proof}

\begin{lemma}\label{lem:J}
$J$ satisfies two-sided restriction.
\end{lemma}

\begin{proof}
Let $F$ be a measurable continuous and bounded function on $\Omega^2(\H,-1,0,\infty) \cup \Omega(\H,-1,0,\infty)$. Let $A\in\mathcal{Q}^{(-\infty,-1)\cup(0,\infty)}$. We have
\begin{align*}
&\E\left( F(\varphi_A(K_\eps)) | K_\eps\cap A=\emptyset \right)
=\frac{\E\left(F(\varphi_A(K_\eps))\mathbf{1}_{K_\eps\cap A=\emptyset}\right)}{\P(K_\eps\cap A=\emptyset)}
\\
=&\frac{\E\left(F(\varphi_A(K))\mathbf{1}_{K\cap A=\emptyset} | \gamma_b\subset[-1,0]^\eps\right)}{\P(K\cap A=\emptyset | \gamma_b\subset[-1,0]^\eps)}
=\frac{\E\left(F(\varphi_A(K))\mathbf{1}_{K\cap A=\emptyset} \mathbf{1}_{ \gamma_b\subset[-1,0]^\eps}\right)}{\P\left(K\cap A=\emptyset , \gamma_b\subset[-1,0]^\eps\right)}
\\
=&\frac{\E\left(F(\varphi_A(K)) \mathbf{1}_{ \gamma_b\subset[-1,0]^\eps} | {K\cap A=\emptyset}\right)}{\P(\gamma_b\subset[-1,0]^\eps | K\cap A=\emptyset)}
=\frac{\E\left(F(K) \mathbf{1}_{ \varphi_A^{-1}(\gamma_b)\subset[-1,0]^\eps} \right)}{\P\left(\varphi_A^{-1}(\gamma_b)\subset[-1,0]^\eps \right)}
\\
=&\E\left(F(K) | { \varphi_A^{-1}(\gamma_b)\subset[-1,0]^\eps} \right).
\end{align*}

On the one hand, by the fact that $K_\eps$ converges to $J$ and by some open and closed approximations as needed for (\ref{kacvg}), we have 
\begin{align*}
\E\left( F(\varphi_A(K_\eps)) | K_\eps\cap A=\emptyset \right)\underset{\eps\to 0}{\longrightarrow} \E\left( F(\varphi_A(J)) | J\cap A=\emptyset \right).
\end{align*}
On the other hand, 
$$\E\left(F(K) | { \varphi_A^{-1}(\gamma_b)\subset[-1,0]^\eps} \right)\underset{\eps\to 0}{\longrightarrow} \E\left(F(J) \right)$$
for the same reason why $K_\eps$ converges in law to $J$.
Hence we have $$ \E\left( F(\varphi_A(J)) | J\cap A=\emptyset \right)= \E\left(F(J) \right)$$ and $J$ satisfies two-sided restriction.
\end{proof}

{Note that for two-sided (trichordal) restriction measures, we can do a similar conformal transformation to obtain a one-sided (trichordal) restriction measure}; this follows either via the same proof, or directly from our Construction \ref{cons:two-sided}.  
Moreover Construction \ref{cons:two-sided} gives us a  precise relation between the exponents.
\begin{lemma}\label{lem:aaa}
Let $K$ be a sample of law $\P^2(\alpha,\beta,\gamma)$ where $\beta,\gamma\ge5/8$ and let $\gamma_r$ be its right boundary. Then $\varphi_{\gamma_r}(K)$ has law $\P^1(\alpha, h(\beta),h(\gamma))$ where
\begin{align*}
h(x):=U^{-1}\left( U(x)-\frac34 \right)  \quad\text{ for } x\ge\frac58.
\end{align*}
\end{lemma}
Recall that $U$, as defined in \S\,\ref{sec:brownian-exponents}, can be viewed as a bijection from $[0,\infty)$ onto itself,  and that for $x\ge5/8$, we have $U(x) \ge 3/4$ so that there are no definition problems here.
Furthermore, as $U(h(x))+U(5/8)=U(x)$, it follows that $h(x)$ can be equivalently described as the only non-negative real such that
$ x = \tilde \xi (5/8, h(x))$. 
Loosely speaking, when we unfold a corner by the conformal maps as we defined, the restriction exponent $\alpha\ge5/8$ goes to $h(\alpha)$ for that corner.
There is also a relation for the case $\gamma<5/8$ but we will not need it.
The same holds for chordal two-sided measures. It was proved in \cite{MR1992830, MR2060031} that we can construct all chordal restriction measure of exponent $\alpha$ by first constructing its right boundary $\gamma$ then putting an independent one-sided chordal restriction sample of exponent 
$h(\alpha)$
in the connected component of $\H\setminus \gamma$ which contains $\R^-$.

 \begin{proof}
 [Proof of Proposition \ref{prop:range-3}]
Let $(\alpha,\beta,\gamma)\in\R^3$ be such that
\begin{align*}
\alpha\ge\frac58,\quad\beta\ge \frac58,\quad \gamma\ge \frac58,\quad \beta\le\tilde\xi(\alpha,\gamma),\quad \gamma\le\tilde\xi(\alpha,\beta),\quad \alpha\le \tilde\xi(\beta,\gamma).
\end{align*}
Assume that the law $\P(\alpha,\beta,\gamma)$ exists and let $K$ be a sample, then by Lemma \ref{lem:J}, $J=\varphi_{\gamma_b}(K)$ is a two-sided restriction measure. We can in particular view $K$ as a two-sided restriction measure and apply Lemma \ref{lem:aaa}, so $J$  has law $\P^2(\alpha, h(\beta), h(\gamma))$.
Therefore $(\alpha, h(\beta), h(\gamma))$ must fall into the range (\ref{two-sided-final-range}).
Hence $\tilde\xi(\alpha,h(\beta),h(\gamma))\ge 1$ which means that
\begin{align*}
U(\alpha)+U(h(\beta))+U(h(\gamma))\ge U(1)=1.
\end{align*}
which can be rephrased as
\begin{align*}
U(\alpha)+U(\beta)+U(\gamma)\ge\frac52
\end{align*}
which is equivalent to
$
\xi(\alpha,\beta,\gamma)\ge 2$.
\end{proof}

\appendix
\section{It\^o calculus for the proof of Lemma \ref{martingale-1-side}} \label{appendix}
We need to show that the process
\begin{align*}
M_t=h_t'(W_t)^{5/8}h_t'(O_t)^{n}h_t'(V_t)^m
\left(\frac{h_t(O_t)-h_t(V_t)}{O_t-V_t}\right)^{l-m-n+\lambda}
{G\left(\frac{h_t(W_t)-h_t(O_t)}{h_t(O_t)-h_t(V_t)}\right)}\Biggm/{G\left(\frac{W_t-O_t}{O_t-V_t}\right)}
\end{align*}
is a local martingale, where $(W_t, O_t, V_t)$ satisfy
\begin{equation*}
\left\{
\begin{split}
&\; dO_t=\frac{2dt}{O_t-W_t},
\quad  dV_t=\frac{2dt}{V_t-W_t},\\
&\; dW_t=\sqrt{\kappa}\dif B_t+J(W_t-O_t,W_t-V_t)dt. \\
\end{split}
\right.
\end{equation*}

This looks like a daunting task, but it is a rather straightforward direct It\^o formula computation: 
To simplify the notations, set
\begin{align*}
\tilde X_t := \frac{h_t(W_t)-h_t(O_t)}{h_t(O_t)-h_t(V_t)},\quad X_t:=\frac{W_t-O_t}{O_t-V_t}, \hbox { and } J_t:= J(W_t-O_t,W_t-V_t).
\end{align*}
First calculate the following It\^o derivatives:
\begin{align*}
&d\, h_t(W_t)= \left[-\frac53 h_t''(W_t)+h_t'(W_t) J_t \right] dt +\sqrt{\frac83} h_t'(W_t) d B_t\\
&d\, h_t'(W_t) =\left[ \frac{h_t''(W_t)^2}{2h_t'(W_t)} +h_t''(W_t) J_t \right] dt +\sqrt{\frac83} h_t''(W_t) d B_t\\
&d\, h_t(O_t)=\frac{2h_t'(W_t)^2}{h_t(O_t)-h_t(W_t)} dt, \quad  h_t(V_t)=\frac{2h_t'(W_t)^2}{h_t(V_t)-h_t(W_t)} dt\\
&d\, h_t'(O_t)=\left[-\frac{2h_t'(W_t)^2 h_t'(O_t)}{(h_t(O_t)-h_t(W_t))^2} +\frac{2h_t'(O_t)}{(O_t-W_t)^2} \right] dt\\
&d\, h_t'(V_t)=\left[-\frac{2h_t'(W_t)^2 h_t'(V_t)}{(h_t(V_t)-h_t(W_t))^2} +\frac{2h_t'(V_t)}{(V_t-W_t)^2} \right] dt.
\end{align*}
Then we have
\begin{eqnarray*}
\frac{d M_t}{M_t} &= & d\, \log M_t +\frac{1}{2M_t^2} d\langle M\rangle_t\\
&=&\frac58 \frac{d h_t'(W_t)}{h_t'(W_t)} -\frac{5h_t''(W_t)^2}{6 h_t'(W_t)^2} dt
+ n \frac{dh_t'(O_t)}{h_t'(O_t)} +m \frac{dh_t'(V_t)}{h_t'(V_t)} \\
&&+ (l-m-n+\lambda) \frac{d h_t(O_t)-d h_t(V_t) }{{h_t(O_t)-h_t(V_t)}} -(l-m-n+\lambda) \frac{d O_t -d V_t}{{O_t-V_t}} \\
&&+ \frac{d\,{G(\tilde X_t)}}{G(\tilde X_t)}
-\frac{4G'(\tilde X_t)^2}{3G(\tilde X_t)^2} \left( \frac{h_t'(W_t)}{h_t(O_t)-h_t(V_t)} \right)^2 dt
\\
&&- \frac{d\, {G(X_t)}}{G(X_t)} 
+\frac{4G'(X_t)^2}{3G(X_t)^2}\frac{1}{\left( O_t-V_t \right)^2} dt
+ \frac{1}{2M_t^2} d\langle M\rangle_t.
\end{eqnarray*}
From the formula above, we see that the local martingale term of $d M_t/ M_t$ is $L_t d B_t$  (and therefore $d\langle M \rangle _t =M_t^2 L_t^2 dt$), where
\begin{align*}
L_t=\frac58 \sqrt{\frac83}\frac{h_t''(W_t)}{h_t'(W_t)} +\frac{G'(\tilde X_t)}{G(\tilde X_t)} \sqrt{\frac83} \frac{h_t'(W_t)}{h_t(O_t)- h_t(V_t)}
-\frac{G'(X_t)}{G(X_t)}\sqrt{\frac83} \frac{1}{O_t-V_t}.
\end{align*}
The drift term of $d M_t/ M_t$ can be written as $D_t dt$ where
\begin{align*}
&D_t =\frac58 \left[ \frac{h_t''(W_t)^2}{2h_t'(W_t)^2} +\frac{h_t''(W_t)}{h_t'(W_t)} J_t  \right]-\frac{5h_t''(W_t)^2}{6 h_t'(W_t)^2} \\
&+n \left[-\frac{2h_t'(W_t)^2 }{(h_t(O_t)-h_t(W_t))^2} +\frac{2}{(O_t-W_t)^2} \right]+ m \left[-\frac{2h_t'(W_t)^2 }{(h_t(V_t)-h_t(W_t))^2} +\frac{2}{(V_t-W_t)^2} \right]\\
&+(l-m-n+\lambda) \frac{-2 h_t'(W_t)^2 }{\left( h_t(O_t)-h_t(W_t) \right) \left( h_t(V_t)-h_t(W_t) \right)}-(l-m-n+\lambda) \frac{-2}{(O_t-W_t)(V_t-W_t)}\\
&+ \frac{G'(\tilde X_t)}{G(\tilde X_t)}\cdot\frac{1}{h_t(O_t)-h_t(V_t)} \left[  \left(-\frac53 h_t''(W_t)+h_t'(W_t) J_t \right) - \frac{2h_t'(W_t)^2}{h_t(O_t)-h_t(W_t)} \right]\\
&-\frac{G'(\tilde X_t)}{G(\tilde X_t)}\cdot\frac{h_t(W_t)-h_t(O_t)}{\left(h_t(O_t)-h_t(V_t)\right)^2} 2 h_t'(W_t)^2 \frac{h_t(V_t)-h_t(O_t)}{\left( h_t(O_t)-h_t(W_t) \right) \left( h_t(V_t)-h_t(W_t) \right)}\\
&+\frac{G''(\tilde X_t)}{2G(\tilde X_t)}\cdot \frac83 \frac{h_t'(W_t)^2}{\left( h_t(O_t)- h_t(V_t) \right)^2}
-\frac{4G'(\tilde X_t)^2}{3G(\tilde X_t)^2} \left( \frac{h_t'(W_t)}{h_t(O_t)-h_t(V_t)} \right)^2\\
&-\frac{G'(X_t)}{G(X_t)}\cdot \left( \frac{J_t}{O_t-V_t}-\frac{2}{(O_t-W_t)(O_t-V_t)} \right)
+\frac{G'(X_t)}{G(X_t)}\cdot \frac{W_t-O_t}{(O_t-V_t)^2} \cdot\frac{2(V_t-O_t)}{(O_t-W_t)(V_t-W_t)}\\
&-\frac{G''(X_t)}{2G(X_t)} \cdot \frac83 \frac{1}{(O_t-V_t)^2} +\frac{4G'(X_t)^2}{3G(X_t)^2}\frac{1}{\left( O_t-V_t \right)^2}
 +\frac{1}{2} L_t^2.
\end{align*}
Recall that by (\ref{def:J}) we have
\begin{align}\label{app:J}
J_t=\frac{8}{3}\frac{G'(X_t)}{G( X_t)}\frac{1}{O_t-V_t}.
\end{align}
In $D_t$, we now interpret the terms that only involve $W_t, O_t, V_t$ as 'constant' and we regroup the coefficients in front of the terms which involve $h_t$. The coefficients for each of these terms are as follows:
The 'constant' term is 
\begin{align*}
&\frac{2n}{(O_t-W_t)^2} +\frac{2m}{(V_t-W_t)^2} +2\frac{l-m-n+\lambda}{(O_t-W_t)(V_t-W_t)}\\
&-\frac{G'(X_t)}{G(X_t)}\cdot \left( \frac{J_t}{O_t-V_t}-\frac{2}{(O_t-W_t)(O_t-V_t)} \right)\\
&+\frac{G'(X_t)}{G(X_t)}\cdot \frac{W_t-O_t}{(O_t-V_t)^2} \cdot\frac{2(V_t-O_t)}{(O_t-W_t)(V_t-W_t)}
-\frac{G''(X_t)}{2G(X_t)} \cdot \frac83 \frac{1}{(O_t-V_t)^2} \\
&+\frac{4}{3}\frac{G'(X_t)^2}{G( X_t)^2}\frac{1}{(O_t-V_t)^2}
+\frac{4G'(X_t)^2}{3G(X_t)^2}\frac{1}{\left( O_t-V_t \right)^2}
\end{align*}
and the remaining ones can be listed as follows: 
\begin{align}
\label{a.1}
\frac{h_t''(W_t)^2}{h_t'(W_t)^2}: \quad & \frac{5}{16}-\frac56 +\frac{25}{48}=0\\
\label{a.2}
\frac{h_t''(W_t)}{h_t'(W_t)}: \quad & \frac58 J_t-\frac{5}{3} \frac{G'(X_t)}{G(X_t)}\cdot \frac{1}{O_t-V_t}=0\\
\label{a.3}
\frac{h_t'(W_t)^2 }{(h_t(O_t)-h_t(W_t))^2} : \quad & -2n \\
\label{a.4}
\frac{h_t'(W_t)^2 }{(h_t(V_t)-h_t(W_t))^2} : \quad & -2m\\
\label{a.5}
\frac{h_t'(W_t)^2 }{(h_t(O_t)-h_t(V_t))^2} : \quad & \frac43 \frac{G''(\tilde X_t)}{G(\tilde X_t)} +\frac43 \frac{G'(\tilde X_t)^2}{G(\tilde X_t)^2}-\frac{4G'(\tilde X_t)^2}{3G(\tilde X_t)^2} =  \frac43 \frac{G''(\tilde X_t)}{G(\tilde X_t)} \\
\label{a.6}
\frac{h_t'(W_t)^2}{\left( h_t(O_t)-h_t(W_t) \right) \left( h_t(V_t)-h_t(W_t) \right)}:\quad & -2(l-m-n+\lambda)\\
\label{a.7}
\frac{h_t'(W_t)^2}{\left(h_t(O_t)-h_t(W_t)\right)\left( h_t(O_t)-h_t(V_t) \right)}: \quad & -2\frac{G'(\tilde X_t)}{G(\tilde X_t)} \\
\label{a.8}
\frac{h_t'(W_t)^2}{\left(h_t(V_t)-h_t(W_t)\right)\left( h_t(O_t)-h_t(V_t) \right)}: \quad &  -2\frac{G'(\tilde X_t)}{G(\tilde X_t)} \\
\label{a.9}
\frac{h_t''(W_t)}{h_t(O_t)-h_t(V_t)}: \quad &  -\frac{5G'(\tilde X_t)}{3G(\tilde X_t)}+ \frac{5G'(\tilde X_t)}{3G(\tilde X_t)}=0\\ 
\label{a.10}
\frac{h_t'(W_t)}{h_t(O_t)-h_t(V_t)}: \quad &\frac{G'(\tilde X_t)}{G(\tilde X_t)} \left[ J_t- \frac{8}{3}\frac{G'(X_t)}{G( X_t)}\frac{1}{O_t-V_t}\right]=0
\end{align}

Adding up the terms in (\ref{a.3}-\ref{a.8}), we get $$E_t \frac{h_t'(W_t)^2}{\left(h_t(V_t)-h_t(W_t)\right)\left( h_t(O_t)-h_t(W_t) \right)}$$ where
\begin{align*}
E_t=-2n\frac{\tilde X_t+1}{\tilde X_t}-2m\frac{\tilde X_t}{\tilde X_t +1} +\frac43 \frac{G''(\tilde X_t)}{G(\tilde X_t)} \tilde X_t (\tilde X_t +1)-2(l-m-n+\lambda) +2  \frac{G'(\tilde X_t)}{G(\tilde X_t)} (2\tilde X_t+1).
\end{align*}
As $G$ satisfies the differential equation (\ref{Gdiff}), we can conclude that $E_t=0$.
Now, it remains only to deal with the 'constant' term. It is equal to $C_t / ( ( V_t-W_t)(O_t-W_t)) $, with 
\begin{align*}
C_t=2n \frac{X_t+1}{X_t} +2m \frac{X_t}{X_t+1} -\frac{4G''(X_t)}{3 G(X_t)} X_t (X_t+1)+2(l-m-n+\lambda) -2 \frac{G'(X_t)}{G(X_t)}(2X_t+1),
\end{align*}
which is also equal to $0$ because of (\ref{Gdiff}).

This finally proves that $D_t=0$ hence $M_t$ is a local martingale.

\section*{Acknowledgements}
The author is very grateful to Wendelin Werner for suggesting this question, for numerous discussions and suggestions, and for his help throughout the preparation of this paper. The author acknowledges support of the SNF grant SNF-155922. The author is also part of the
NCCR Swissmap.

\end{document}